\documentclass[dvipsnames,11pt]{amsart}
\usepackage[margin=1.25in]{geometry}
\usepackage[utf8]{inputenc}
\usepackage[dvipsnames]{xcolor}
\usepackage{lmodern}
\usepackage{tikz-cd}
\usepackage{stmaryrd}
\SetSymbolFont{stmry}{bold}{U}{stmry}{m}{n}
\usepackage{microtype}
\usepackage[pagebackref]{hyperref}
\hypersetup{
  colorlinks=true,
    linkcolor=blue,
    filecolor=blue,      
    urlcolor=blue,
    citecolor=ForestGreen,
    linktoc=page}
\usepackage{amsmath,amsfonts,amssymb,amsthm,mathrsfs}
\usepackage[abbrev]{amsrefs} 
\usepackage[scr=rsfs]{mathalfa}
\usepackage{bm}

\usepackage{caption} 
\theoremstyle{plain}
\newtheorem{theorem}[]{Theorem}
\newtheorem{proposition}[theorem]{Proposition}
\newtheorem{lemma}[theorem]{Lemma}

\newtheorem{ex}[theorem]{Example}
\newtheorem{corollary}[theorem]{Corollary}
\newtheorem{conjecture}{Conjecture}

\theoremstyle{definition}
\newtheorem{definition}[theorem]{Definition}

\newtheorem{remark}[theorem]{Remark}

\newcommand{\Z}{\mathbb{Z}}
\newcommand{\R}{\mathbb{R}}
\newcommand{\C}{\mathbb{C}}
\newcommand{\bS}{\mathbb{S}}

\newcommand{\CF}{\mathrm{CF}}

\newcommand{\bL}{\mathbf{L}}

\newcommand{\bT}{\mathbf{T}}

\newcommand{\val}{\textrm{val}\,}

\newcommand{\Fuk}{\mathrm{Fuk}}

\usepackage{enumitem}
\numberwithin{theorem}{section}
\numberwithin{equation}{section}
\setlist[itemize]{leftmargin=2em}

\title[SYZ mirror symmetry for del Pezzo surfaces and affine structures]
{SYZ mirror symmetry for del Pezzo surfaces 
and affine structures}

\author{Siu-Cheong~Lau}
\email{lau@math.bu.edu}
\address{Mathematics and Statistics Department, Boston University, 111 Cummington Mall, Boston MA 02215}

\author[T.-J. Lee]{Tsung-Ju Lee}
\email{tsungju@gs.ncku.edu.tw}
\address{Department of Mathematics, National Cheng Kung University, 
No.~1 Daxue Rd., East District, Tainan 70101, Taiwan}

\author{Yu-Shen~Lin}
\email{yslin@bu.edu}
\address{Mathematics and Statistics Department, Boston University, 111 Cummington Mall, Boston MA 02215}


\begin{document}
\begin{abstract}
  We prove that the Landau--Ginzburg superpotential of 
  del Pezzo surfaces can be realized as a limit of their hyperK\"ahler rotation 
  toward the large complex structure limit point. 
  As a corollary, we compute the limit of the complex affine structure of 
  the special Lagrangian fibrations constructed by 
  Collins--Jacob--Lin in $\mathbf{P}^1\times \mathbf{P}^1$ 
  \cite{2021-Collins-Jacob-Lin-special-lagrangian-submaniflds-of-log-calabi-yau-manifolds} 
  and compare it with the integral affine structures used in the work of 
  Carl--Pumperla--Siebert \cite{2022-Carl-Pumperla-Siebert-a-tropical-view-of-landau-ginzburg-models}. We also construct the Floer-theoretical Landau--Ginzburg mirrors of smoothing of $A_n$-singularities and monotone del Pezzo surfaces, by using the gluing method of Cho--Hong--Lau \cite{CHL-glue} and Hong--Kim--Lau \cite{HKL}.  They agree with the result of limit of hyperK\"ahler rotations.
\end{abstract}
\maketitle
\tableofcontents

\section{Introduction} 
\label{sec: intro}
The Strominger--Yau--Zaslow conjecture \cite{1996-Strominger-Yau-Zaslow-mirror-symmetry-is-t-duality} (SYZ conjecture) predicts that 
  \begin{enumerate}
  	\item a Calabi--Yau manifold $X$ near a large complex structure limit, where mirror symmetry is expected to happen, admits a special Lagrangian torus fibration;
  	\item the mirror $\check{X}$ of $X$ can be constructed from the dual torus fibration of $X$, with the complex structure receiving ``quantum corrections'' from the holomorphic discs with boundaries on the special Lagrangian torus fibres in $X$. 
  \end{enumerate}
For the purpose of mirror symmetry, the goal is to understand how the quantum corrections affect the mirror construction.
  
 However, the original conjecture is hard due to the difficulties from analysis: there is very little known about the behavior of the Calabi--Yau metrics. To construct the special Lagrangian fibrations, one would need to first understand the Calabi--Yau metric. To understand the Calabi--Yau metric, the conjecture leads us to the study of holomorphic discs with boundaries on the special Lagrangian fibres. Thus, the problems of Calabi--Yau metrics, existence of special Lagrangian fibrations and quantum corrections from holomorphic discs are closely linked together. 
 
 To avoid the difficulties in analysis in the metric level and due to the interests from mirror symmetry, there are replacements for the original proposal from algebraic geometry and symplectic geometry. On the algebraic side, Kontsevich--Soibelman \cite{2006-Kontsevich-affine-structures-and-non-archimedean-analytic-spaces} and Gross--Siebert programs \cite{2011-Gross-Siebert-from-real-affine-geometry-to-complex-geometry} introduced the notion of scattering diagrams for understanding the quantum correction. 
On the symplectic side, Fukaya proposed the family Floer homology \cite{F} and later studied by Tu \cite{T2}, Abouzaid \cites{A1,A2,A3}, and Yuan \cite{Y2}. There is another approach using relative symplectic cohomology \cite{GV}. 
 Both approaches are designed to capture the quantum corrections in the mirror construction and both have big success in explaining mirror symmetry in an intrinsic way. However, it is still unclear how the two approaches are related to the original conjecture explicitly, again due to the lacking of the existence of genuine SYZ fibrations. Until very recently, T.~Collins, A.~Jacob, and the third author constructed the first non-trivial example\footnote{In the sense that a priori without knowing the existence of elliptic fibration structure after hyperK\"ahler rotation, which is a technical part in analysis.} of genuine special Lagrangian fibrations \cite{2021-Collins-Jacob-Lin-special-lagrangian-submaniflds-of-log-calabi-yau-manifolds}. In this paper, we will study the complex affine structures on the base of the genuine special Lagrangian fibrations in these examples and explain that there can be subtle differences between the SYZ bases and the affine manifold with singularities used in the Gross--Siebert program (in this case is worked out by Carl--Pumperla--Siebert), which is a long-termed folklore conjecture. We refer the readers to \cite{Gross}*{Section 7} and \cite{CL}*{Section 1} and references therein.

Mirror symmetry has been extended to Fano manifolds \cite{Kontsevich}. In which case, the mirror of a Fano manifold $Y$ is a Landau--Ginzburg superpotential $W$. The superpotential $W$ is expected to recover the enumerative geometry of $Y$ in various aspects. For instance, the quantum cohomology $\mathrm{QH}^{\ast}(Y)$ is isomorphic to the Jacobian ring $\mathrm{Jac}(W)$ 
\cites{1993-Batyrev-quantum-cohomology-rings-of-toric-manifolds,2009-Fukaya-Oh-Ohta-Ono-lagrangian-intersection-floer-theory-anomaly-and-obstruction-part-i} and the quantum periods of $W$ give the generating function of the descendant Gromov--Witten invariants \cites{2013-Coates-Corti-Galkin-Golyshev-Kasprzyk-mirror-symmetry-and-fano-manifolds,2016-Coates-Corti-Galkin-Kasprzyk-quantum-periods-for-3-dimensional-fano-manifolds,1998-Givental-a-mirror-theorem-for-toric-complete-intersections,2018-Hong-Lin-Zhao-bulk-deformed-potentials-for-toric-fano-surfaces-wall-crossing-and-periods}. Recently it is proved that superpotential agrees with the open mirror map of the corresponding local Calabi--Yau threefold mirror symmetry in the case of toric del Pezzo surfaces \cite{GRZ}. Auroux--Katzarkov--Orlov \cite{2006-Auroux-Katzarkov-Orlov-mirror-symmetry-for-del-pezzo-surfaces-vanishing-cycles-and-coherent-sheaves} proved homological mirror symmetry, i.e., the Fukaya--Seidel category $FS(W)$ is $A_{\infty}$-equivalent to the derived category of coherent sheaves $D^b\mathrm{Coh}(Y)$. In particular, they showed that the superpotential of a del Pezzo surface of degree $d$ can be topologically compactified as a rational elliptic surface with an $\mathrm{I}_d$ fibre. Therefore, an important task is to derive the superpotential for a given Fano manifold.
In the case of toric del Pezzo surfaces, the superpotential can be read off from the toric data \cites{1995-Givental-homological-geometry-i-projective-hypersurfaces,2000-Hori-Vafa-mirror-symmetry}, which is the generating function of a weighted count of Maslov index two pseudo-holomorphic discs with boundary on SYZ fibres in $Y$ \cite{CO}. The latter is called the Lagrangian Floer-theoretic potential of the moment fibres. A folklore conjecture is that the Lagrangian Floer-theoretic potentials of monotone Lagrangians in $Y$ can be glued via the wall-crossing formula and one obtains the Landau--Ginzburg superpotential of $Y$.

Let $Y$ be a del Pezzo surface of degree $d$ and $D\in |-K_Y|$ be a smooth anti-canonical divsior. There exists a meromorphic $2$-form on $Y$ with a simple pole on $D$ unique up to $\mathbb{C}^*$-scaling and restrict to a nowhere vanishing holomorphic $2$-form on $X=Y\setminus D$. Tian--Yau proved that 
\(X\) is equipped with an exact complete Calabi-Yau metric $\omega_{TY}$ such that $2\omega_{TY}^2=\Omega\wedge \bar{\Omega}$
\cite{1990-Tian-Yau-complete-kahler-manifolds-with-zero-ricci-curvature-i}. Throughout the paper, we will denote denote $\omega=\omega_{\mathrm{TY}}$ the exact Tian--Yau metric for simplicity. 
It is crucial to note that we will use the exact Tian--Yau metric. 
Responding to the SYZ conjecture \cite{1996-Strominger-Yau-Zaslow-mirror-symmetry-is-t-duality}
and conjectures of Auroux \cite{2007-Auroux-morror-symmetry-and-t-duality-in-the-complement-of-an-anticanonical-divisor}, Collins--Jacob--Lin proved that
\begin{theorem}[cf.~\cite{2021-Collins-Jacob-Lin-special-lagrangian-submaniflds-of-log-calabi-yau-manifolds}] \label{CJL}
There exists a special Lagrangian fibration $X\rightarrow B\cong \mathbb{R}^{2}$ 
with respect to $(\omega,\Omega)$. 
\end{theorem}	
Actually given any primitive homology class $\alpha\in \mathrm{H}_1(D,\mathbb{Z})$, the homology class $\tilde{\alpha}\in 
\mathrm{H}_2(X,\mathbb{Z})$ represented as a trivial $S^1$-bundle over a geodesic in $D$ in the class $\alpha$ can be realized as the special Lagrangian torus fibre in Theorem \ref{CJL}. Due to the dimensional coincidence, one has $\mathrm{Sp}(1)=\mathrm{SU}(2)$ and thus $X$ is hyperK\"ahler. Let $\check{X}$ be the hyperk\"{a}hler rotation with the Calabi-Yau metric and holomorphic volume form given by 
\begin{align}
\label{eq:hyperkah-rotation}
\begin{split}
\check{\Omega} &= \omega - \mathrm{i}\cdot\operatorname{Im}\Omega,\\
\check{\omega} & = \operatorname{Re}\Omega.
\end{split}
\end{align} Then the special Lagrangian fibration $X\rightarrow B$ becomes an elliptic fibration $\check{X}\rightarrow \mathbb{C}$ after hyperK\"ahler rotation. 
Collins--Jacob--Lin further proved that 
\begin{theorem}[cf.~\cite{2021-Collins-Jacob-Lin-special-lagrangian-submaniflds-of-log-calabi-yau-manifolds}]
\label{theorem:CJL}
The fibration \(\check{X}\to\mathbb{C}\) admits a compactification to a \emph{rational elliptic surface}
\(\check{Y}\to\mathbf{P}^{1}\) by adding 
 an \(\mathrm{I}_{d}\) fibre over \(\infty\), where 
\(d = (-K_{Y})^{2}\). In other words,  we have the following commutative diagram.
\begin{equation*}
\begin{tikzcd}[column sep=2em]
&X\ar[d]\ar[bend left,rrr,"\mathrm{HK~rotation}"]& & & 
\check{X}\ar[d]\ar[r,hook] & (\check{Y},\mathrm{I}_{d})\ar[d]\\
&\mathbb{R}^2 & & & \mathbb{C}\ar[r] & (\mathbf{P}^{1},\infty)
\end{tikzcd}
\end{equation*}
\end{theorem}
Note that \(X=\check{X}\) as topological spaces. In particular, 
\(\mathrm{H}_{k}(X;\mathbb{Z})=\mathrm{H}_{k}(\check{X};\mathbb{Z})\). We also remark that in the case of $Y=\mathbb{P}^2$, then $\check{X}$ is exactly the fibrewise compactification of the superpotential of $\mathbb{P}^2$ and a special case of Theorem \ref{main thm} below. 

The third author in \cite{2020-Collins-Jacob-Lin-the-syz-mirror-symmetry-conjecture-for-del-pezzo-surfaces-and-rational-elliptic-surfaces} introduced the notion of large complex structure limits for pairs from the view point of the original SYZ conjecture. 
We say that $(Y_t,D_t)$, \(t\in (0,1]\), a $1$-parameter family of pairs of del Pezzo surfaces $Y_t$ of degree $d$ with a smooth anti-canonical divisor $D_t$ is converging to a large complex structure limit if the limit $Y_t\rightarrow Y_0$ is a del Pezzo surface and $D_t\rightarrow D_0$ with $D_0$ being an irreducible nodal curve. We will denote $\alpha_t\in 
\mathrm{H}_1(D_t,\mathbb{Z})$ the vanishing cycle. Let $\pi_t: X_t=Y_t\backslash D_t\rightarrow B_t$ be the special Lagrangian fibration in Theorem \ref{CJL} with fibre class $\tilde{\alpha}_t$. 

Here we explain why this is a reasonable definition of large complex structure limit. Kontsevich--Soibelman \cite{2001-Kontsevich-Soibelman-homological-mirror-symmetry-and-torus-fibrations} modified the original SYZ conjecture: the Calabi--Yau manifolds approaching to a large complex structure limit will Gromov--Hausdorff converge to an affine manifold with singularity. The conjecture is proved for the case of K3 surfaces \cites{2000-Gross-Wilson-large-complex-structure-limits-of-k3-surfaces,OO} and compact hyperK\"ahler manifolds with Lagrangian fibrations \cite{2013-Gross-Tosatti-Zhang-collapsing-of-abelian-fibered-calabi-yau-manifolds} and for Fermat type Calabi--Yau hypersurfaces \cite{Y?}\footnote{However, there are no results in the literature of comparison between the collapsing limits with the affine structures being used in the Gross--Siebert program to the authors knowledge. The result in \cite{Y?} seems pretty close to such a statement but one requires stronger regularity on the solutions of real Monge--Amp\`{e}re equations.}.  
Therefore, the collapsing of the SYZ fibration is usually viewed as the characterization of the large complex structure limit from the SYZ point of view.  Now we compare the above SYZ collapsing picture with the semi-stable degenerations of K3 surfaces \cite{Kulikov} studied by Kulikov. The semi-stable degenerations of K3 surfaces are classified into type I, II, III: the type I degenerations correspond to non-collapsing degenerations with the non-compact version studied in \cite{LSZ}. In the type II and type III degenerations, the metric K3 surfaces have collapsing limits and their Gromov--Hausdorff limits are harder to analyze. On one hand, the type II degeneration would lead to an interval as the  Gromov--Hausdorff limits with respect to suitable scaling of the metrics \cites{HSVZ,2019-Honda-Sun-Zhang-a-note-on-the-collapsing-geometry-of-hyperkahler-four-manifolds}. On the other hand, type III degenerations have their collapsing limits as an integral affine spheres with singularities and is the conjectural description of large complex structure limit \cites{2001-Kontsevich-Soibelman-homological-mirror-symmetry-and-torus-fibrations,OO}. The key distinguishing feature of the latter two cases is that the collapsing limits have different dimensions. While the Tian--Yau spaces, are known as the bubbling limits of K3 surfaces, intuitively can be viewed as non-compact analogues of K3 surfaces. From the local model of the Tian--Yau spaces, one can see the geometry has a $2$-dimensional collapsing limit near infinity when $D_t$ degenerates \cite{2020-Collins-Jacob-Lin-the-syz-mirror-symmetry-conjecture-for-del-pezzo-surfaces-and-rational-elliptic-surfaces}. Moreover, the third author and R. Takahashi recently proved that there is a global collapsing, i.e.~with suitable scaling of the metrics, the Tian--Yau metrics Gromov--Hausdorff converge to a $2$-dimensional affine manifold with singularities \cite{LT}. Therefore, the authors believe that the above definition of large complex structure limits of pairs is a reasonable one from the metric perspective.
 
With the above understanding and the motivation from the Gross--Siebert program, the authors earlier conjectured that   
 \begin{conjecture} \label{naive conj}
 	The limiting complex affine manifold, i.e. the limit of complex affine structure on $B_t$ coincides with the affine manifold with singularity in the work of Carl--Pumperla--Siebert 
 	\cite{2022-Carl-Pumperla-Siebert-a-tropical-view-of-landau-ginzburg-models}. 
 \end{conjecture} In the previous work, the authors proved
 \begin{theorem} \cite{2022-Lau-Lee-Lin-on-the-complex-affine-structures-of-syz-fibration-of-del-pezzo-surfaces}
 	Conjecture \ref{naive conj} holds in the case when $d=9$.
 \end{theorem}

In this paper, we will extend the SYZ mirror symmetry of \cite{2020-Collins-Jacob-Lin-the-syz-mirror-symmetry-conjecture-for-del-pezzo-surfaces-and-rational-elliptic-surfaces} to its compactification: providing the recipe for superpotentials of del Pezzo surfaces from the SYZ fibration constructed in \cite{2021-Collins-Jacob-Lin-special-lagrangian-submaniflds-of-log-calabi-yau-manifolds}. 
\begin{theorem}[=Theorem \ref{thm:period-point}] 
\label{main thm}
		Let $(Y_t,D_t)$, $t\in [0,1]$, be a $1$-parameter family of pairs of del Pezzo surfaces $Y_t$ and smooth anti-canonical divisors $D_t$. Assume that $Y_0$ is a smooth del Pezzo surface and $D_0$ is an irreducible nodal anti-canonical divisor of $Y_0$. Set $X_t=Y_t\setminus D_t$. Then the rational elliptic surfaces $\check{Y}_t$, $t\neq 0$ from Theorem \ref{theorem:CJL} converge to the distinguished rational elliptic surface $\check{Y}_e$, which is the minimal smooth compactification of the Landau--Ginzburg mirror superpotential of the del Pezzo surfaces, viewed as monotone symplectic manifolds.
	\end{theorem}
In particular, the limiting affine structure of the special Lagrangian fibration in the del Pezzo surfaces can be computed from $\check{Y}_e$. 
We further compute explicitly the limiting complex affine manifolds for the case del Pezzo surfaces of degree eight in Section \ref{sec: dP8'} and for degree three and four in Section \ref{sec: dP3dP4}. 
\begin{remark}
	The slogan is ``the mirror is given 
	by the limit of the hyperK\"ahler rotations towards the large complex structure limit.'' 
	This seems to be a new phenomenon and doesn't appear in the K3 cases with respect to 
	the mirror map in Gross--Wilson \cite{1997-Gross-Wilson-mirror-symmetry-via-3-tori-for-a-class-of-calabi-yau-threefolds}. 
\end{remark}
\begin{remark}
 Ruddat and Siebert proved that Friedman's period points give monomial functions in the canonical coordinates of the base of the large complex structure limit degeneration \cite{RS}, which may give another interpretation of Theorem \ref{thm:period-point}. On the other hand, we point out that the limit of complex affine structures of the SYZ fibrations towards the large complex structure limit in the sense of \cite{2020-Collins-Jacob-Lin-the-syz-mirror-symmetry-conjecture-for-del-pezzo-surfaces-and-rational-elliptic-surfaces} are different from the integral affine structure used in \cite{2022-Carl-Pumperla-Siebert-a-tropical-view-of-landau-ginzburg-models} except for the $\mathbf{P}^2$ case, which is a coincidence.
    \end{remark}
The above theorem provides an understanding of SYZ mirror symmetry from the perspective of the metrics and hyper-K\"ahler rotation. An interesting feature of the result is the following: the superpotentials (together with the correct complex structures) of the mirror involves the enumeration of holomorphic discs and thus depends on the global geometry of the del Pezzo surfaces. On the other hand, the Tian--Yau metric and the special Lagrangian fibration only depend on the non-compact part. This is a weak evidence that the Tian--Yau metric may detect some global geometry for some specific compactification. This kind of statement may help to understand the behavior of the Calabi-Yau metrics and the authors would pursue it in the future.

Now we explain why the limit $\check{Y}_e$ is indeed the correct mirror. 
Recall that from the Torelli theorem of the Looijenga pairs \cite{2015-Gross-Hacking-Keel-moduli-of-surfaces-with-an-anti-canonical-cycle}, there exists a distinguished complex structure $(\check{Y}_e,\check{D}_e)$ which has trivial periods in the deformation family of Looijenga pairs. Hacking--Keating \cite{2020-Hacking-Keating-homological-mirror-symmetry-for-log-calabi-yau-surfaces} explains how to construct the mirror of a Lefschetz fibration $w\colon M\rightarrow \mathbb{C}$, where $M$ is a Weinstein manifold. Furthermore, various versions of homological mirror symmetry are proved with A-side categories of $M$ and B-side categories of $(\check{Y}_e,\check{D}_e)$. In the case we consider in this paper, $\check{Y}_e$ is a rational elliptic surface with $\check{D}_e$ being an $I_d$-fibre, then $M$ can be realized as a complement of a smooth anti-canonical divisor in a del Pezzo surface. On the other hand, T.~Collins, A.~Jacob, and the third author constructed special Lagrangian fibrations on the complement of a smooth anti-canonical divisor in del Pezzo surfaces of degree $d$ and on the complement of an $\mathrm{I}_d$-fibre in rational elliptic surfaces. Moreover, the special Lagrangian fibrations on the two geometries are dual to each other with respect to a suitable mirror map when one chooses the exact Tian--Yau metric on the complement in the del Pezzo surface and the distinguished complex structure on the rational elliptic surface \cites{2021-Collins-Jacob-Lin-special-lagrangian-submaniflds-of-log-calabi-yau-manifolds, 2020-Collins-Jacob-Lin-the-syz-mirror-symmetry-conjecture-for-del-pezzo-surfaces-and-rational-elliptic-surfaces}. This is mirror symmetry incorporating both A and B models on both geometries. 

In the second part of this paper, we construct the mirrors for del Pezzo surfaces from the perspective in symplectic geometry, using the gluing scheme of \cites{CHL-glue,HKL} based on immersed Lagrangian Floer theory \cites{2009-Fukaya-Oh-Ohta-Ono-lagrangian-intersection-floer-theory-anomaly-and-obstruction-part-i,2009-Fukaya-Oh-Ohta-Ono-lagrangian-intersection-floer-theory-anomaly-and-obstruction-part-ii,Seidel-g2,2010-Akaho-Joyce-immersed-lagrangian-floer-theory}.  The immersed Lagrangians that we use are essentially the nodal singular fibers of the SYZ fibration (up to Lagrangian isotopy), which form the sources of quantum corrections and play a crucial role in (partial) compactification of the incomplete SYZ mirror constructed from merely smooth SYZ fibers.  

In \cite{CHL-glue}, Cho, Hong and the first author found a method of constructing mirrors by gluing deformation spaces of immersed SYZ Lagrangian fibers.  A key ingredient is to establish and solve equations for Fukaya isomorphisms between the objects.  The upshot is that there exists a canonical $A_\infty$ functor from the Fukaya category to the matrix factorization category of the resulting LG model.  The method was applied to construct the mirrors of punctured Riemann surfaces using Seidel's immersed Lagrangians.  In \cite{HKL}, Hong, Kim and the first author studied the deformation and obstruction for nodal immersed spheres, and  applied the method to construct the mirrors of two-plane Grassmannians.

In this paper, we consider del Pezzo surfaces of degree $\geq 3$ which comes from smoothings of toric Gorenstein Fano surfaces.  These are smoothings of $A_n$ singularities, and we can construct a collection of immersed Lagrangians using $A_n$ local models.  

The case of $n=0$ corresponds to $\C^2$, which has a special Lagrangian fibration with exactly one nodal singular fiber.  The fibration can be understood as  `pushing in the corner' of the toric moment image of $\C^2$.  Mirror construction using such a fibration for $\C^2$ were found by the celebrated work of Auroux \cite{2007-Auroux-morror-symmetry-and-t-duality-in-the-complement-of-an-anticanonical-divisor} in the study of SYZ mirrors for various interesting examples of toric surfaces. 

As in \cites{CHL-glue,HKL}, the mirror is constructed by gluing the unobstructed deformation spaces of these immersed Lagrangians, which will be denoted as $\bL_{i,j}$ for certain indices $i,j$.  Due to the nature of mirror construction from symplectic geometry, the result is a rigid analytic variety. What is important is that such a mirror is exactly the analytification of $\check{X}_e$.  This is the reverse direction of mirror symmetry statement comparing to the work of Hacking--Keating \cite{HK2}. More precisely, we prove the following theorem:
 

\begin{theorem}[Theorem \ref{thm:mir-del-Pezzo}]
	Let $X$ be a symplectic smoothing of a toric Gorenstein Fano surface with a monotone symplectic form, and let $D$ be a smooth anti-canonical divisor.  Let $\check{X}$ be the complex surface constructed as follows.  First, we take the (multiple) blowing up of the toric variety $X_{\check{\Delta}}$ at every special point $z_i=-1$ in the $i$-th toric divisor for $(n_i+1)$-times (where $n_i+1$ is the multiplicity of  the $i$-th toric fixed point of $X_\Delta$).  We define $\check{X}$ to be the complement of the strict transform of all the toric divisors of $X_{\check{\Delta}}$.
	Then there is an $A_\infty$ functor $\Fuk(X-D) \to D^b(\check{X}+\Lambda_+^2)$ that is injective on the morphism spaces of $\bL_{i,j}$.
\end{theorem}

\begin{remark}
	The multiple blowing up of $X_{\check{\Delta}}$ above is a rational elliptic surface as in Theorem \ref{theorem:CJL}.  Thus, we have shown that the Lagrangian Floer construction agrees with Theorem \ref{thm:period-point} for these cases.
\end{remark}

The existence of the $A_\infty$ functor and the injectivity on reference Lagrangians follow from the general theory in \cite{CHL-glue}.  The main additional ingredient here is the local mirror construction for smoothings of $A_n$ singularities (see Section \ref{sec:An}) and blowing up over the Novikov ring.

The notations for the Novikov ring, its maximal ideal and the Novikov field are:
\begin{align*}
	\Lambda_0 :=& \left\{\sum_{i=0}^\infty a_i \bT^{A_i} \mid a_i \in \C, A_i \geq 0 \textrm{ and increases to } +\infty \right\} ,\\
	\Lambda_+ :=& \left\{\sum_{i=0}^\infty a_i \bT^{A_i} \mid a_i \in \C, A_i > 0 \textrm{ and increases to } +\infty \right\} ,\\
	\Lambda :=& \left\{\sum_{i=0}^\infty a_i \bT^{A_i} \mid a_i \in \C, A_i \textrm{ increases to } +\infty \right\}.
\end{align*}
The group of elements invertible in $\Lambda_0$ will also be important:
$$ \Lambda_0^\times := \left\{\sum_{i=0}^\infty a_i \bT^{A_i} \in \Lambda_0\mid A_0 = 0 \textrm{ and } a_0\not=0 \right\}. $$
A Lagrangian torus is equipped with flat $(\Lambda_0^\times)^2$-connections; a Lagrangian nodal sphere is equipped with boundary deformations in $\{(u,v)\in \Lambda_0: \val(uv)>0\}$, where $u$ and $v$ are associated with the two immersed sectors at the nodal self-intersection.
In general, the gluing of these charts results in a variety defined over $\Lambda$.  On the other hand, when the immersed Lagrangians are constructed nicely so that they are in the same energy level, the gluing between them that solves the isomorphism equations does not involve the formal $\bT$ parameter, and the resulting space is an $\Lambda_0$-extension of a variety over $\C$.  This is why we have $\check{X}+\Lambda_+^2$ in the above theorem.
 
\subsection*{Acknowledgment} 
The authors would like to thank S.-T.~Yau and B.~Lian for their
constant encouragement and support. They also thank CMSA at Harvard for providing a warm and
comfortable working environment.
The first author express his deep gratitude to Cheol-Hyun Cho, Hansol Hong and Yoosik Kim on several related collaborative works and a lot of motivating ideas. The third author would like to thank P.~Hacking, Yuji Okada, and B.~Siebert
for the explanation on their related works. The third author also wants to thank T.~Collins and A.~Jacob for discussion on the earlier collaborations which motivated the project. The authors would all want to thank the anonymous reviewer for the precious comments and suggestions. 
The first author is supported by the Simons Collaboration Grant \#580648. 
The second author is supported by the Simons Collaboration on HMS Grant
and Taiwan NSTC \#112-2115-M-006-016-MY3. 
The third author is supported by the Simons Collaboration Grant  
\#635846 and NSF grant DMS \#2204109.

\section{Torelli theorem of log Calabi--Yau surfaces}
In this section, we recall basic notions and properties 
of Looijenga pairs which will be used throughout this paper.
\begin{definition}[cf.~\cites{2015-Gross-Hacking-Keel-moduli-of-surfaces-with-an-anti-canonical-cycle,2020-Hacking-Keating-homological-mirror-symmetry-for-log-calabi-yau-surfaces}]
A \emph{Looijenga pair} \((\check{Y},\check{D})\) 
is a pair consisting of a smooth 
projective surface \(\check{Y}\) and a \emph{singular} 
cycle \(\check{D}\in |-K_{\check{Y}}|\). By adjunction, 
the arithmetic genus \(p_{a}(\check{D})=1\), which implies that
\(\check{D}\) is either an irreducible nodal \(\mathbf{P}^{1}\)
or a cycle of smooth rational curves. 
A \emph{log Calabi--Yau surface with maximal boundary} 
is a Looijenga pair \((\check{Y},\check{D})\)
where \(\check{Y}\) is a smooth rational projective surface.
\end{definition}

For a log Calabi--Yau pair \((\check{Y},\check{D})\), it has been proven in 
\cite{2015-Gross-Hacking-Keel-moduli-of-surfaces-
with-an-anti-canonical-cycle}*{Lemma 1.3}
that after blowing up
some nodes in \(\check{D}\), the resulting pair \((\check{Z},\check{F})\),
where \(\check{F}\) is the reduced inverse 
image of \(\check{D}\), admits a toric
model \((\check{Y}_{\mathrm{tor}},\check{D}_{\mathrm{tor}})\);
namely, there is a birational morphism of pairs
\((\check{Z},\check{F})\to 
(\check{Y}_{\mathrm{tor}},\check{D}_{\mathrm{tor}})\)
such that 
\begin{itemize}
\item \(\check{Y}_{\mathrm{tor}}\) is a smooth projective toric variety and 
\(\check{D}_{\mathrm{tor}}\) is the toric boundary;
\item the pair \((\check{Z},\check{F})\) is obtained by successive blowups 
at interior points of the components of \(\check{D}_{\mathrm{tor}}\), 
and \(\check{F}\) is the strict transform of \(\check{D}_{\mathrm{tor}}\). 
\end{itemize}

\begin{definition}
The data 
\((\check{Y},\check{D})\leftarrow (\check{Z},\check{F})\rightarrow 
(\check{Y}_{\mathrm{tor}},\check{D}_{\mathrm{tor}})\) is called
\emph{a toric model} of \((\check{Y},\check{D})\).
\end{definition}

\subsection{Torelli theorem for log CY surfaces with maximal boundary}
\label{subsection:torelli}
We begin by recalling a standard cohomology computation.
Let \((\check{Y},\check{D})\) be a log Calabi--Yau 
surface with maximal boundary.
Put \(\check{X}:=\check{Y}\setminus\check{D}\).
Consider the long exact sequence of the relative 
homology for \((\check{Y},\check{X})\)
\begin{equation}
\label{seq:relative-homology}
0\to \mathrm{H}_{3}(\check{Y},\check{X};\mathbb{Z})\to 
\mathrm{H}_{2}(\check{X};\mathbb{Z})
\to\mathrm{H}_{2}(\check{Y};\mathbb{Z})\to 
\mathrm{H}_{2}(\check{Y},\check{X};\mathbb{Z})
\to \mathrm{H}_{1}(\check{X};\mathbb{Z})\to 0.
\end{equation}
Both sides in \eqref{seq:relative-homology} are zero since
\(\mathrm{H}_{1}(\check{Y};\mathbb{Z})=
\mathrm{H}_{3}(\check{Y};\mathbb{Z})=0\).
Via Lefschetz--Poincar\'{e} duality 
\(\mathrm{H}_{3}(\check{Y},\check{X};\mathbb{Z})
\simeq \mathrm{H}^{1}(\check{D};\mathbb{Z})\), we see that
the sequence \eqref{seq:relative-homology} is transformed into
\begin{equation} \label{rel les}
0\to \mathrm{H}^{1}(\check{D};\mathbb{Z})\to 
\mathrm{H}_{2}(\check{X};\mathbb{Z})
\to\mathrm{H}_{2}(\check{Y};\mathbb{Z})\to \mathrm{H}^{2}(\check{D};\mathbb{Z})
\to \mathrm{H}_{1}(\check{X};\mathbb{Z})\to 0.
\end{equation}
Under the isomorphism \(\mathrm{H}^{2}(\check{D};\mathbb{Z})
\cong\mathbb{Z}^{d}\), the map 
\(\mathrm{H}_{2}(\check{Y};\mathbb{Z})\to 
\mathrm{H}^{2}(\check{D};\mathbb{Z})\)
is identified with taking the signed intersection 
\begin{equation}
\mathrm{H}_{2}(\check{Y};\mathbb{Z})\to \mathbb{Z}^{d},~
[\gamma]\mapsto ([\gamma]\cdot [\check{D}_{i}])_{i=1}^{d}.
\end{equation}

For each \(1\le i\le d\), let \([\check{D}_{i}]^{\perp}:=
\left\{[\gamma]\in \mathrm{H}_{2}
(\check{Y};\mathbb{Z})~\left|~[\gamma]\cdot[\check{D}_{i}]=0
\right.\right\}\). Put \([\check{D}]^{\perp}:= 
\cap_{i=1}^{d} [\check{D}_{i}]^{\perp}\).
We then have a short exact sequence
\begin{equation}
\label{equation:fund-seq-1}
0\to \mathrm{H}^{1}(\check{D};\mathbb{Z})\to 
\mathrm{H}_{2}(\check{X};\mathbb{Z})
\to [\check{D}]^{\perp}\to 0.
\end{equation}

Since \((\check{Y},\check{D})\) is log Calabi--Yau, there 
exists a unique (up to a constant) holomorphic two form 
\(\check{\Omega}\) on \(\check{X}\)
with a simple pole along \(\check{D}\).
We then fix a basis \([\beta]\in\mathrm{H}_{1}(\check{D};\mathbb{Z})\) by demanding 
\begin{equation} \label{normalization GHK}
\int_{\beta} \mathrm{Res}~\check{\Omega} = 1.
\end{equation}
Following Friedman (see also the proof of Lemma \ref{lemma:deform-cycles}), 
every cycle \([\gamma]\in [\check{D}]^{\perp}\) can be deformed into 
a cycle \([\tilde{\gamma}]\in \mathrm{H}_{2}(\check{X};\mathbb{Z})\). 
Moreover, if \([\tilde{\gamma}']\) is an
element in \(\mathrm{H}_{2}(\check{X};\mathbb{Z})\)
whose image in \([\check{D}]^{\perp}\) is equal to \([\gamma]\), then
\begin{equation}
\int_{[\tilde{\gamma}']}\check{\Omega}-
\int_{[\tilde{\gamma}]}\check{\Omega}\in\mathbb{Z}.
\end{equation}
In particular, the map
\begin{equation}
\label{eq:period-map}
\phi\colon [\check{D}]^{\perp} \to \mathbb{C}^{\ast},~[\gamma]\mapsto 
\exp\left(2\pi\mathrm{i}\cdot\int_{\tilde{\gamma}}\check{\Omega}\right)
\end{equation}
is well-defined.
\begin{definition}
For a log Calabi--Yau surface with maximal boundary \((\check{Y},\check{D})\),
the map \(\phi_{(\check{Y},\check{D})}\in\mathrm{Hom}
([\check{D}]^{\perp},\mathbb{C}^{\ast})\)
obtained in \eqref{eq:period-map} is called \emph{the period point 
associated with \((\check{Y},\check{D})\)}.
\end{definition}
The period map has another incarnation. 
For each line bundle $L$ on $\check{Y}$ such that $L|_{\check{D}_i}\cong \mathcal{O}_{\check{D}_i}$, i.e.~$L\in [\check{D}]^{\perp}$, its restriction to $\check{D}$ has $L_{\check{D}}\in \mbox{Pic}^0(\check{D})\cong \mathbb{C}^*$. Therefore, one naturally associate an element in $\mbox{Hom}(\check{D}^{\perp},\mathbb{C}^*)$ for each pair $(\check{Y},\check{D})$. From the discussion in \cite{2015-Friedman-on-the-geometry-of-anticanonical-pairs}*{p.21}, the two notations of the periods actually coincide. Fix a deformation family of Looijenga pairs and a reference Looijenga pair $(\check{Y}_0,\check{D}_0)$. We identify $\mbox{Hom}(\check{D}^{\perp},\mathbb{C}^*)\cong \mbox{Hom}(\check{D}_0^{\perp},\mathbb{C}^*)$ via parallel transport for any Looijenga pair $(\check{Y},\check{D})$ within the same deformation family. The weak Torelli theorem states that the periods classify the complex structure of the pairs within its deformation family.
\begin{theorem}[Torelli theorem 
\cite{2015-Gross-Hacking-Keel-moduli-of-surfaces-with-an-anti-canonical-cycle}] \label{Torelli}
The period \(\phi_{(\check{Y},\check{D})}\in 
\mathrm{Hom}([\check{D}_0]^{\perp},\mathbb{C}^{\ast})\)
determines the pair \((\check{Y},\check{D})\) 
uniquely within its deformation type. 
\end{theorem}

From \cite{2015-Friedman-on-the-geometry-of-
anticanonical-pairs}*{Proposition 9.15 
and Proposition 9.16},
there are $10$ deformation families of rational 
elliptic surfaces with $I_d$ fibres, $d=1,\ldots, 9$. 
More precisely, there is exactly one family for 
each $d\neq 8$ and two for $d=8$. These correspond 
to $10$ families of del Pezzo surfaces. Recall that 
the del Pezzo surfaces are either blowups of $\mathbf{P}^2$ 
at $d$ points with $0\leq d\leq 8$ in general position 
or $\mathbf{P}^1\times \mathbf{P}^1$. 
We will explain which deformation families of rational 
elliptic surfaces with an $I_8$ fibre correspond 
to $\mathbb{F}_{0}:=\mathbf{P}^1\times \mathbf{P}^1$ 
versus the Hirzebruch surface $\mathbb{F}_1$ later. 

Within each deformation family, there is a 
distinguished rational elliptic surface 
whose period point is trivial by the Torelli theorem. 
The rational elliptic surface in each deformation with trivial periods can be characterized as follows.
Let \((\check{Y},\check{D})\) be a rational elliptic surface with trivial period point.
Let \((\check{Y},\check{D})\leftarrow (\check{Z},\check{F})\rightarrow 
(\check{Y}_{\mathrm{tor}},\check{D}_{\mathrm{tor}})\) be a toric model of
\((\check{Y},\check{D})\).
Notice that \((\check{Z},\check{F})\) and \((\check{Y},\check{D})\) have
the same period point under the identification 
\(\mathrm{Hom}([\check{D}]^{\perp},\mathbb{C}^{\ast})\cong
\mathrm{Hom}([\check{F}]^{\perp},\mathbb{C}^{\ast})\). 
Then \((\check{Y},\check{D})\) has trivial period point
if and only if \((\check{Z},\check{F})\) does. Let
\(\check{D}_{\mathrm{tor},1},\ldots,\check{D}_{\mathrm{tor},p}\) 
be the irreducible components of 
\(\check{D}_{\mathrm{tor}}\). 
\((\check{Z},\check{F})\)
has trivial period point if and only if \((\check{Z},\check{F})\)
is obtained by blowing up the points 
\(-1\in\mathbb{C}^{\ast} = \check{D}_{\mathrm{tor},i}\setminus 
\bigcup_{j\ne i}\check{D}_{\mathrm{tor},j}\) in the toric coordinates
on \((\check{Y}_{\mathrm{tor}},\check{D}_{\mathrm{tor}})\)
(cf.~\cite{2015-Gross-Hacking-Keel-moduli-of-surfaces-with-an-anti-canonical-cycle}).

This distinguished complex structure plays an essential role in the context of mirror symmetry.
Hacking--Keating proved the homological mirror symmetry: let $\check{X}\rightarrow \mathbb{C}$ be the Landau-Ginzburg mirror of the Looijenga pair $(Y,D)$. When $\check{X}$ is equipped with the exact symplectic form and $(Y,D)$ has the distinguished complex structure, then the wrapped Fukaya category of $\check{X}$ is isomorphic to the derived category of coherent sheaves on $X$ \cite{2020-Hacking-Keating-homological-mirror-symmetry-for-log-calabi-yau-surfaces}.
Collins--Jacob--Lin also proved SYZ mirror symmetry between del Pezzo surfaces and 
rational elliptic surfaces with such distinguished complex structures. 
Here in this paper, we will explain how such distinguished complex structures 
arise naturally from the large complex structure limit. 

Recall that rational elliptic surface $\check{Y}$ is extremal if its Modell--Weil group is finite. All the extremal rational elliptic surfaces are classified 
(see \cite{1989-Miranda-the-basic-theory-of-elliptic-surfaces}*{p.77}). 

\subsection{Periods of certain toric surfaces}
In this subsection, we will compute the period point of 
certain log Calabi--Yau surfaces with maximal boundary \((\check{Z},\check{F})\)
which arise naturally from mirror symmetry. Let 
us fix the following notation.
\begin{itemize}
\item Let \(N=\mathbb{Z}^{2}\) and \(M=\mathrm{Hom}(N,\mathbb{Z})\) be the dual lattice.
Let \(N_{\mathbb{R}}=N\otimes\mathbb{R}\) and \(M_{\mathbb{R}}=M\otimes\mathbb{R}\) be
the scalar extension of \(N\) and \(M\) respectively.
\item Let \(\Delta \subset M_{\mathbb{R}}\) be a reflexive polytope. Denote by \(\Sigma_{\Delta}\)
the normal fan of \(\Delta\)
and by \(\mathbf{P}_{\Delta}\) the toric variety associated with \(\Sigma_{\Delta}\).
Let \(\nabla=\Delta^{\vee}\subset N_{\mathbb{R}}\) be the dual polytope 
of \(\Delta\). Note that \(\Sigma_{\Delta}\)
is equivalent to the face fan of \(\nabla\).
\item Let \(\Delta_{0}\) (resp.~\(\nabla_{0}\)) 
be the set of vertices of \(\Delta\) (resp.~\(\nabla\)).
\end{itemize}
We will tacitly assume that \(\mathbf{P}_{\Delta}\) is \emph{smooth};
in other words, \(\nabla_{0} = \nabla\cap N\setminus\{\mathbf{0}\}\). 
Let \(t_{1},t_{2}\) be the coordinate on \((\mathbb{C}^{\ast})^{2}\).
Consider the superpotential 
\begin{equation*}
W\colon (\mathbb{C}^{\ast})^{2}\to \mathbb{C},~(t_{1},t_{2})\mapsto
\sum_{n\in \nabla_{0}} t^{n}.~
(\mbox{Here}~t^{n}:=t_{1}^{n_{1}}t_{2}^{n_{2}}.)
\end{equation*}
We may regard \(W\) as a \emph{section} of the anti-canonical bundle on \(\mathbf{P}_{\nabla}\),
that is, we compactify the torus \((\mathbb{C}^{\ast})^{2}\) into a projective toric surface
\(\mathbf{P}_{\nabla}\)
in a way such that \(W\) extends to a section of the anti-canonical bundle on \(\mathbf{P}_{\nabla}\).
Let \(\check{Y}_{\mathrm{tor}}\to\mathbf{P}_{\nabla}\) be the maximal projective crepant partial
desingularization. In the present case, 
the resolution is achieved by taking a sequence of weighted blow ups at codimension two strata
(which are torus invariant points)
and \(\check{Y}_{\mathrm{tor}}\) is a smooth semi-Fano toric surface.
\begin{lemma}
\label{lem:do-not-pass-toric-points}
\(\{W=0\}\) does not meet any
torus fixed point on \(\mathbf{P}_{\nabla}\).
\end{lemma}
\begin{proof}
To simply the notation, put \(\Sigma=\Sigma_{\nabla}\). Denote by \(\Sigma(1)\)
the set of \(1\)-cones in \(\Sigma\). For \(\rho\in\Sigma(1)\), 
by abuse of notation, we also denote 
its primitive generator (in the lattice \(M\)) by the same symbol \(\rho\).
Consider the 
homogeneous coordinate ring \(\mathbb{C}[w_{\rho}:\rho\in\Sigma(1)]\).
Then every section \(t^{n}\in\mathrm{H}^{0}(\mathbf{P}_{\nabla},-K_{\mathbf{P}_{\nabla}})\)
can be identified with a monomial via
\begin{equation}
\label{eq:sec-monomial}
t^{n}\mapsto \prod_{\rho\in\Sigma(1)} w_{\rho}^{\langle n,\rho\rangle+1}.
\end{equation}
Put \(p=|\Sigma(1)|\) and \(U_{\Sigma}=\mathbb{C}^{p}\setminus Z(\Sigma)\).
Since \(\mathbf{P}_{\nabla}\) 
is simplicial, \(U_{\Sigma}\slash G\cong\mathbf{P}_{\nabla}\) is a 
geometric quotient with \(G=\mathrm{Hom}(\mathrm{Cl}(\mathbf{P}_{\nabla}),\mathbb{C}^{\ast})\)
(cf.~\cite{2011-Cox-Little-Schenck-toric-varieties}*{Chapter 5}). 
The symbol \(w_{\rho}\) can be thought of as
a coordinate function on \(U_{\Sigma}\) and 
the monomial in the right hand side of \eqref{eq:sec-monomial}
is a \(G\)-equivariant function with respect to a certain character of \(G\).
The toric divisor \(D_{\rho}\) is the image of \(\{w_{\rho}=0\}\) under
the projection 
\begin{equation*}
\pi\colon U_{\Sigma}\to U_{\Sigma}\slash G.
\end{equation*}

Since \(\mathbf{P}_{\nabla}\) is simplicial, each torus invariant
point is an intersection of two toric divisors. \(D_{\rho}\cap D_{\mu}\ne\emptyset\) if and 
only if \(\rho\) and \(\mu\) span a \(2\)-cone in \(\Sigma\). Recall that 
\begin{equation*}
\nabla = \{n\in N_{\mathbb{R}}~|~\langle n,\rho\rangle\ge -1,~\forall\rho\in\Sigma(1)\}.
\end{equation*}
It follows that \(n\in\nabla_{0}\) if and only if there are exactly two \(1\)-cones,
say \(\rho\) and \(\mu\), such that \(\langle n,\rho\rangle=\langle n,\mu\rangle=-1\)
(since \(\mathbf{P}_{\nabla}\) is Fano).
Fix \(n\in \nabla_{0}\) and pick \(\rho\) and \(\mu\) as above. Now in the superpotential
\begin{equation*}
W = \sum_{n\in \nabla_{0}\cap N} \prod_{\rho\in\Sigma(1)} w_{\rho'}^{\langle n,\rho'\rangle+1},
\end{equation*}
there exists a unique monomial which contains neither \(w_{\rho}\) nor \(w_{\mu}\).
Consequently
\begin{equation*}
\left.W\right|_{w_{\rho}=w_{\mu}=0} = 
\prod_{\substack{\rho'\in\Sigma(1)\\\rho'\notin\{\rho,\mu\}}} 
w_{\rho'}^{\langle n,\rho'\rangle+1}.
\end{equation*}
If there is a point \(w\in\mathbb{C}^{p}\) such that \(w_{\kappa}=w_{\rho}=w_{\mu}=0\) 
with \(\kappa\ne \rho\) and \(\kappa\ne \mu\), 
then \(\{\kappa,\rho,\mu\}\) contains
a primitive collection of \(\Sigma\) which implies that \(w\in Z(\Sigma)\).
We thus conclude that \(W(w)\ne 0\) for any \(w\in\pi^{-1}(D_{\rho}\cap D_{\mu})\).
\end{proof}

Let \(\Sigma_{\nabla}(1)=\{\rho_{1},\ldots,\rho_{p}\}\). Write 
\(\rho_{i}=\begin{bmatrix} a_{1i} & a_{2i} \end{bmatrix}^{\intercal}\). 
The matrix
\begin{equation*}
A=\begin{bmatrix}
a_{11} & a_{12} & \cdots & a_{1p}\\
a_{21} & a_{22} & \cdots & a_{2p}
\end{bmatrix}
\end{equation*}
gives rise to a map \(A\colon \mathbb{Z}^{p}\to M,~e_{i}\mapsto \rho_{i}\)
whose cokernel is finite.
The dual \(A^{\intercal}\colon N\to \mathbb{Z}^{p}\) induces for each \(i\) an injection
\begin{equation*}
\mathbb{C}[\rho_{i}^{\perp}\cap N]\to\mathbb{C}[e_{i}^{\perp}\cap \mathbb{Z}^{p}],~
s_{i} \mapsto \prod_{j=1}^{p} w_{\rho_{j}}^{\langle n_{i},\rho_{j}\rangle}
\end{equation*}
where \(n_{i}\) is a primitive generator of 
\(\rho_{i}^{\perp}\cap N\cong\mathbb{Z}\) and
\(s_{i}\) is the corresponding coordinate function 
on \(\mathbb{C}^{\ast}=\mathrm{Spec}(\mathbb{C}[\rho_{i}^{\perp}\cap N])\).

Fix \(i\) from now on. We will write a point in \(\mathbb{C}^{p}\) as \((w_{\rho_{i}},\mathbf{w})\), 
where \(\mathbf{w}\) is a vector indexed by \(\Sigma(1)\setminus \{\rho_{i}\}\).
Let us compute the intersection \(\{W=0\}\cap D_{\rho_{i}}\). 
As in the proof of Lemma \ref{lem:do-not-pass-toric-points}, we write 
\begin{equation}
\label{eq:w-homogeneous}
W = \sum_{n\in \nabla_{0}\cap N} \prod_{j=1}^{p} w_{\rho_{j}}^{\langle n,\rho_{j}\rangle+1}
\end{equation}
and regard \(W\) as a function on \(U_{\Sigma}\).
Restricting \(W\) to \(w_{\rho_{i}}=0\), we see that 
only two terms in \eqref{eq:w-homogeneous} survive:
the monomials \(t^{n}\) with \(\langle n,\rho_{i}\rangle=-1\).
We denote them by \(t^{q_{1}}\) and \(t^{q_{2}}\).
Then \(\{W=0\}\cap D_{\rho_{i}}\) is equal to
\begin{equation*}
\left\{(0,\mathbf{w})\in U_{\Sigma}~\Big|~
\prod_{j=1}^{p} w_{\rho_{j}}^{\langle q_{1},\rho_{j}\rangle+1}+
\prod_{j=1}^{p} w_{\rho_{j}}^{\langle q_{2},\rho_{j}\rangle+1}=0\right\}\Big\slash G
\end{equation*}
By our assumption, the face fan of \(\nabla\) is smooth.
It follows that \(q_{1}-q_{2}\) is a primitive generator of \(\rho_{i}^{\perp}\cap N\)
and therefore
\begin{align*}
&~\prod_{j=1}^{p} w_{\rho_{j}}^{\langle q_{1},\rho_{j}\rangle+1}+
\prod_{j=1}^{p} w_{\rho_{j}}^{\langle q_{2},\rho_{j}\rangle+1}=0\\
\Leftrightarrow&~
\prod_{j=1}^{p} w_{\rho_{j}}^{\langle q_{1}-q_{2},\rho_{j}\rangle}=-1\\
\Leftrightarrow&~s_{i}=-1.
\end{align*}
We summarize this into the following proposition.
\begin{proposition}
\label{prop:-1-point}
Assume that \(\mathbf{P}_{\Delta}\) is smooth.
\(\{W=0\}\) and \(D_{\rho_{i}}\) can only meet at \(-1\in\mathbb{C}^{\ast}=
\mathrm{Spec}(\mathbb{C}[\rho_{i}^{\perp}\cap N])\subset D_{\rho_{i}}\).
\end{proposition}

Let us focus on the MPCP desingularization now.
Let \(\check{D}_{\mathrm{tor}}\) be the union of toric divisors on 
\(\check{Y}_{\mathrm{tor}}\).
Since \(\check{Y}\to \mathbf{P}_{\nabla}\) is crepant, 
\(W\) can be thought of as a section of \(-K_{\check{Y}}\) as well.
One immediately obtains the following corollary.
\begin{corollary}
\label{cor:trivial-period-points}
Assume that \(\mathbf{P}_{\Delta}\) is smooth as before and
regard \(W\) as a section of the anti-canonical bundle on \(\check{Y}_{\mathrm{tor}}\).
Let \((\check{Z},\check{F})\) be the blow up of \((\check{Y}_{\mathrm{tor}},\check{D}_{\mathrm{tor}})\)
at \(\{W=0\}\cap \check{D}_{\mathrm{tor}}\)
and let \(\check{F}\) be the reduced inverse image of \(\check{D}_{\mathrm{tor}}\).
Then \((\check{Z},\check{F})\) has trivial period point.
\end{corollary}
\begin{proof}
Since \(\{W=0\}\) contains no torus fix points in \(\mathbf{P}_{\nabla}\),
the maximal projective crepant partial desingularization \(\check{Y}\to\mathbf{P}_{\nabla}\)
induces an isomorphism in a neighborhood of \(\{W=0\}\),
because \(\{W=0\}\) only can intersect \(\check{D}_{\mathrm{tor}}\) at interior points of
each irreducible component of \(\check{D}_{\mathrm{tor}}\).
Combined with Proposition \ref{prop:-1-point}, we deduce that if they intersect, they must meet at 
\(-1\in\mathbb{C}^{\ast}\). Since
\(\check{Z}\) is obtained by a sequence of blow ups at those intersections, we conclude that
\((\check{Z},\check{F})\) must have trivial period point (cf.~\cite{2015-Gross-Hacking-Keel-moduli-of-surfaces-with-an-anti-canonical-cycle}*{Lemma 2.8
and Proposition 2.9}). 
\end{proof}

\section{From SYZ fibrations to LG superpotentials}
\label{subsection:period-calculation}
In this section, we will prove the slogan ``the limit of hyperK\"ahler rotations towards the large complex structure limit gives the mirror.'' Recall that del Pezzo surfaces are either $\mathbf{P}^1\times \mathbf{P}^1$ or blow up of $\mathbf{P}^2$ at $d$ generic points, where $0\leq d\leq 8$. Now given a holomorphic family of del Pezzo surfaces $p:\mathcal{Y}\rightarrow C$, there is a natural polarization $\mathcal{L}$ coming from the anti-canonical divisors. From the Riemann--Roch theorem, one gets that $h^0(Y,-K_Y)$ doesn't depend on the complex structure of the del Pezzo surface $Y$ and $\mathrm{H}^i(Y,-K_Y)$ always vanishes. One deduces that $R^0p_{\ast}\mathcal{L}$ is a vector bundle over $C$ with fibre $\mathrm{H}^0(Y_c,-K_{Y_c})$ over $c$ with $Y_c:=p^{-1}(c)$. Then the total space of $R^0p_{\ast}\mathcal{L}$ parametrizes the pairs of del pezzo surfaces with anti-canonical divisors. Outside a complex codimension one subset, a point in the total space of $R^0p_{\ast}\mathcal{L}$ corresponds to a pair with smooth anti-canonical divisors. Now choosing a path $[0,1]$ in the total space $R^0p_{\ast}\mathcal{L}$ induces a \(1\)-parameter family of pairs of del Pezzo surfaces $Y_t$ of degree \(d\) 
with anti-canonical divisors $D_t$. Through out this section, we will assume that $D_t$ are smooth for $t\neq 0$ and $D_0$ is irreducible and nodal\footnote{Notice that generic del Pezzo surfaces admits an irreducible nodal anti-canonical divisor.}.

Assume that \(D_t\) degenerates to an irreducible nodal curve \(D_{0}\).
Let \(\alpha_t\in \mathrm{H}_1(D_t,\mathbb{Z})\) be the vanishing cycle. 
Let \(\pi_t\colon X_t=Y_t\setminus D_t\rightarrow B_t\) be the special 
Lagrangian fibration produced in 
Theorem \ref{theorem:CJL} with fibre class $\tilde{\alpha_t}$.
Applying hyperk\"{a}hler rotation to \(X_{t}\) 
yields another family \(\check{X}_{t}\to \mathbb{C}\).
By Theorem \ref{theorem:CJL}, we obtain a \(1\)-parameter 
family of rational elliptic surfaces \(\check{Y}_{t}\), \(t\in (0,1]\). 
This section is devoted to proving that 
\begin{theorem} 
\label{thm:period-point}
Assume that $Y_0$ is a generic del Pezzo surface. Then
	the rational elliptic surface $\check{Y}_t$ converges to the distinguished rational elliptic surface $\check{Y}_e$ in its deformation family, as $t\rightarrow 0$. 
\end{theorem}

From the Torelli theorem in the previous section, it suffices to prove that:
\begin{equation*} 
\tag{\(\ast\)}
\label{theorem:period-limit}
\mbox{The period point \(\phi_{\check{Y}_{t}}\) converges to the trivial homomorphism 
\(e\) as \(t\to 0\)}. 
\end{equation*}
In other words, the pair \((\check{Y}_{t},\check{D}_{t})\) converges to 
\((\check{Y}_{e},\check{D}_{e})\),
the log Calabi--Yau surface with trivial period point.




Recall the following construction in 
\cite{2015-Friedman-on-the-geometry-of-anticanonical-pairs}*{\S3}
\begin{lemma}
\label{lemma:deform-cycles}
Let \(Y\) be a del Pezzo surface.
Let \(D\in |-K_{Y}|\) be an irreducible nodal curve or smooth curve. Then for each 
\(\gamma\in [D]^{\perp}\), there exists 
a $2$-cycle \(\tilde{\gamma}\) of $Y$ 
such that $[\tilde{\gamma}]=\gamma$ in \(Y\) 
and \(\tilde{\gamma}\cap D=\emptyset\).
\end{lemma}

\begin{proof}
One can write \(\gamma=\sum_{i} a_{i}[C_{i}]\) with \(a_{i}=\pm 1\) and
\(C_{i}\) being smooth curves intersecting transversally 
at distinct points on \(D^{\mathrm{reg}}\). 
We can rearrange the signed intersections as
\(\sum (p_{j}-q_{j})\). Let \(\sigma_{j}\) be a smooth curve in \(D^{\mathrm{reg}}\)
going from \(q_{j}\) to \(p_{j}\) and \(\tau(\sigma_{j})\) be 
a tube over \(\sigma_{j}\) in \(Y\).
We can glue \(\tau(\sigma_{j})\) with \(\gamma\setminus\{p_{j},q_{j}\}\) inductively
and obtain \(\tilde{\gamma}\).
\end{proof}

One can find a $1$-parameter family of diffeomorphisms $f_t\colon Y_0\cong Y_t$ with $t\in [0,1]$. Fix a neighborhood $N\subseteq Y_0$ of $D_0$. There exists an $\epsilon>0$ such that $D_t\subseteq f_t(N)$ for $t<\epsilon$. 
Let \(\{\gamma_{1},\ldots,\gamma_{s}\}\) be a basis of \([D_{0}]^{\perp}
\subset \mathrm{H}_{2}(Y;\mathbb{Z})\) and 
\(\{\tilde{\gamma}_{1},\ldots,\tilde{\gamma}_{s}\}\) be
the corresponding deformed cycles in \(X_{0}:=Y_0\setminus D_{0}\).
We can also regard \(\tilde{\gamma}_{i}\) as a cycle in \(X_{t}\).
By shrinking the neighborhood $N$ if necessary, we may assume that $f_t(\tilde{\gamma}_i)\cap D_t=\emptyset$ for $t<\epsilon$. 
\begin{lemma}
Let \(D_{t}\) be the degeneration as above and \(X_{t}:=Y_t\setminus D_{t}\).
Let \(\alpha_{t}\in
\mathrm{H}_{1}(D_{t};\mathbb{Z})\) be the homology class of the vanishing cycle.
Let \(\beta_{t}\in \mathrm{H}_{1}(D_{t};\mathbb{Z})\) such that 
\(\{\alpha_{t},\beta_{t}\}\) is a symplectic basis of \(\mathrm{H}_{1}(D_{t};\mathbb{Z})\). 
Then \(\{[\tilde{\alpha}_{t}],[\tilde{\beta}_{t}],[f_t(\tilde{\gamma}_{1})],
\ldots,[f_t(\tilde{\gamma}_{s})]\}\) is 
a basis of \(\mathrm{H}_{2}(X_{t};\mathbb{Q})\) for all $t\in [0,\epsilon]$.
\end{lemma}
\begin{proof}
Adapting the discussion in \S\ref{subsection:torelli} to \((Y_t,D_{t})\),
we obtain an exact sequence 
\begin{equation}
0\to \mathrm{H}^{1}(D_{t};\mathbb{Z})\to \mathrm{H}_{2}(X_{t};\mathbb{Z})
\to [D_{t}]^{\perp}\to 0.
\end{equation}
Notice that \([D_{t}]=[D_{0}]\) in \(\mathrm{H}^{2}(Y_t;\mathbb{Z})\). The result follows immediately.
\end{proof}

We will now study the period point associated with
\((\check{Y}_{t},D_t)\) as $t$ approaches zero. From the setup of $(Y_t,D_t)$, there is a smooth family of holomorphic $2$-forms  \(\Omega_{t}\) on \(X_{t}\) for $t\in [0,1]$.  
Notice that $f_t^*\Omega_t|_{Y_0\setminus N}$ converges to $\Omega_0$. In particular, this implies that
the period integrals
$\int_{f_t^*(\tilde{\gamma}_{i})} \Omega_{t}$ to $\int_{f_t^*(\tilde{\gamma}_{i})} \Omega_{0}$ as $t\rightarrow 0$. In particular, $\int_{f_t^*(\tilde{\gamma}_t)}\Omega_t$ is bounded. 
Recall that 
\begin{align*}
\begin{split}
\check{\Omega}_{t} = \omega_{t} - \mathrm{i}\cdot\operatorname{Im}\Omega_{t}.
\end{split}
\end{align*} 
The exactness of the Tian--Yau metric \(\omega_{t}\) implies that  
\begin{equation}
\int_{\delta} \check{\Omega}_{t} = 
\int_{\delta} \omega_{t} - \mathrm{i}\cdot\int_{\delta}\operatorname{Im}\Omega_{t} =
- \mathrm{i}\cdot\operatorname{Im}\int_{\delta}\Omega_{t}
\end{equation}
for all \(\delta\in\mathrm{H}_{2}(X_{t};\mathbb{Z})\).
Hence it suffices to compute the integrals
\begin{equation*}
\int_{\delta}\Omega_{t},~\delta\in \mathrm{H}_{2}(X_{t};\mathbb{Z}),
\end{equation*}
on the Tian--Yau space.

\begin{lemma} \label{local calculation}
Let \(\mu_{t}:=\mathrm{Res}_{D_{t}} \Omega_{t}\). We have
\begin{itemize}
	\itemsep=5pt
\item[(1)] \(\displaystyle\frac{\mathrm{i}}{2}\int_{D_{t}} \mu_{t}\wedge\bar{\mu}_{t} \to \infty\)
as \(t\to 0\);
\item[(2)] \(\displaystyle\int_{\alpha_{t}} \mu_{t}\)
is bounded.
\end{itemize}
\end{lemma}

\begin{proof}
Let \(p\in D_{0}^{\mathrm{sing}}\). Choose a coordinate chart \(V\) of \(Y\) around \(p\)
such that the degeneration \(D_{t}\to D_{0}\)
is given by \(xy=t\) in \(\mathbb{C}^{2}\).
Set \(E_{t}=\{xy=t+O(t)\}\). For \(t\ne 0\), we may write 
\begin{equation*}
\frac{\mathrm{d}x\wedge\mathrm{d}y}{xy-t}=
\frac{\mathrm{d}(xy-t)}{xy-t}\wedge\eta_{t}.
\end{equation*}
We see that 
\begin{equation*}
\mathrm{Res}_{E_{t}}\left(\frac{\mathrm{d}x\wedge\mathrm{d}y}{xy-t}\right)
=\left.\eta_{t}\right|_{E_{t}} = 
-\left.\frac{\mathrm{d}x}{x}\right|_{E_{t}}+O(t) = 
\left.\frac{\mathrm{d}y}{y}\right|_{E_{t}}+O(t).
\end{equation*}
Let \(\epsilon>0\) be small and set
\begin{equation*}
E_{t,\epsilon}:=\{(x,y)~|~|x|^{2}+|y|^{2}<\epsilon\}\cap E_t.
\end{equation*}
Since \(\lim_{t\to 0}\Omega_{t}=\Omega_{0}\) exists, there is a holomorphic function
\(f(x,y)\) with \(f(0,0)\ne 0\) such that
\begin{equation*}
\left.\Omega_{t}\right|_{V} = \frac{f(x,y)\mathrm{d}x\wedge\mathrm{d}y}{xy-t}+O(t).
\end{equation*}
It follows that
\begin{align*}
\frac{\mathrm{i}}{2}\int_{D_{t}} \mu_{t}\wedge\bar{\mu}_{t}
&\ge\frac{\mathrm{i}}{2}\int_{E_{t,\epsilon}} \mu_{t}\wedge\bar{\mu}_{t}+O(t)\\
&\ge\frac{\mathrm{i}}{2}\int_{E_{t,\epsilon}} 
|f(x,y)|^{2} \frac{\mathrm{d}x\wedge\mathrm{d}\bar{x}}{|x|^{2}}+O(t)\\
&\ge C \int_{0}^{2\pi}\int_{r=t/\sqrt{\epsilon}}^{\sqrt{\epsilon}}
\frac{\mathrm{d}r \mathrm{d}\theta}{r}+O(t) \to \infty
\end{align*}
as \(t\to 0\). This proves the item (1). 

Now we prove the item (2). For \(\epsilon>0\), we 
fix a point \((x_{0},y_{0})\in E_{t,\epsilon}\). Then 
\begin{equation*}
r_{t}(\theta):=(x_{0}e^{\mathrm{i}\theta},y_{0}e^{-\mathrm{i}\theta}),~\theta\in [0,2\pi),
\end{equation*}
is a closed curve in \(E_{t,\epsilon}\) representing \(\alpha_{t}\). We can compute
\begin{equation*}
\int_{\alpha_{t}}\mu_{t}=
\int_{0}^{2\pi} f(r_{t}(\theta))\mathrm{d}\theta
\end{equation*}
and conclude that it is bounded for \(\epsilon\ll 1\) 
since \(f\) is continuous at \((0,0)\).
\end{proof}
Now we are ready to prove Theorem \ref{thm:period-point}.
\begin{proof}[Proof of Theorem \ref{thm:period-point}]
 One first observe that  
	  \begin{align*}
	  \int_{\tilde{\beta}_t}\Omega_t=2\pi\int_{\beta_t}\mu_t \rightarrow \infty,
	  \end{align*} 
    where the first equality comes from residue calculation and Lemma \ref{local calculation} implies the asymptotics. Normalize $\Omega_t$ to $\tilde{\Omega}:=(\int_{\tilde{\beta}_t}\Omega_t)^{-1}\Omega_t$ for $t\neq 0$ and rescale the Tian--Yau metric accordingly, which does not change the complex structure of the hyperK\"ahler rotation $\check{X}_t$ and thus does not change $\check{Y}_t$. Notice that $\tilde{\beta}_t$ is the generator of the image of $\mathrm{H}^1(\check{D}_t,\mathbb{Z})$ in \eqref{rel les}, and thus the normalization here is exactly the one in \eqref{normalization GHK}. Therefore, all the periods of $\check{Y}_t$ converge to zero as $t\rightarrow 0$ under the normalization.
    In fact, note that the cycle \(f_{t}(\tilde{\gamma}_{i})\)
    is only defined up to an element in \(\mathrm{Im}(\mathrm{H}^{1}(D_{t},\mathbb{C})
    \to\mathrm{H}_{2}(X_{t},\mathbb{C}))=\mathrm{Span}_{\mathbb{C}}\{\tilde{\alpha}_{t},
    \tilde{\beta}_{t}\}\). Under the normalization, we see that
    \begin{eqnarray*}
    \exp\left(2\pi\mathrm{i}\cdot\int_{f_{t}(\tilde{\gamma}_{i})}\tilde{\Omega}_{t} \right)
    \to 1~\mbox{as}~t\to 0
    \end{eqnarray*}
    due to the fact that
    \begin{eqnarray*}
    \int_{f_{t}(\tilde{\alpha})}\Omega_{t}\left\slash\int_{f_{t}(\tilde{\beta})}\Omega_{t}\right.\to 0,~\mbox{as}~t\to 0.
    \end{eqnarray*}
    We conclude by 
    Lemma \ref{local calculation} and the observation above that all periods 
    converge to zero under the normalization. This also implies that $\check{Y}_t$ converges to the unique rational elliptic surface $\check{Y}_e$ with trivial periods from Theorem \ref{Torelli}.
 \end{proof}

The following is proved in Appendix \ref{ext res}
\begin{theorem}\label{extremal-LG} Consider a del Pezzo surface of degree $d$ as a monotone symplectic manifold and let $W\colon (\mathbb{C}^{\ast})^2\rightarrow \mathbb{C}$ be the mirror Landau--Ginzburg superpotential. Then it can be compactified to a rational elliptic surface $\check{Y}$ with an $I_d$-fibre and with trivial periods. In particular, $\check{Y}$ is extremal when $d=9, 8, 6, 3,2,1$ with singular configuration
$\mathrm{I}_9\mathrm{I}_1^3, \mathrm{I}_8\mathrm{I}_2\mathrm{I}_1^2, 
\mathrm{I}_6\mathrm{I}_3\mathrm{I}_2\mathrm{I}_1, \mathrm{I}_3\mathrm{IV}^*\mathrm{I}_1, 
\mathrm{I}_2\mathrm{III}^*\mathrm{I}_1, \mathrm{I}_1\mathrm{II}^*\mathrm{I}_1$. 
\end{theorem}

\section{Limiting Complex affine structure of special Lagrangian fibration in 
\texorpdfstring{$\mathbf{P}^1\times \mathbf{P}^1$}{}}
\label{sec:complex-affine-structure-of-slf}
 We first review the definition of complex affine structures
\cite{1997-Hitchin-the-moduli-space-of-special-lagrangian-submanifolds}. With the notation in Section \ref{sec: intro}, recall that
 $X\subseteq Y$ is a Calabi--Yau surface with a holomorphic 
volume form $\Omega$ and let
$\pi_{\alpha}\colon X\rightarrow B$ be a special Lagrangian fibration constructed in 
\cite{2021-Collins-Jacob-Lin-special-lagrangian-submaniflds-of-log-calabi-yau-manifolds} with a choice of $\alpha\in H_1(D,\mathbb{Z})$. 
Denote by $B_0$ the complement of the discriminant locus. 
Fix a reference point $u_0\in B_0$ and $\gamma\in  \mathrm{H}_{1}(L_{u_0},\mathbb{Z})$. 
Let $u\in B_0$ and $\phi\colon [0,1]\rightarrow B_0$ be a path 
with $\phi(0)=u_0$ and $\phi(1)=u$. Denote by $C_{\gamma,\phi}$ the ($n-1$)-cycle 
fibration over $\phi([0,1])$ such that fibre over $\phi(t)$ is the $1$-cycle in $L_{\phi(t)}$ 
which is the parallel transport of $\gamma$ along $\phi$. We define
\begin{align*}
x_{\gamma}(u):=\int_{C_{\gamma}}\operatorname{Im}\Omega.
\end{align*} 
If $\{\gamma_i\}$ is a basis of $\mathrm{H}_{1}(L_{u_0},\mathbb{Z})$, 
then $x_{\gamma_i}(u)$, $i=1,2$, form a coordinate system near 
the reference point $u_0$.
It is straightforward to check that the transition functions 
fall in $\mbox{GL}(2,\mathbb{Z})\rtimes\mathbb{R}^{2}$, and thus they define 
an integral affine structure on $B_0$. 
It is worth noticing that given an affine line (with rational slope) passing 
through $u\in B_0$, its tangent vector determines 
a cycle $\gamma\in \mathrm{H}_1(L_u,\mathbb{R})$ 
(or $\mathrm{H}_1(L_u,\mathbb{Z})$) and vice versa. 
Therefore, we will denote such an affine line 
by $l_{u,\gamma}$ or simply by $l_{\gamma}$ if no confusion occurs. 
\begin{remark}\label{singular ref}
	If one chooses a sequence of points $u_i\rightarrow u_{\ast}$ from a sector, 
	where $u_{\ast}$ falls in the discriminant locus, 
	then $\lim_{i\rightarrow \infty}x_{\gamma_i}(u_i)$ exists. Therefore, 
	one may take $u_{\ast}$ as the reference point as well. 
\end{remark}

\begin{remark} \label{germ}
	The germ of affine structures on a punctured disc is determined 
	by the affine monodromy around the puncture. In particular, if the affine structure 
	comes from a special Lagrangian fibration as above, the germ only depends 
	on the monodromy of the fibration.
\end{remark}
For the case $Y=\mathbf{P}^2$, the suitable hyperK\"ahler rotation $\check{X}$ of $X$ can always be compactified to an extremal rational elliptic surface \cite{2020-Collins-Jacob-Lin-the-syz-mirror-symmetry-conjecture-for-del-pezzo-surfaces-and-rational-elliptic-surfaces}. In particular, the base $B$ with the complex affine structure of the special Lagrangian fibration $\pi_{\alpha}$ as an affine manifold with singularities is independent of the choice of $\alpha \in \mathrm{H}_1(D,\mathbb{Z})$. Actually, this is true in a more general setting. 
\begin{proposition} Fix a pair $(Y,D)$ and $\alpha,\alpha'\in \mathrm{H}_1(D,\mathbb{Z})$. The bases $B,B'$ of the special Lagrangian fibrations $\pi_{\alpha},\pi_{\alpha'}$ with their complex affine coordinates are isomorphic as affine manifolds with singularities.  
\end{proposition}
\begin{proof}
	It suffices to prove that there exists a diffeomorphism $\psi\colon X\rightarrow X$ such that $\pi_{\alpha}=\pi_{\alpha'}\circ \psi$. Then the induced map of $\psi$ gives a diffeomorphism $B\cong B'$ identifying the complex affine structures. In particular, this answers a question of Hacking--Keating \cite{HK2}*{\S7.4}. 
	
	For any $\alpha,\alpha'\in \mathrm{H}_1(D,\mathbb{Z})$, there exists $1$-parameter family of pairs $(Y_t,D_t), t\in [0,1]$ such that $(Y_0,D_0)=(Y_1,D_1)=(Y,D)$ and the parallel transport sends $\alpha'$ to $\alpha$. Such a $1$-parameter family induces a diffeomorphism $\phi$ of $(X,\omega_{TY})$ for the exact Tian-Yau metric $\omega_{TY}$, sending the $[\tilde{\alpha}]$ to $[\tilde{\alpha}']$. Let $\Omega$ (and $\Omega'$) be the meromorphic $2$-form on $Y$ with simple pole along $D$ and $\int_{[\tilde{\alpha}]}\Omega\in \mathbb{R}_+$ (and $\int_{[\tilde{\alpha}']}\Omega'\in \mathbb{R}_+$ respectively). Then $\phi^{\ast}\Omega$ is a closed $2$-form with vanishing self-wedge and thus defines an integrable complex structure. 
	Moreover, the same underlying space $\underline{X}$ as $X$ with K\"ahler form $\phi^*\omega_{TY}$ and holomorphic volume form $\phi^*\Omega$ is an $\mathrm{ALH}^*$ gravitational instanton since the volume growth is still $r^{\frac{4}{3}}$. Now we have two $\mathrm{ALH}^*$ gravitational instantons $(\underline{X},\omega_{TY},\Omega)$ and $(\underline{X},\omega_{TY},\phi^*\Omega)$ as well as a morphism $\phi\colon X\rightarrow X$ 
  satisfying $\phi^*[\omega_{TY}]=[\omega_{TY}]$ and $\phi^*[\Omega]=[\Omega']$.
	From the Torelli theorem of the $ALH^*$ gravitational instantons \cite{CJL3}*{Theorem 3.9}, there exists a diffeomorphism $\phi':X\rightarrow X$ such that $\phi'^*\phi^*{\omega}_{TY}=\omega_{TY}$ and $\phi'^*\phi^*\Omega=\Omega'$. Taking $\psi=\phi\circ \phi'$, then $\psi_*[\tilde{\alpha}]=[\tilde{\alpha}']$. Any special Lagrangian torus of class $[\tilde{\alpha}]$ (and $[\tilde{\alpha}']$) must be a fibre of the special Lagrangian fibration with fibre class $[\tilde{\alpha}]$ (and $[\tilde{\alpha}']$ respectively). 
	  Thus, $\psi$ sends the special Lagrangian fibration in $X$ with fibre class $[\tilde{\alpha}]$ to the special Lagrangian fibration in $X$ with fibre class $[\tilde{\alpha}']$ as claimed.

\end{proof}

\subsection{The integral affine structure from Carl--Pumperla--Siebert}
Let $Y$ be a del Pezzo surface and $D$ be a smooth anti-canonical divisor.   
Carl--Pumperla--Siebert \cite{2022-Carl-Pumperla-Siebert-a-tropical-view-of-landau-ginzburg-models}
construct the mirror for the pair $(Y,D)$. 
We now describe the integral affine manifold with singularities, 
denoted as $B_{CPS}$, used in their
construction when $Y=\mathbf{P}^1\times \mathbf{P}^1$. 

We first start with $\mathbb{R}^2$ with the standard 
integral affine structure. There are four singularities 
located at \((\pm1/2,\pm1/2)\)
with the monodromy around each of the singularities conjugate to 
\(
\begin{pmatrix}
1 & 1 \\  0 & 1
\end{pmatrix}	
\).
    
To cooperate with the standard affine structure on $\mathbb{R}^2$, 
the branch cuts for the singularities are put at the following locations:
\begin{equation*}
\begin{aligned}[c]
l_1^+&=\left\{\left(\frac{1}{2}+t,\frac{1}{2}\right)\Big|~t\geq 0\right\},\\
l_2^+&=\left\{\left(-\frac{1}{2},\frac{1}{2}+t\right)\Big|~t\geq 0\right\},\\ 
l_3^+&=\left\{\left(-\frac{1}{2}+t,-\frac{1}{2}\right)\Big|~t\leq 0\right\},\\
l_4^+&=\left\{\left(\frac{1}{2},-\frac{1}{2}+t\right)\Big|~t\leq 0\right\}, 
\end{aligned}
\quad\quad
\begin{aligned}[c]
l_1^-&=\left\{\left(\frac{1}{2},\frac{1}{2}+t\right)\Big|~t\geq 0\right\},\\
l_2^-&=\left\{\left(-\frac{1}{2}+t,\frac{1}{2}\right)\Big|~t\leq 0\right\},\\
l_3^-&=\left\{\left(-\frac{1}{2},-\frac{1}{2}+t\right)\Big|~t\leq 0\right\},\\
l_4^-&=\left\{\left(\frac{1}{2}+t,-\frac{1}{2}\right)\Big|~t\geq 0\right\}.
\end{aligned}  
\end{equation*} 
$B_{CPS}$ is constructed by discarding the sectors 
bounded by $l_i^{\pm}$ in $\mathbb{R}$ and 
gluing the branch cuts $l_i^{\pm}$ with the affine transformations shown in Figure
\ref{fig:CPS-affine-base}.
\begin{figure}[ht]
\includegraphics[scale=0.5]{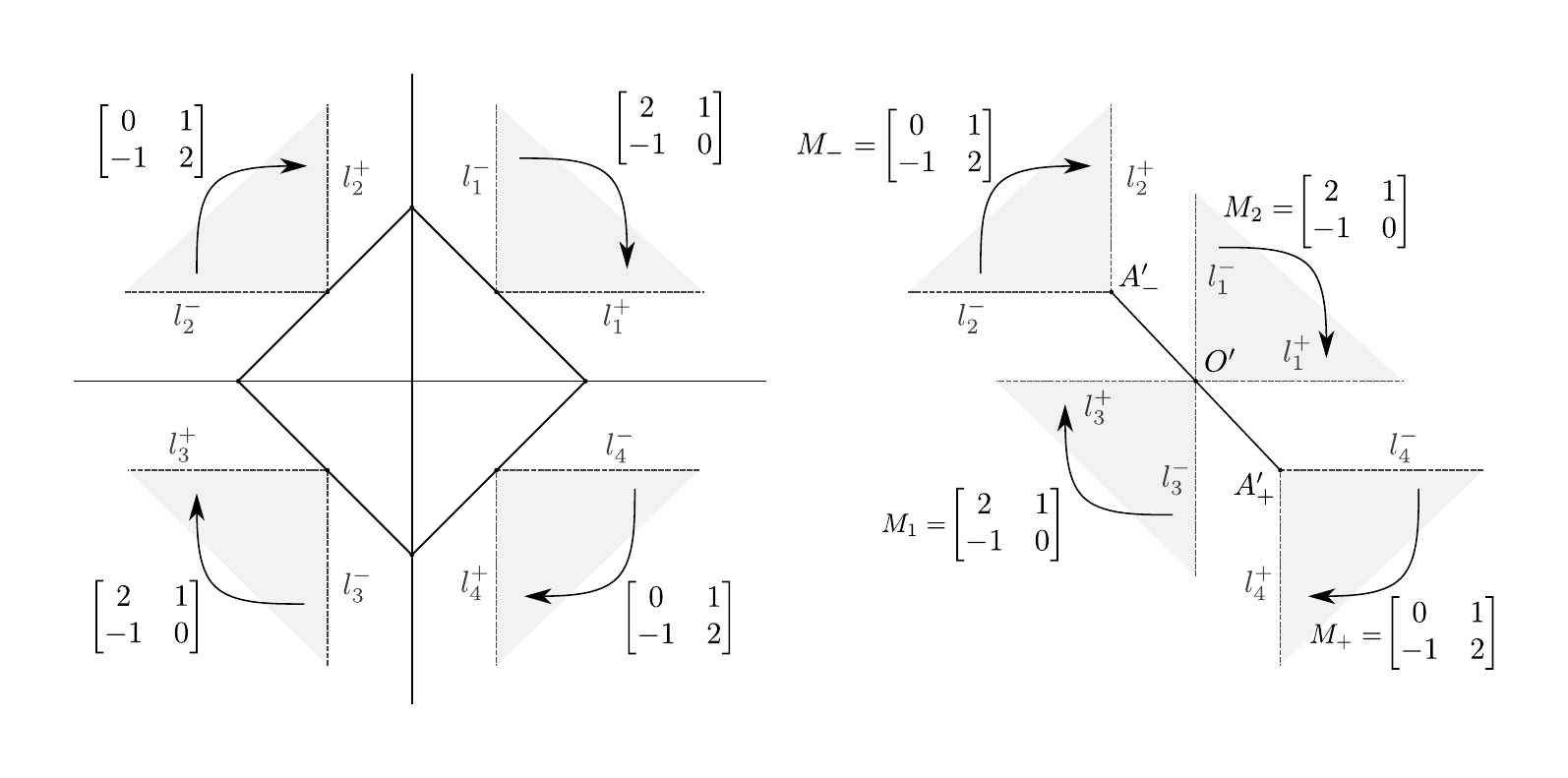}
\caption{The affine structure on \(B_{CPS}\) and its limit \(B_{CPS}'\).}
\label{fig:CPS-affine-base}
\end{figure}
However, this will not be the limit of 
the base of the special Lagrangians constructed in 
\cite{2021-Collins-Jacob-Lin-special-lagrangian-submaniflds-of-log-calabi-yau-manifolds}. 
We will compare the latter to a degeneration 
of $B_{CPS}$, denoted by $B_{CPS}'$. 
Roughly speaking, $B_{CPS}'$ is the integral 
affine manifold obtained by collapsing two of 
the singularities together. Let $A'_-(-1/2,1/2), O'=(0,0), 
A'_+=(1/2,-1/2)$ be the three singularities. 
We will choose the branch cuts as 
\begin{equation*}
\begin{aligned}[c]
l_1^+&=\left\{\left(t,0\right)\Big|~t\geq 0\right\},\\
l_2^+&=\left\{\left(-\frac{1}{2},\frac{1}{2}+t\right)\Big|~t\geq 0\right\},\\ 
l_3^+&=\left\{\left(t,0\right)\Big|~t\leq 0\right\},\\
l_4^+&=\left\{\left(\frac{1}{2},-\frac{1}{2}+t\right)\Big|~t\leq 0\right\}, 
\end{aligned}
\quad\quad
\begin{aligned}[c]
l_1^-&=\left\{\left(0,t\right)\Big|~t\geq 0\right\},\\
l_2^-&=\left\{\left(-\frac{1}{2}+t,\frac{1}{2}\right)\Big|~t\leq 0\right\},\\
l_3^-&=\left\{\left(0,t\right)\Big|~t\leq 0\right\},\\  
l_4^-&=\left\{\left(\frac{1}{2}+t,-\frac{1}{2}\right)\Big|~t\geq 0\right\}.
\end{aligned}  
\end{equation*} 

Then $B_{CPS}'$ is defined similarly as the complement of the sectors 
bounded by $l_{i}^{\pm}$ in $\mathbb{R}^2$ with the standard integral affine structure, 
where we glue the branch cuts $l_{i}^{\pm}$ with respect to the affine structures as in 
Figure \ref{fig:CPS-affine-base}. 
Notice that it is not clear that $B_{CPS}$ and $B_{CPS}'$ are related by moving worms 
introduced in \cite{2006-Kontsevich-affine-structures-and-non-archimedean-analytic-spaces}.

\subsection{Explicit calculation of the complex affine structure}\label{sec: dP8'}
In this section, we will compute the limit of the complex affine 
structure of the special Lagrangian fibration of 
$X_t=\mathbf{P}^1\times \mathbf{P}^1\setminus D_t$, 
where $D_t$ are smooth anti-canonical divisors 
degenerating to a nodal curve \(D_{0}\).

\begin{lemma}
 	Assume that $\sigma$ is a fibrewise involution on $X$. 
	Then $\sigma$ induces an 
 	involution $\underline{\sigma}$ on $B$. If furthermore
 	$\sigma^{\ast}(\operatorname{Im}\Omega)=-\operatorname{Im}\Omega$, then the fixed locus
 	$\mbox{Fix}(\underline{\sigma})\subseteq B$ defines an affine line with a rational slope. 
\end{lemma}	
\begin{proof}
	The first part of the lemma is straightforward. Notice that 
	$\mbox{Fix}(\underline{\sigma})$ is a manifold. Let $p\in \mbox{Fix}(\underline{\sigma})$ 
	and take $p$ as the reference point of the local affine coordinate chart. 
	Let $l_{\gamma}$ be the affine line tangent to $p\in \mbox{Fix}(\underline{\sigma})$ 
	with $\gamma\in \mathrm{H}_1(L_p,\mathbb{R})$. Then one 
	has $l_{\sigma_*(\gamma)}=\underline{\sigma}(l_{\gamma})=l_{\gamma}$. 
	In particular, $l_{\gamma}\subseteq \mbox{Fix}(\underline{\sigma})$, and thus the 
	fixed locus of $\underline{\sigma}$ is simply the affine line $l_{\gamma}$. 
	It is worth noticing that the induced action of $\sigma$ on $\mathrm{H}_{1}(L_{u},\mathbb{R})$ 
	is defined over $\mathbb{Z}$. In particular, if $n=2$ and
	$\sigma\neq \mathrm{id}$, then $l_{\gamma}$ has a rational slope. 
\end{proof}
 
The hyperK\"{a}hler rotation $\check{X}_{t}$ 
can be compactified into a rational elliptic surface $\check{Y}_{t}$ 
and $\lim_{t\rightarrow 0}\check{Y}_{t}=:\check{Y}$ converges in the moduli. 
Together with the hyperK\"{a}hler rotation relation \eqref{eq:hyperkah-rotation}
and the fact that the Tian--Yau metric is exact, it is sufficient to 
compute the affine structure induced by $\operatorname{Im}\check{\Omega}$.

For the case $Y=\mathbf{P}^1\times \mathbf{P}^1$ and \(D_{t}\) a family of 
smooth elliptic curves
degenerating to a nodal curve \(D_{0}\), as we have explained, 
$\check{Y}=\lim_{t\to 0}\check{Y}_{t}$ is the 
extremal rational elliptic surface with singular configuration $I_8I_2I_1^2$ 
with $\check{D}$ being the $I_8$ fibre at infinity. 
The toric model \((\check{Y}_{\mathrm{tor}},\check{D}_{\mathrm{tor}})\)
of \((\check{Y},\check{D})\)
is given by the maximal projective crepant partial 
desingularization \(\check{Y}_{\mathrm{tor}}\to \mathbf{P}_{\nabla}\), where 
\begin{equation*}
\nabla=\mathrm{Conv}\{(1,0),(0,1),(-1,0),(0,-1)\}.
\end{equation*}
The superpotential in this case is 
\begin{equation}
\label{eq:superpotential-p1xp1}
W\colon(\mathbb{C}^{\ast})^{2}\to\mathbb{C},~
(t_{1},t_{2})\mapsto t_{1}+t_{2}+\frac{1}{t_{1}}+\frac{1}{t_{2}}.
\end{equation}
Regarding \(W\) as an element in 
\(\mathrm{H}^{0}(\check{Y}_{\mathrm{tor}},-K_{\check{Y}_{\mathrm{tor}}})=
\mathrm{H}^{0}(\mathbf{P}_{\nabla},-K_{\mathbf{P}_{\nabla}})\), 
we see that  \(\check{Y}=\check{Z}\) is obtained by blowing up 
at \(\{W=0\}\cap \check{D}_{\mathrm{tor}}\).

The critical values of the superpotential \(W\)
are $0, \pm 4\in B:=\mathbb{C}$, 
which will be denoted by $O, A_{\pm}$ respectively in the sequel. 
Note that the fibre over $O$ is an $I_2$ fibre, while
the fibres over \(A_{\pm}\) are \(I_{1}\) fibres. 
Let $\check{\Omega}$ be the unique (up to a constant) 
meromorphic $(2,0)$-form on $\check{Y}$ 
with a simple pole along $\check{D}$ such that the complex conjugation 
on $\check{Y}$ is a fibre-preserving involution and thus the fixed locus 
of the induced action on $B$ is an affine line. 
Explicitly, \(\check{\Omega}\) is the pullback
of a meromorphic top form 
under \(\check{Y}\to\check{Y}_{\mathrm{tor}}\)
which comes from
a sequence of blowups at smooth points in \(\check{D}_{\mathrm{tor}}\).
In what follows, we will compute integrations of \(\check{\Omega}\)
over certain Lefschetz thimbles which are not contained in 
the exceptional locus of \(\check{Y}\to\check{Y}_{\mathrm{tor}}\).
We thus can compute these integrals on the 
maximal torus \((\mathbb{C}^{\ast})^{2}\).
In the sequel, we shall omit the pullback and simply write
\begin{equation}
\label{eq:Omega-check}
\check{\Omega} = \mathrm{i}\cdot
\frac{\mathrm{d}t_{1}}{t_{1}}\wedge\frac{\mathrm{d}t_{2}}{t_{2}}
\end{equation}
if no confusion occurs.
Here \((t_{1},t_{2})\) stands for the 
coordinates on \((\mathbb{C}^{\ast})^{2}\subset \check{Y}_{\mathrm{tor}}\).

Denote by \(q\) the coordinate on \(B\).
Look at the diagram
\begin{equation}
\label{dia:cover-ramifications}
\begin{tikzcd}
&(t_{1},t_{2})\ar[d,maps to] &[-2em]
(\mathbb{C}^{\ast})^{2}\ar[r,"W"]\ar[d,"\mathrm{pr}_{2}"] &\mathbb{C}.\\
&t_{2} &[-2em]\mathbb{C}^{\ast}
\end{tikzcd}
\end{equation}
For general \(q\in\mathbb{C}\), the pre-image 
\(W^{-1}(q)\) is an elliptic curve
with four points removed.
\begin{lemma}
The ramification points \(t_{2}\) of the 
double cover \(W^{-1}(q)\to\mathbb{C}^{\ast}\)
induced from the vertical arrow in \eqref{dia:cover-ramifications}
satisfy \((t_{2}^{2}+1-qt_{2})^{2}-4t_{2}^{2}=0\).
\end{lemma}
\begin{proof}
Compute
\begin{align*}
t_{1}+t_{2}+\frac{1}{t_{1}}+\frac{1}{t_{2}} - q = 0~&\Leftrightarrow~
t_{1}^{2}t_{2}+t_{1}t_{2}^{2}+t_{2}+t_{1}-qt_{1}t_{2}=0\\
&\Leftrightarrow~t_{2}t_{1}^{2}+(t_{2}^{2}+1-qt_{2})t_{1}+t_{2}=0.
\end{align*}
We see that \(t_{2}\) is a ramification point if and only if
\((t_{2}^{2}+1-qt_{2})^{2}-4t_{2}^{2}=0\).
\end{proof}

Making use of \eqref{eq:superpotential-p1xp1},
we can write \eqref{eq:Omega-check} as
\begin{equation}
\label{eq:omega-check-1}
\mathrm{i}\cdot\frac{\mathrm{d}q\wedge
\mathrm{d}t_{2}}{2t_{1}t_{2}+t_{2}^{2}+1-qt_{2}}
\end{equation}
whenever \(t_{1}^{2}\ne 1\).
Moreover, from \(W(t_{1},t_{2})-q=0\), we can solve
\begin{equation}
\label{eq:double-cover-t1}
t_{1}= f_{\pm}(t_{2}):= \frac{-(t_{2}^{2}+1-qt_{2})
\pm\sqrt{(t_{2}^{2}+1-qt_{2})^{2}-4t_{2}^{2}}}{2t_{2}}.
\end{equation}
Here we have chosen the branch cut to be 
the non-positive real axis to define the square root.
For a nonzero complex number \(z\notin\mathbb{R}_{-}\), we define
\begin{equation*}
\sqrt{z}:=\exp\left(\frac{1}{2}\log z\right),
\end{equation*}
where 
\begin{equation*}
\log z = \log|z|+\mathrm{i}\theta,~\theta\in (-\pi,\pi).
\end{equation*}
Substituting \(t_{1}\), we may rewrite \eqref{eq:omega-check-1} and get
\begin{equation}
\label{eq:check-omega-compute}
\check{\Omega}=
\mathrm{i}\cdot\frac{\mathrm{d}q\wedge\mathrm{d}t_{2}}
{\pm\sqrt{(t_{2}^{2}+1-qt_{2})^{2}-4t_{2}^{2}}}
\end{equation}
\begin{figure}[ht]
\includegraphics[scale=1]{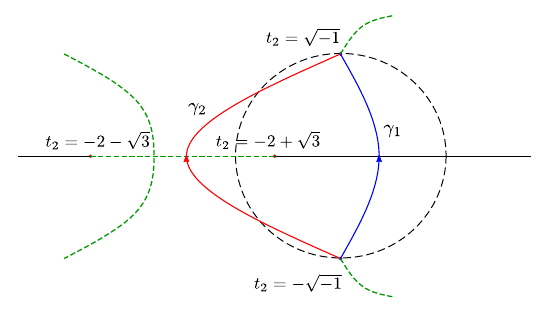}
\caption{The \(t_{2}\)-plane at \(q=-2\).
There are four ramification points: \(t_{2}=\pm~\mathrm{i}\) and 
\(t_{2}=-2\pm\sqrt{3}\). The green dashed line segments 
indicate the branch cuts we chose,
i.e., the loci where \((t_{2}^2+1-qt_2)^{2}-4t_{2}^{2}\le 0\).}
\label{figure:q-2-fibre}
\end{figure}

Each line segment of $\{q\in \mathbb{R}\}$ between $O,A_{\pm}$ 
is an affine line. Notice that due to the monodromy 
the corresponding $1$-cycles 
in the fibres for each of the affine line segments
$\overline{\infty A_{-}},\overline{A_{-}O},
\overline{OA_{+}},\overline{A_{+}\infty}$ 
are not the same after parallel transport. 

\begin{figure}[ht]
\includegraphics[scale=0.9]{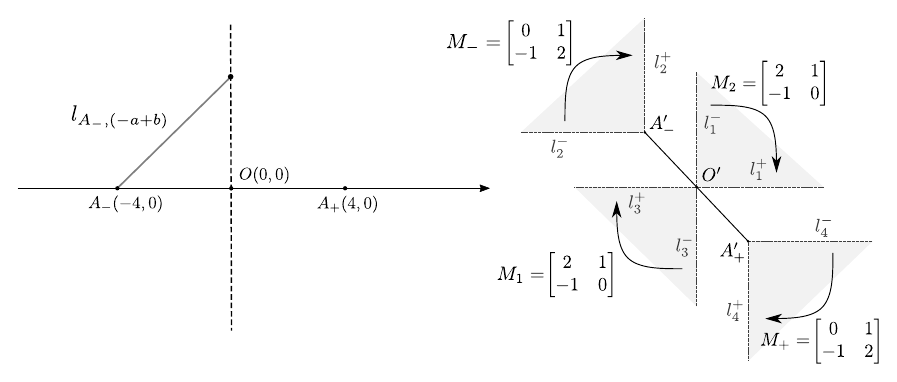}
\caption{Base of the SYZ fibration in coordinate \(q\) and $B'_{CPS}$.}
\label{fig:comparison}
\end{figure}

For the rest of this section, we will take $O$ as the reference point 
of the affine structure (see Remark \ref{singular ref}). We will restrict 
ourselves to the region $\operatorname{Im}q\geq 0$.  Let $a(q),b(q)$ be the vanishing thimbles 
from $O,A_-$ up to parallel transport via a path contained in the region $\mbox{Im}(q)>0$ into a fibre over \(q\)
with the orientation such that 
\begin{enumerate}
\item $\int_{a(q)}\operatorname{Re}\check{\Omega}<0$, if $q\in \overline{A_-O}$ and 
\item $\int_{b(q)}\operatorname{Re}\check{\Omega}>0$, if $\mbox{Im}(q)>0$ near $A_-$.  
\end{enumerate}
Let us describe the cycles \(a(q)\) and \(b(q)\) explicitly.
Let \(\gamma_{i}\colon [0,1]\to \mathbb{C}^{\ast}\)
be a parametrization of the oriented 
smooth curve drawn in \textsc{Figure}~\ref{figure:q-2-fibre}.
The cycle \(\partial a\) can be parameterized by
\begin{equation}
\label{eq:parameterize-a}
\gamma_{\partial a}(s)=
\begin{cases}
(\gamma_{1}(2s),f_{+}(\gamma_{1}(2s)),~&0\le s\le 1/2,\\
(\gamma_{1}(2-2s),f_{-}(\gamma_{1}(2-2s)),~&1/2\le s\le 1,
\end{cases}
\end{equation}
where \(f_{\pm}(t)\) is defined in \eqref{eq:double-cover-t1}. We
equip \(\partial a\) with 
the orientation induced from this parametrization to achieve the item (1) above.

\begin{lemma}
\label{lem:orientation-a}
Under the parametrization \eqref{eq:parameterize-a}, 
the item (1) above holds; namely
\begin{equation*}
\int_{a(q)}\operatorname{Re}\check{\Omega}<0,~\mbox{for}~q\in \overline{A_-O}.
\end{equation*}
\end{lemma}
\begin{proof}
Note that
\begin{equation*}
\int_{a(q)}\check{\Omega} = -\int_{a(q)} \iota_{\partial/\partial q}
\check{\Omega}\wedge\mathrm{d}q=
-\int_{0}^{q}\int_{\partial a(q)}(\iota_{\partial/\partial q}
\check{\Omega})\mathrm{d}q.
\end{equation*}
Since \(q<0\), the result will follow if we can show that
\begin{equation*}
\operatorname{Re}\left(\int_{\partial a(q)}\iota_{\partial/\partial q}
\check{\Omega}\right)<0.
\end{equation*}
Deforming the curve, we may 
assume \(\gamma_{1}(\theta)=\exp(\mathrm{i}\theta)\), 
\(\theta\in [-\theta_{0},\theta_{0}]\subseteq [-\pi,\pi]\),
with the counterclockwise orientation. By \eqref{eq:check-omega-compute}, 
\begin{equation*}
\int_{\partial a(q)}\iota_{\partial/\partial q}
\check{\Omega}=2\int_{-\theta_{0}}^{\theta_{0}} 
\frac{-\exp(\mathrm{i}\theta)\mathrm{d}\theta}
{\sqrt{\exp(2\mathrm{i}\theta)((2\cos\theta-q)^{2}-4)}}.
\end{equation*}
Note that \((2\cos\theta-q)^{2}-4>0\) for all 
\(\theta\in (-\theta_{0},\theta_{0})\).
We deduce that
\begin{equation*}
\sqrt{\exp(2\mathrm{i}\theta)((2\cos\theta-q)^{2}-4)}=\exp(\mathrm{i}\theta)
\sqrt{(2\cos\theta-q)^{2}-4}
\end{equation*}
and therefore 
\begin{equation*}
\int_{-\theta_{0}}^{\theta_{0}} 
\frac{-\exp(\mathrm{i}\theta)\mathrm{d}\theta}{\sqrt{\exp(2\mathrm{i}\theta)((2\cos\theta-q)^{2}-4)}}
=\int_{-\theta_{0}}^{\theta_{0}} 
\frac{-\mathrm{d}\theta}{\sqrt{(2\cos\theta-q)^{2}-4}}<0.
\end{equation*}
\end{proof}

Likewise, we can equip \(\partial b\) with an 
orientation through \(\gamma_{2}\)
to achieve item (2).
Starting from \(t_{2}=-\mathrm{i}\), we parameterize \(\partial b\) via
\begin{equation*}
\gamma_{\partial b}(s)=(\gamma_{2}(s),f_{+}(\gamma_{2}(s))),~0\le s\ll 1.
\end{equation*}
Let \(\{\gamma_{2}(s_{i})~|~0<s_{1}<\cdots<s_{k}<1\}\)
be the intersection of \(\gamma_{2}\) and the branch cut.
When \(\gamma_{2}\) meets the branch cut, 
the curve \(\gamma_{\partial b}\)
enters a different sheet and we shall exchange \(f_{\pm}(t)\). In other words,
\(\partial b\) can be parameterized by
\begin{equation*}
\gamma_{\partial b}(s)=
\begin{cases}
(\gamma_{2}(2s),f_{+}(\gamma_{2}(2s)),~&0\le s\le s_{1}/2\\
(\gamma_{2}(2s),f_{-}(\gamma_{2}(2s)),~&s_{1}/2 \le s\le s_{2}/2\\
\hspace{2cm}\vdots&\hspace{1cm}\vdots\\
(\gamma_{2}(2s),f_{\pm}(\gamma_{2}(2s)),~&s_{k}/2 \le s\le 1/2,\\
(\gamma_{2}(2-2s),f_{\mp}(\gamma_{2}(2-2s)),~&1/2 \le s\le 1-s_{k}/2,\\
\hspace{2cm}\vdots&\hspace{1cm}\vdots\\
(\gamma_{2}(2-2s),f_{+}(\gamma_{2}(2-2s)),~&1-s_{2}/2\le s\le 1-s_{1}/2,\\
(\gamma_{2}(2-2s),f_{-}(\gamma_{2}(2-2s)),~&1-s_{1}/2 \le s\le 1.\\
\end{cases}
\end{equation*}
and we equip \(\partial b\) with the induced orientation
of this parametrization.
Here the sign depends on the parity of \(k\).
For instance, if \(k=1\), then we shall pick \(f_{-}\)
for \(s_{1}/2\le s\le 1/2\) and \(f_{+}\)
for \(1/2\le s\le 1-s_{1}/2\).
Similar to the proof of Lemma~\ref{lem:orientation-a}, 
we can prove the follow lemma, which shows that the orientation
fulfills the requirement. 
\begin{lemma}
For \(q\in\overline{A_{-}O}\), we have
\begin{equation*}
\int_{b(q)}\operatorname{Re}\check{\Omega} = 0~\mbox{and}~
\int_{b(q)}\operatorname{Im}\check{\Omega} < 0.
\end{equation*}
\end{lemma}

\begin{lemma}
	If the reference point is changed to a point near $O$, then $\langle \partial a,\partial b\rangle=-2$, where $\langle, \rangle$ is the natural pairing of the fibre torus.
\end{lemma}

\begin{proof}
Since the intersection number is topological, we
can compute \(\langle \partial a,\partial b\rangle\)
at \(q_{0}\) near \(A_{-}\). 
It is clear that 
\(\langle \partial a,\partial b\rangle =\pm 2\).
To pin down the sign, we notice that if 
\(\ell_{\partial a}\) (resp.~\(\ell_{\partial b}\))
denotes the curve starting from \(q_{0}\)
with the direction such that

\begin{equation*}
\int_{\partial a} \check{\Omega}\in\mathbb{R} 
~~\left(\mbox{resp.~}
\int_{\partial b} \check{\Omega}\in\mathbb{R} \right)
\end{equation*}
and increases, then we have
\begin{equation}
\left\langle v_{\partial a},v_{\partial b}\right\rangle > 0
\end{equation}
with respect to the orientation on the base \(\mathbb{C}\),
where \(v_{\partial a}\) (resp.~\(v_{\partial b}\))
is the tangent vector of \(\ell_{\partial a}\) 
(resp.~\(\ell_{\partial b}\)) at \(q_{0}\) pointing in 
the direction in which the symplectic area is increasing.
Since \(\langle \partial a,\partial b\rangle\)
and \(\left\langle v_{\partial a},v_{\partial b}\right\rangle\)
differ by a sign, this completes the proof.
\end{proof}

Now define another set of affine coordinates by $x:=x_{\partial (a-b)}$ and $y:=x_{\partial(a+b)}$. 

\begin{figure}[ht]
\includegraphics[scale=3]{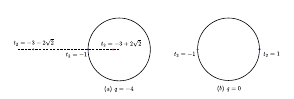}
\caption{When \(q\) moves from \(-4\) to \(0\) in \(\mathbb{R}\), two of the ramifications (red marks
in figure (a)) collapse to \(t_{2}=-1\) in figure (b) along the real axis \(t_{2}\in\mathbb{R}\).
Note that \(t_{2}=-1\) in figure (a) is a double ramification.
One can also check that the ramifications marked in blue travel along 
the semicircles \(t_{2}=(2+q\pm\sqrt{q^{2}+4q})/2\) 
when \(q\) moves from \(-4\) to \(0\).}
\end{figure}

\begin{lemma}
\label{lemma:affine-line-OA} 
The affine line $\operatorname{Im}q=0$ between $O$ and $A_{-}$ satisfies $x+y=0$. 
The same is true for the affine line $\operatorname{Im}q=0$ between $O$ and $A_+$. 
\end{lemma}
\begin{proof}
We will prove the former statement. The proof of the latter statement is similar. 
It suffices to prove that \(x_{\partial a}=0\); in other words, we have to show that
\begin{equation}
\label{lem:eq:to-prove}
\int_{a(q)} \check{\Omega}\in \mathbb{R}
\end{equation}
where \(a(q)\) is the vanishing thimble from \(O\) to \(q\in \overline{OA_{-}}\).
From the proof of Lemma \ref{lem:orientation-a}, one can
directly see that \eqref{lem:eq:to-prove} holds. However, we shall
give a more conceptual proof which will be useful later.

We can choose \(a\) such that the image of \(\partial a\) (over \(q\in\overline{OA_{-}}\))
under the projection 
\(\mathrm{pr}_{2}\) in diagram \eqref{dia:cover-ramifications} is given by 
a path \(\gamma\) connecting two ramification points \((2+q\pm\sqrt{q^{2}+4q})/2\)
and passing through the positive real axis \(\{t_{2}\in\mathbb{R}_{+}\}\).
We may also assume that \(\gamma\) is invariant under complex conjugation as well
(cf.~\textsc{Figure}~\ref{fig:vanishing-cycle-a}).
\begin{figure}[ht]
\includegraphics[scale=1]{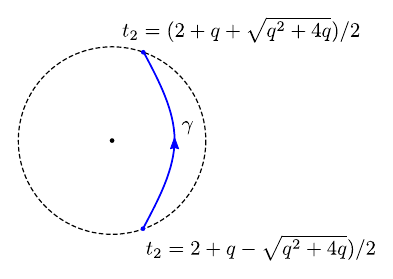}
\caption{The image of \(\partial a(q)\) under the projection \(\mathrm{pr}_{2}\).}
\label{fig:vanishing-cycle-a}
\end{figure}
Now we have
\begin{equation}
\label{eq:lemma-want-to-show}
\int_{\partial a} \iota_{\partial/\partial q}\check{\Omega}=
2\int_{\gamma}
\mathrm{i}\cdot\frac{\mathrm{d}t_{2}}{\sqrt{(t_{2}^{2}+1-qt_{2})^{2}-4t_{2}^{2}}}.
\end{equation}

Let \(\gamma_{+}:=\gamma\cap\{\operatorname{Im}t_{2}\ge 0\}\) and
\(\gamma_{-}:=\gamma\cap\{\operatorname{Im}t_{2}\le 0\}\).
Notice that \(\bar{\gamma}_{+}=-\gamma_{-}\) (orientation reversing). We can write
\begin{align*}
\int_{\gamma}&\frac{\mathrm{d}t_{2}}{\sqrt{(t_{2}^{2}+1-qt_{2})^{2}-4t_{2}^{2}}}\\
&=
\int_{\gamma_{+}}\frac{\mathrm{d}t_{2}}{\sqrt{(t_{2}^{2}+1-qt_{2})^{2}-4t_{2}^{2}}}+
\int_{\gamma_{-}}\frac{\mathrm{d}t_{2}}{\sqrt{(t_{2}^{2}+1-qt_{2})^{2}-4t_{2}^{2}}}\\
&=\int_{\gamma_{+}}\frac{\mathrm{d}t_{2}}{\sqrt{(t_{2}^{2}+1-qt_{2})^{2}-4t_{2}^{2}}}-
\int_{\gamma_{+}}\frac{\mathrm{d}\bar{t}_{2}}{\sqrt{(\bar{t}_{2}^{2}+1-q\bar{t}_{2})^{2}-4\bar{t}_{2}^{2}}}
\in\mathrm{i}\cdot\mathbb{R}.
\end{align*}
It follows from \eqref{eq:lemma-want-to-show} that
\begin{equation*}
\int_{\partial a} \iota_{\partial/\partial q}\check{\Omega}\in \mathbb{R}
\end{equation*}
which implies 
\begin{equation*}
\int_{a(q)} \check{\Omega}=
-\int_{0}^{q}\int_{\partial a} (\iota_{\partial/\partial q}\check{\Omega})\mathrm{d}q\in \mathbb{R}
\end{equation*}
as desired.
\end{proof}

Let us examine the branch cut 
when \(q=\mathrm{i}\xi\in \mathrm{i}\cdot\mathbb{R}_{+}\). 
We observe that 
the set \(\{t_{2}\in\mathbb{C}^{\ast}~|~(t_{2}^{2}+1-\mathrm{i} \xi t_{2})^{2} - 4t_{2}^{2} \le 0\}\)
is invariant under \(t_{2}\mapsto -\bar{t}_{2}\). Moreover,
\(\{t_{2}\in\mathbb{C}^{\ast}~|~(t_{2}^{2}+1-\mathrm{i} \xi t_{2})^{2} - 4t_{2}^{2} \le 0\}\cap 
\mathrm{i}\cdot\mathbb{R}_{+}=\emptyset\). Indeed, if \(t_{2}=\mathrm{i}x\) with \(x\in\mathbb{R}_{+}\),
\begin{align*}
(t_{2}^{2}+1-\mathrm{i} \xi t_{2})^{2} - 4t_{2}^{2}=
(-x^{2}+1+\xi x)^{2} + 4x^{2}>0.
\end{align*}

The following lemma shows that $\mathrm{i}\cdot\mathbb{R}_+$ is an affine ray. 
\begin{lemma} 
\label{branch cut from O}
Recall that the reference point is the origin \(O\).
We have 
\begin{equation*}
x_{\partial(-a+b)}(q)=0~\mbox{for}~q\in \mathrm{i}\cdot\mathbb{R}_+.
\end{equation*} 	
\end{lemma}
\begin{proof}
Denote by \(\gamma_{1}\) and \(\gamma_{2}\)
the image of \(\partial a\) and \(\partial b\) 
under the projection \(\mathrm{pr}_{2}\) in \eqref{dia:cover-ramifications}.
For \(q=\mathrm{i}\cdot \xi
\in\mathrm{i}\cdot\mathbb{R}_{+}\), we may assume the \(\gamma_{i}\) are smooth curves 
connecting two out of four ramifications that collapse to 
one point when \(q\to 0\) 
(cf.~\textsc{Figure}~\ref{fig:ramifications-imaginary}). 
By our choice of orientations, we see
that the image of \(\gamma_{2}-\gamma_{1}\) is a closed curve on \(\mathbb{C}^{\ast}\).
As in the proof of Lemma \ref{lemma:affine-line-OA}, 
we shall compute an integral of a holomorphic function
over \(\gamma_{2}-\gamma_{1}\). We can deform \(\gamma_{2}\) a bit and assume that 
\(\gamma_{2}-\gamma_{1}\) is symmetric with respect to the imaginary axis.

Write \(\gamma_{2}-\gamma_{1}=\gamma_{+}\cup\gamma_{-}\) where
\begin{align*}
\gamma_{+}:=\{p\in \gamma_{2}-\gamma_{1}~|~\operatorname{Re}p\ge 0\}~\mbox{and}
~\gamma_{-}:=\{p\in \gamma_{2}-\gamma_{1}~|~\operatorname{Re}p\le 0\}.
\end{align*}
We see that \(\gamma_{-}\) and \(-\bar{\gamma}_{+}\) are set-theoretically symmetric 
with respect to the imaginary axis, but the orientation is reversed.

The integrand is given by 
\begin{equation*}
\iota_{\mathrm{i}\partial/\partial \xi}\check{\Omega}=\pm
\frac{\mathrm{d}t_{2}}{\sqrt{(t_{2}^{2}+1-qt_{2})^{2}-4t_{2}^{2}}}=:F(t_{2})
\end{equation*}
with \(q\in\mathrm{i}\cdot
\mathbb{R}_{+}\). 
Since the branch cut does not intersect \(\mathrm{i}\cdot\mathbb{R}_{+}\), 
the integrand \(F(t_{2})\)
satisfies \(F(-\bar{t}_{2})=-F(\bar{t}_{2})\).
Now for such a function \(F(t_{2})\),
\begin{align*}
\int_{\gamma_{+}} F(t_{2})+\int_{\gamma_{-}} F(t_{2})
=\int_{\gamma_{+}} F(t_{2})+\int_{\gamma_{+}} 
F(\bar{t}_{2})\in\mathbb{R}
\end{align*}
which implies the result.
\begin{figure}[ht]
\includegraphics[scale=1]{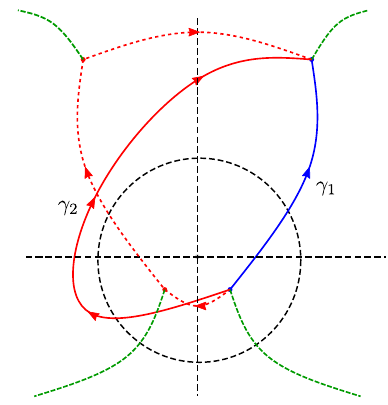}
\caption{}
\label{fig:ramifications-imaginary}
\end{figure}
\end{proof}

Next we will choose the branch cut of $A_{\pm}$ to be the affine 
line segments $\overline{\infty A_{\pm}}$ contained in $\{\mbox{Im}(q)=0\}$ as in 
\textsc{Figure}~\ref{fig:comparison}. 
If we move $q$ across $\overline{\infty A_{-}}$ from $\mbox{Im}(q)>0$ to $\mbox{Im}(q)<0$, 
then the counter-clockwise monodromy around $A_{\pm}$ is given by 
  \begin{align*}
       \partial a &\mapsto \partial a+ 2\partial b \\ 
       \partial b &\mapsto \partial b.
  \end{align*}
Therefore, the corresponding clockwise affine transformation is 
\(M_{\pm}\) (cf.~\textsc{Figure}~\ref{fig:comparison}).

\begin{lemma} \label{affine cut 2}
$x_{\partial(-a+b)}(q)-x_{\partial (-a+b)}(A_-)=0$ for $q<-4$. 	
\end{lemma}
\begin{proof}
The argument used in Lemma \ref{branch cut from O} applies to this case.
Denote again by \(\gamma_{1}\) and \(\gamma_{2}\)
the image of \(\partial a\) and \(\partial b\) on the \(t_{2}\)-plane
under the projection \(\mathrm{pr}_{2}\) in \eqref{dia:cover-ramifications}.
For \(q<-4\), we may assume the \(\gamma_{i}\) are smooth curves 
connecting two out of four ramifications that collapse to 
one point when \(q\to -4\) 
(cf.~\textsc{Figure}~\ref{fig:ramifications-negative-real}). 
We can deform \(\gamma_{1}\) a bit and assume that 
\(\gamma_{2}-\gamma_{1}\) is symmetric with respect to the real axis.

Write \(\gamma_{2}-\gamma_{1}=\gamma_{+}\cup\gamma_{-}\) where
\begin{align*}
\gamma_{+}&:=\{p\in \gamma_{2}-\gamma_{1}~|~\operatorname{Im}p\ge 0\},~\mbox{and}\\
\gamma_{-}&:=\{p\in \gamma_{2}-\gamma_{1}~|~\operatorname{Im}p\le 0\}.
\end{align*}
We see that in the present case 
\(\gamma_{-}\) and \(\bar{\gamma}_{+}\) are set-theoretically symmetric 
with respect to the real axis, but the orientation is reversed.

Since \(\gamma_{2}-\gamma_{1}\) does not intersect the branch cut on the real axis,
it follows that the integrand \(F(t_{2})\) obeys the rule \(F(\bar{t}_{2})=-\bar{F}(t_{2})\),
and we have
\begin{align*}
\int_{\gamma_{+}} F(t_{2})+\int_{\gamma_{-}} F(t_{2})
=\int_{\gamma_{+}} F(t_{2})+\int_{\gamma_{+}} 
\bar{F}(t_{2})\in\mathbb{R}.
\end{align*}

\begin{figure}[ht]
\includegraphics[scale=1]{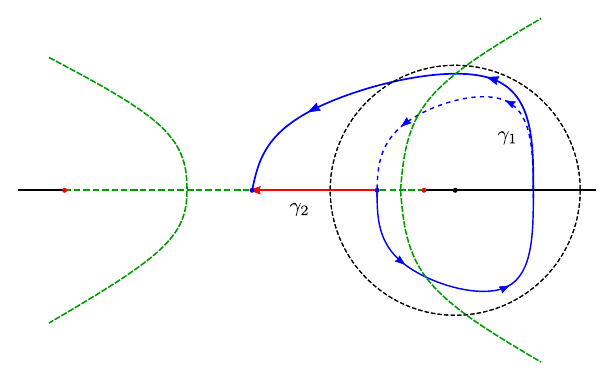}
\caption{\(\gamma_{1}\) (the blue solid curve) is the image of \(\partial a\), and 
\(\gamma_{2}\) (the red solid curve) is the image of \(\partial b\). 
We can deform the upper part of \(\gamma_{1}\) into the union of 
the blue dashed curve and \(\gamma_{2}\).}
\label{fig:ramifications-negative-real}
\end{figure}
\end{proof}
The proof of the following lemma is 
deferred to Appendix \ref{app-proof-lemma}.
\begin{lemma}\label{initial ray b}
  The affine line $l_{A_-,\partial b}$ intersects $l_{O,\partial (-a+b)}$.
\end{lemma}
Now we will choose $\overline{\infty A_{\pm}}$, which is an 
affine ray by Lemma \ref{affine cut 2}, to be the branch cuts from $A_{\pm}$. 
 The monodromy $M_{\pm}$ coincides with the gluing transformation of cuts around $A'_{\pm}$ in $B'_{CPS}$. Thus, there exists an affine isomorphism $\Psi$ between a punctured neighborhood of $A_{\pm}$ in $B$ and a punctured neighborhood of $A'_{\pm} \in B'_{CPS}$ by Remark \ref{germ}. 
 With the scaling of the affine coordinates chosen appropriately, the affine isomorphism $\Psi$ can be chosen such that it can be extended to a neighborhood of $\overline{OA_-}$ in $B$. Moreover, by Lemma \ref{lemma:affine-line-OA}, the extended isomorphism will identify a neighborhood of $\overline{OA_-}$ in $B$ with a that of $\overline{O'A'_-}$ in $B'_{CPS}$ 
 with $\Psi(\overline{OA_-})=\overline{O'A'_-}$.
 Let $M_1,M_2$ be as in \textsc{Figure} \ref{fig:comparison}. Then $M=M_1M_2$ is the monodromy around $O$. We will choose two branch cuts from $O$ to be $\mathrm{i}\mathbb{R}_{\pm}$ with the gluing transformation $M_1,M_2$. 
 From Lemma \ref{branch cut from O} and Remark \ref{germ}, we can again extend the affine isomorphism $\Psi$ to an affine isomorphism 
 identifying a neighborhood of $O$ with a neighborhood of $O'$ such that $\Psi(\mathrm{i}\mathbb{R}_+)$ is the branch cut from $O'$ in \textsc{Figure} \ref{fig:comparison}. 
 
 Next let us denote the intersection of the affine lines $l_{A_-,\partial b}$ and $l_{O,\partial(-a+b)}$ by $C$. Then $C'=\Psi(C)=(0,1)$. Moreover, $\Psi$ extends to an affine isomorphism from the triangle $OA_-C\subseteq B$ to $O'A_-C'\subseteq B_{CPS}$ by Lemma \ref{lemma:affine-line-OA}, Lemma \ref{branch cut from O}, and Lemma \ref{initial ray b}. Again Lemma \ref{branch cut from O}, Lemma \ref{affine cut 2}, and Lemma \ref{initial ray b} imply that such an affine isomorphism can be extended from the unbounded region in $B$ surrounded by $l_{A_-,\partial (-a+b)}, l_{O,-\partial a}, l_{O,\partial (-a+b)}$ to the corresponding  unbounded region in $B_{CPS}$ as shown in \textsc{Figure} \ref{fig:comparison}. By the symmetry $q\rightarrow \bar{q}$ and $q\rightarrow -\bar{q}$, the affine isomorphism extends to $\Psi\colon B\cong B'_{CPS}$.

 \begin{remark}
 	 Although the integral affine structures with singularities on $B$ and $B_{CPS}$ are different (even up to moving worms), the authors expect that the corresponding tropical counting of the $\mathbb{A}^1$-curves and the product structures of the algebra generated by theta functions are the same. The authors will leave it for future work. 
 \end{remark}

 \begin{remark}
 	It is worth noticing that the Mordell--Weil group of $\check{Y}_e$ is $\mathbb{Z}_4$ \cite{MP} (see also \cite{SS}*{p.102}) and thus gives a $\mathbb{Z}_4$-action on $\check{Y}_e$ which descends to the identity on the base of the elliptic fibration. 
 \end{remark}

\section{Limiting Complex Affine Structure for del Pezzo Surface of Degree \texorpdfstring{$3$}{3} and \texorpdfstring{$4$}{4}} \label{sec: dP3dP4}
 In this section, we will describe the limiting complex affine structure for del Pezzo surface of degree $3$ and $4$. 
 
 We first deal the the case of del Pezzo surface of degree $3$ and recall the geometric setup: let $Y=\mathbb{P}^2$, $D$ is a smooth cubic surface and $X$ be the complement $X=Y\setminus D$. Then $X$ admits a special Lagrangian fibration with respect to $\omega_{TY},\Omega$ from Theorem \ref{CJL} and a suitable hyperK\"ahler rotation $\check{X}$ with hyperK\"ahler triple $(\check{\omega},\check{\Omega})$ can be compactified to a rational elliptic surface $\check{Y}_9$ with an $I_9$-fibre from Theorem \ref{theorem:CJL}. It is known that there exists a unique rational elliptic surface with an $I_9$-fibre with singular configuration $I_9I_1^3$. The authors then use the geometry of $\check{Y}_9$ to computed the complex affine structure of the special Lagrangian fibration on $X$. It is given by the affine subspace of $\mathbb{R}^2$ as shown in Figure \ref{fig: P2} and then glue the corresponding edges with the prescribed linear transformations.
   \begin{figure}
 	\centering
 	\includegraphics[scale=0.8]{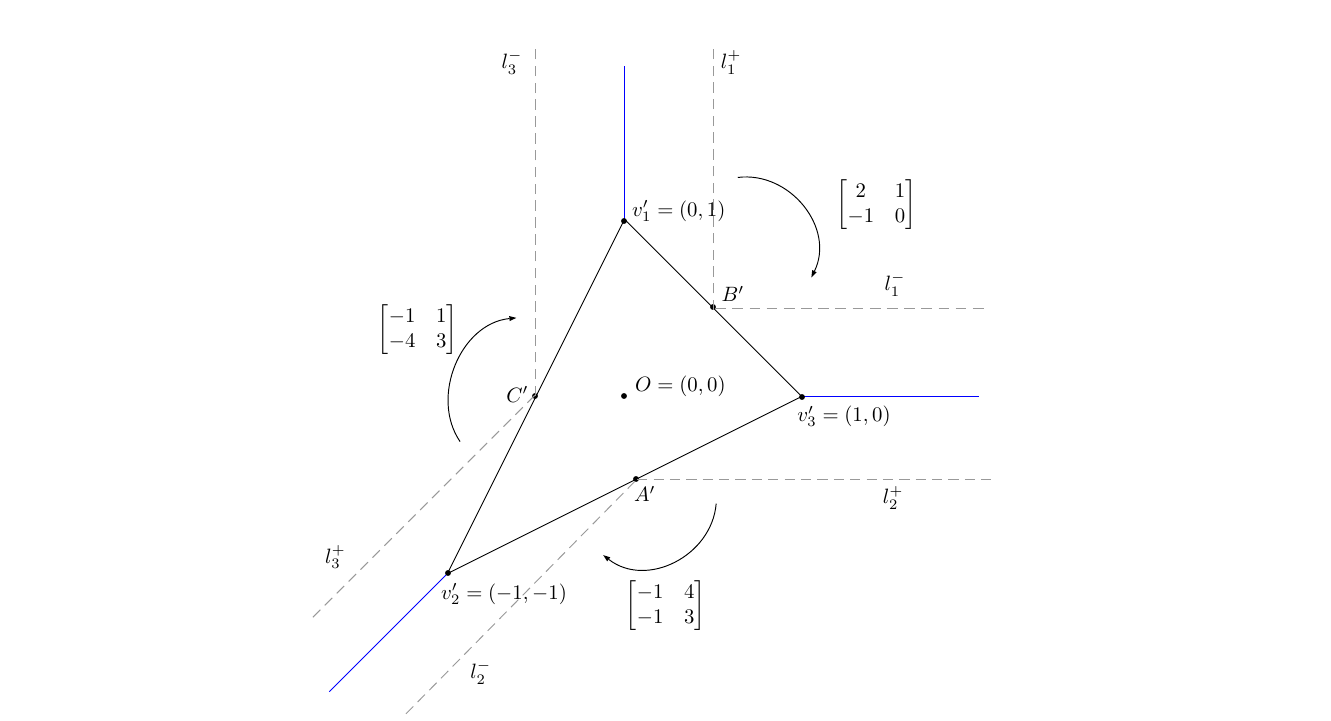}
 	\caption{The limiting complex affine structure of special Lagrangian fibration in $\mathbf{P}^2$}
 	\label{fig: P2}
 \end{figure}
 It is proved that there exists an $\mathbb{Z}_3\oplus \mathbb{Z}_3$-action on $\check{X}$ preserving $(\check{\omega},\check{\Omega})$ in \cite{2020-Collins-Jacob-Lin-the-syz-mirror-symmetry-conjecture-for-del-pezzo-surfaces-and-rational-elliptic-surfaces}*{Corollary 5.11} and the proof therein. Denote $t_1,t_2$ be the local complex coordinates of $\check{Y}_9$, where $t_1$ is the fibre coordinate and $t_2$ is the base coordinate on $\mathbb{P}^1$.  
  The first copy of $\mathbb{Z}_3$ is the Mordell-Weil group of $\check{Y}_9$ and acts as translation by holomorphic sections. In other words, it is of the form 
  \begin{align*}
  	  t_1\mapsto t_1+\sigma(t_2), t_2\mapsto t_2,
  \end{align*} where $\sigma$ is the holomorphic section.   
   The second copy of $\mathbb{Z}_3$ preserves a holomorphic section of $\check{Y}_9$ and sends fibres to fibres. Thus, the second $\mathbb{Z}_3$-action descends to the base $\mathbb{P}^1$ of $\check{Y}_9$ and acts simply as the rotation 
   \begin{align*}
   	  t_1\mapsto t_1, t_2\mapsto e^{2\pi i/3} t_2.
   \end{align*} In particular, the fibres over $t_2=0,\infty$ are fixed. Let $r=(1,1)\in \mathbb{Z}_3\oplus \mathbb{Z}_3$ and denote $\langle r\rangle\cong \mathbb{Z}_3$ be the subgroup generated by $r$. Since both $\mathbb{Z}_3$-action is fibration preserving, $\check{Y}_9/\langle r\rangle$ still admits an elliptic fibration structure. The $I_9$-fibre of $\check{Y}_9$ reduces to an $I_3$-fibre after the quotient. The three $I_1$-fibres are identified to a single $I_1$-fibre. The fibre over $t_2=0$ is set-wisely fixed by the action of $\langle r\rangle$ and after the quotient becomes an orbifold $\mathbb{P}^1$ with three orbifold singularities locally modeled by $\mathbb{C}^2/\mathbb{Z}_3$. Thus, the minimal resolution $\check{Y}$ of $\check{Y}_9/\langle r\rangle$ is a crepant resolution. From the formula of anti-canonical divisor for ramified cover, for instance \cite{BPV}*{Lemma I.17.1}, we have $\check{Y}$ is a again a rational elliptic surface and its singular configuration is $IV^*I_3I_1$. We then have the following observations:
\begin{itemize}
	\item consider the del Pezzo surface of degree $3$ as a monotone symplectic manifold, then its mirror Landau-Ginzburg potential can be compactified to $\check{Y}$ from Theorem \ref{extremal-LG}. 
	\item The $\langle r\rangle$-action preserves $(\check{\omega},\check{\Omega})$ on $\check{X}$ implies that the $\langle r\rangle$-action preserves $(\omega,\Omega)$ on $X$ and descends to a $\mathbb{Z}_3$-action on the complex affine structure of $X$. 
	Therefore, the underlying space of $\check{X}\setminus \{t_2=0\}$, viewed as a subspace of $X$,  admits a special Lagrangian fibration with respect to $(\omega,\Omega)$. Moreover, the complex affine structure of this special Lagrangian fibration is the $\mathbb{Z}_3$-quotient of the complex affine structure of the special Lagrangian fibration in $X$. 
	\item Recall that the complex affine structure of the the limiting special Lagrangian fibration of del Pezzo surfaces of degree $3$ is determined by $\check{Y}$ by Theorem \ref{thm:period-point}. Since $\check{Y}\rightarrow \check{Y}_9/\langle r\rangle$ is a crepant resolution, the complex affine structure from the limiting special Lagrangian fibration in del Pezzo surface of degree $3$ is exactly the same as that of the special Lagrangian fibration in $\check{X}\setminus \{t_2=0\}$. 
\end{itemize} 
Now the $\mathbb{Z}_3$-action on the complex affine structure of the special Lagrangian fibration in $\mathbb{P}^2$ can be described as the linear transformation $\begin{pmatrix} 0 & 1 \\ -1 & -1 \end{pmatrix}$ restricting to the affine subspace in Figure \ref{fig: P2}. It is straight-forward to check that the transformation preserves the gluings listed Figure \ref{fig: P2}. 
Altogether, we reach the following conclusion: 
\begin{theorem}
	The limiting complex affine structure for del Pezzo surface of degree $3$ is a $\mathbb{Z}_3$ quotient of the complex affine structure of the special Lagrangian in $\mathbb{P}^2$. The $\mathbb{Z}_3$-action is generated by $\begin{pmatrix}0 & -1\\ 1 & -1 \end{pmatrix}$. Explicitly, it is given by the shaded affine subset of $\mathbb{R}^2$ in Figure \ref{fig: dP3} below with $\overline{OB'}$ identified with $\overline{OC'}$ and two vertical boundaries identified.
	 \begin{figure}[h]
		\centering
		\includegraphics[scale=0.3]{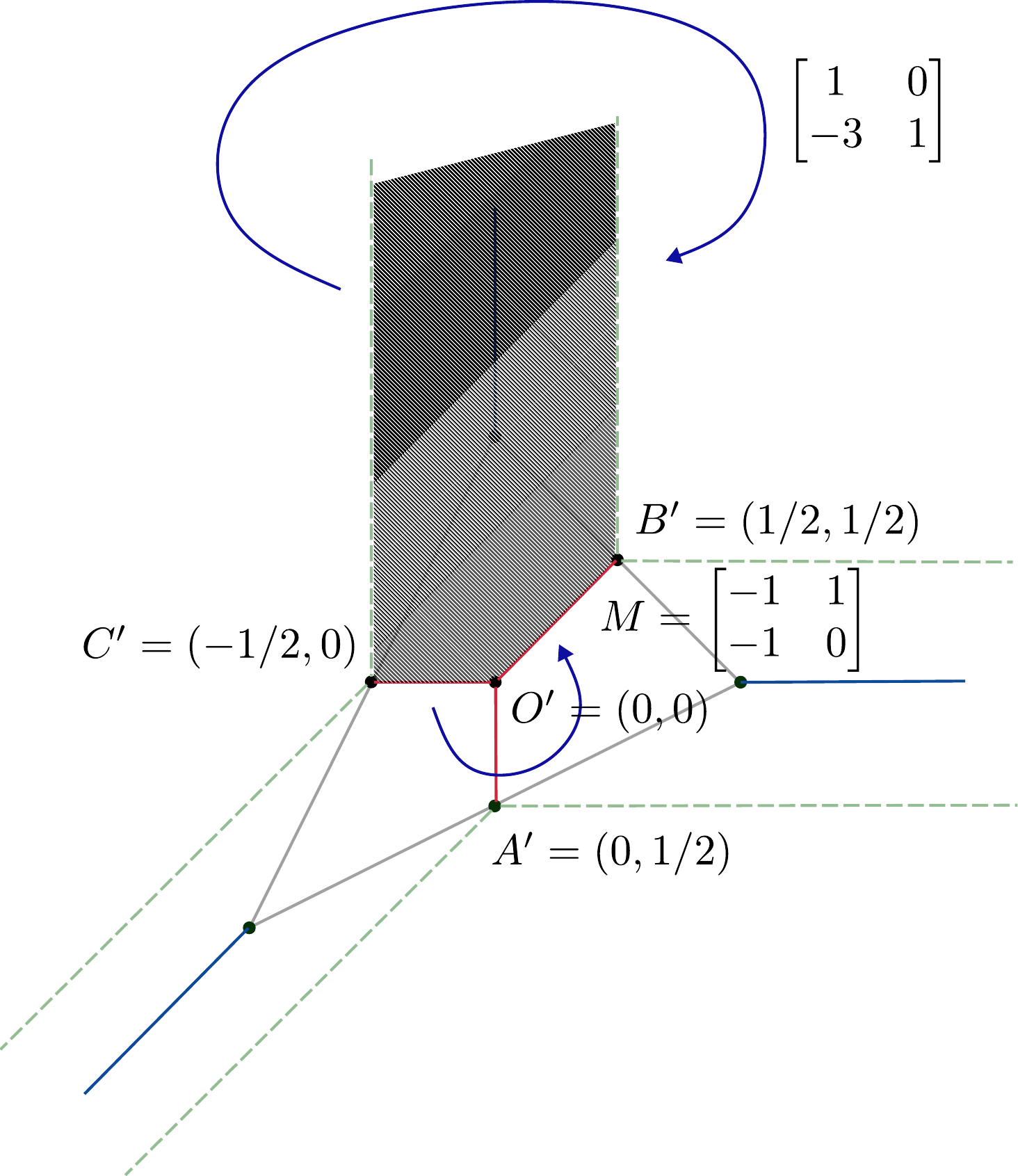}
		\caption{The limiting affine structure for del Pezzo surfaces of degree $3$.}
		\label{fig: dP3}
	\end{figure}
\end{theorem}	

Next we will use a similar argument to derive the limiting complex affine structure for the special Lagrangian fibration of del Pezzo surface of degree $4$. We will start from $Y=\mathbb{P}^1\times \mathbb{P}^1$, $D_t$ smooth anti-canonical divisors of $Y$ converging to an irreducible nodal anti-canonical divisor and $X_t=Y\setminus D_t$. Following the similar argument above, a suitable hyperK\"ahler rotation $\check{X}_t$ can be compactified to a rational elliptic surface $\check{Y}_{t}$ and $\check{Y}_t$ converging to a rational elliptic surface $\check{Y}_{8'}$ with singular configuration $I_8I_2I_1^2$. There exists a $\mathbb{Z}_2\oplus \mathbb{Z}_2$-action on $\check{Y}_{8'}$. The first copy of the $\mathbb{Z}_2$ is the fibrewise negation and the second copy descends to the base of the fibration. Again let $r=(1,1)\subseteq \mathbb{Z}_2\oplus \mathbb{Z}_2$ and consider the elliptic fibration $\check{Y}_{8'}/\langle r\rangle$. The $I_8$-fibre of $\check{Y}_{8'}$ reduces to an $I_4$-fibre, the two $I_1$-fibres are identified under the quotient and the two components of the $I_2$-fibre are set-wisely fixed, each with two fixed points. The minimal resolution $\check{Y}$ of $\check{Y}_{8'}/\langle r\rangle$ is the unique rational elliptic surface with singular configuration $I_1^*I_4I_1$, which is the compactification of the mirror Landau-Ginzburg potential of del Pezzo surface of degree $4$, viewed as a monotone symplectic manifold. Slightly different from the previous case, $r^*\check{\Omega}=-\check{\Omega}$. From the definition of the complex affine structure and the hyperK\"ahler rotation relation \eqref{eq:hyperkah-rotation}, this still implies that $r$ sends affine line to affine line with respect to the limiting complex affine structure of the special Lagrangian fibration in $\mathbb{P}^1\times \mathbb{P}^1$. This leads the following theorem. 
\begin{theorem}
		The limiting complex affine structure for del Pezzo surface of degree $4$ is a $\mathbb{Z}_2$ quotient of the limiting complex affine structure of the special Lagrangian in $\mathbb{P}^1\times \mathbb{P}^1$. Explicitly, it is given by the shaded affine subset of $\mathbb{R}^2$ in Figure \ref{fig: dP4} below with $l_1^+$ identified with $l_3^-$ and $l_4^+$ identified with $l_4^-$.
	\begin{figure}[h]
		\centering
		\includegraphics[scale=1]{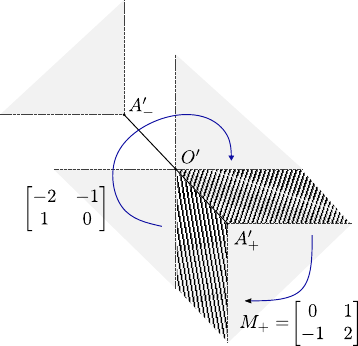}
		\caption{The limiting affine structure for del Pezzo surfaces of degree $4$.}
		\label{fig: dP4}
	\end{figure}
\end{theorem}

%
%
%

\section{Mirror construction for del Pezzo surfaces using immersed Lagrangians} \label{sec:Floer}

\subsection{Smoothing of $A_n$ singularities} \label{sec:An}

Smoothings and resolutions of $A_n$ singularities provide excellent examples of local Calabi-Yau manifolds.  The symplectic geometry of $A_n$ smoothings has been well studied by the early work of \cites{Thomas, ST}.  SYZ mirror symmetry for $A_n$ singularities was studied in \cites{LLW, Chan-An}.  Rigorously speaking, the previous construction concerns only about Floer theory of smooth SYZ fibers, and points that are mirror to the singular SYZ fibers are missing.
In this section, we use the method of \cites{CHL-glue, HKL} to glue in the deformation spaces of immersed nodal spheres in a smoothing of an $A_n$-singularity.  This fills in the corresponding punctures and completes the SYZ mirrors.

Consider a smoothing of $A_n$ surface singularities
$$ S = S^{(n)}=\left\{(X,Y,Z) \in \C^3: XY = \prod_{i=1}^{n+1} (Z- \epsilon_i) \right\} $$
where $\epsilon_1 < \epsilon_2 < \ldots < \epsilon_{n+1}$ are taken to be real numbers for simplicity.  It has a K\"ahler form restricted from $\C^3$.

We construct Lagrangians in $S$ using symplectic reduction.  The $\bS^1$-action $(X,Y,Z)\mapsto (e^{i\theta}X, e^{-i\theta}Y,Z)$ is Hamiltonian and has the moment map $\mu = |X|^2 - |Y|^2: \C^3 \to \R$.  The coordinate $Z$ is $\bS^1$-invariant and descends to the reduced space $S \sslash_a \bS^1$.  This gives an identification $S \sslash_a \bS^1 \cong \C$.  We will take the level to be $a=0$.  By virtue of the dimension, any curve in $\C$ corresponds to a Lagrangian in $S$.

Consider a simple loop $C$ in $\C$ which winds around all the points $\epsilon_1,\ldots,\epsilon_{n+1} \in \R$, and is invariant under complex conjugation.  It corresponds to a Lagrangian torus $$L_0 := Z^{-1}(C) \cap \mu^{-1}\{0\} \subset S.$$
For simplicity, let's apply a diffeomorphism $\rho$ on the base $\C$ that commutes with complex conjugation (and in particular preserves the real line) and is equal to identity outside a compact subset, such that $\rho(C)$ is a circle with center lying in the real line.

For each $i=1,\ldots,n+1$, we take a circle $C'_i$ that satisfies the following requirements.
\begin{enumerate}
	\item $C'_i$ passes through the point $\epsilon_i' = \rho(\epsilon_i)$.
	\item The center of $C'_i$ lies in the real line.
	\item $C'_i$ and $\rho(C)$ intersect at two points.
	\item The two strips bounded by $\rho(C)$ and $C'_i$ have the same symplectic area with respect to $(\rho^{-1})^*\omega_{\mathrm{red}}$ (where $\omega_{\mathrm{red}}$ denotes the reduced symplectic form).
\end{enumerate}
See Figure \ref{fig:An-Lag}.

\begin{figure}[h]
	\begin{center}
		\includegraphics[scale=0.5]{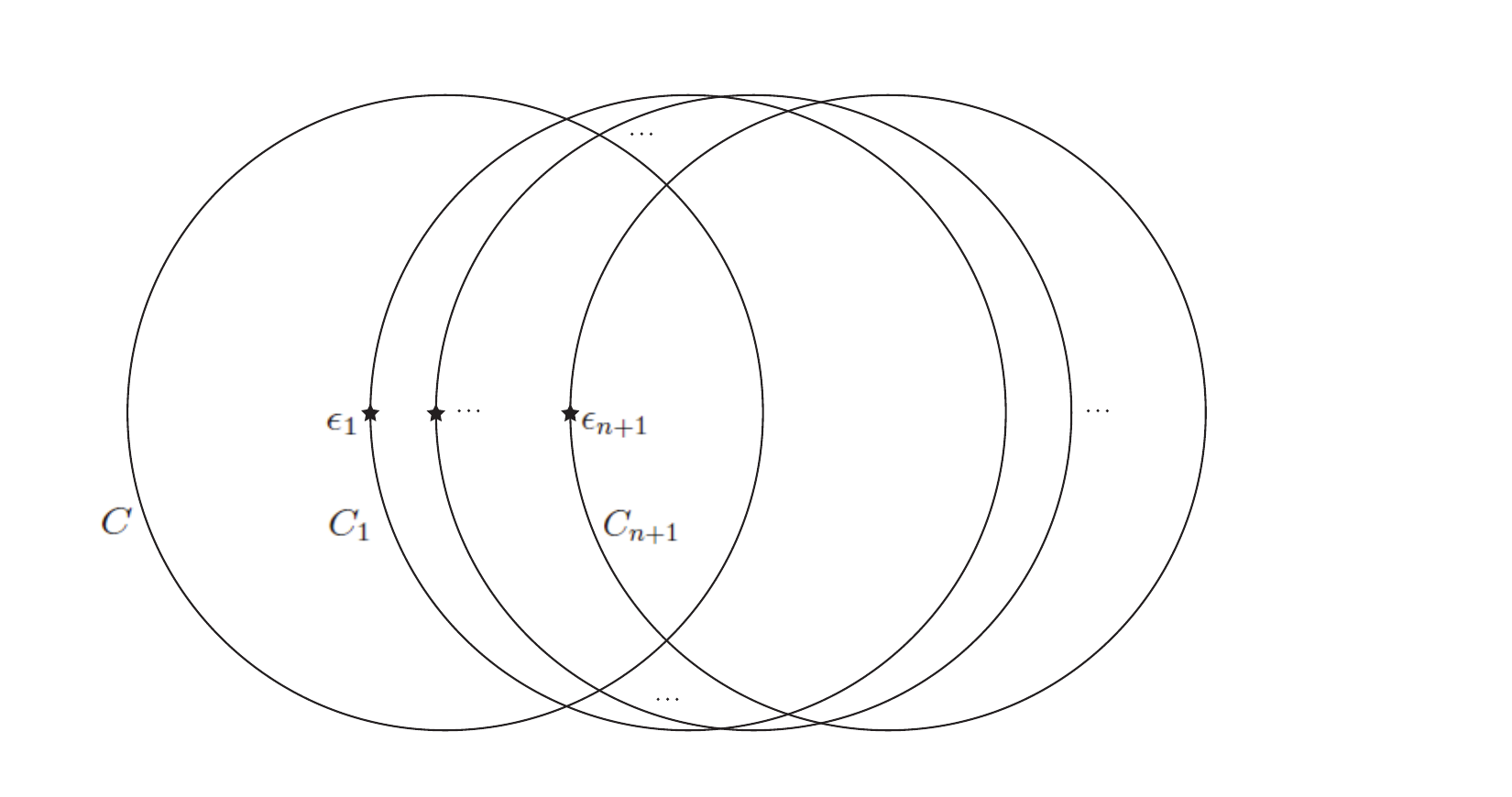}
		\caption{The images of the Lagrangians in the reduced space.}\label{fig:An-Lag}
	\end{center}
\end{figure}

\begin{lemma}
	For each $i=1,\ldots,n+1$, there exists a unique circle $C'_i$ that satisfies the above requirements.
\end{lemma}

\begin{proof}
	Given Condition (1) and (2), the only remaining freedom is the choice of the radius $R$ (in usual Euclidean sense) of $C'_i$.  There exists $a>0$ such that Condition (3) is satisfied if and only if $R>a$.  For $R<R'$, the left strip bounded by $(C'_{R'}, C)$ is a subset of that bounded by $(C'_{R}, C)$.  Similarly the right strip bounded by $(C'_{R}, C)$is a subset of that bounded by $(C'_{R'}, C)$.  This implies the symplectic area of the left (or right) strip strictly decreases (or increases resp.) as $R$ increases.  At the limit $R=a$ (at which $C'_i$ and $C$ touch at a point), the symplectic area of the right strip is equal to $0$; At the limit $R=+\infty$ (at which $C'_i$ is a straight line), the right strip contains an unbounded right half plane, and so has $+\infty$ symplectic area since $(\rho^{-1})^*\omega_{\mathrm{red}} = \omega_{\mathrm{red}}$ (as $\rho=\mathrm{Id}$ outside a compact subset).  As a consequence, there exists a unique intermediate value $R$ of the radius such that Condition (4) is satisfied.
\end{proof}

In the above choice, $C$ and $C_i :=\rho^{-1}(C'_i)$ intersect at two points and bound two strips with the same area under $\omega_{\mathrm{red}}$.  

\begin{corollary}
	In this setting, each pair of $C_i$ and $C_j$ for $i\not=j$ also intersects at two points and bounds two strips, where the two strips have areas equal to each other under $\omega_{\textrm{red}}$.
\end{corollary}
\begin{proof}
	Without loss of generality, let $\epsilon_i' < \epsilon_j'$ in the real line.  We want to show that $C_j'$ does not entirely lie in the disc bounded by $C_i'$, and hence $C_i',C_j'$ intersect at two points and bound two strips.  We consider the area form $(\rho^{-1})^*\omega_{\textrm{red}}$ below.
	
	Suppose this is not true.  Let $A$ be the area of each of the two strips bounded by $C',C_i'$.  In this situation, the area of the left strip bounded by $C',C_j'$ is greater than $A$, and the area of the right strip bounded by $C',C_j'$ is less than $A$.  See the left figure below. This contradicts that the strips bounded by $C',C_j'$ have equal area $\tilde{A}$.
	
	Thus, we are in the situation of the right two figures below.  Then
	$$ \tilde{A}-A = B-D = B'$$
	and hence the strips bounded by $C_i'$ and $C_j'$ have equal area $B = B'+D$ under the area form $(\rho^{-1})^*\omega_{\textrm{red}}$.  Applying $\rho^{-1}$, conclusion follows for $C_i=\rho^{-1}(C_i')$ and $C_j=\rho^{-1}(C_j')$.
\end{proof}

\begin{figure}[h]
	\begin{center}
		\includegraphics[scale=0.4]{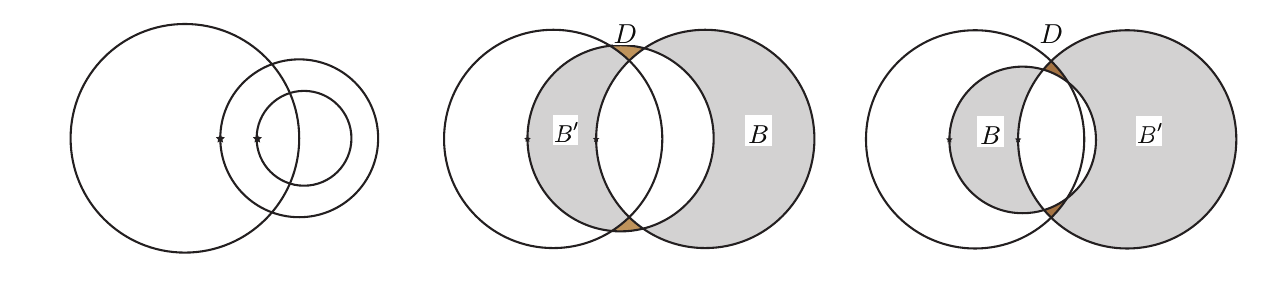}
		\caption{Strips involved in Fukaya isomorphism}\label{fig:areas-An}
	\end{center}
\end{figure}

We fix a point $p\in \R\subset \C$ in the reduced base, which lies in the common intersection of the discs bounded by the circles $C$ and $C_j$ for $j=1,\ldots,n+1$. This corresponds to an anti-canonical divisor $Z^{-1}\{p\} \subset S$.  We denote the complement by
$$ S^\circ := S - Z^{-1}\{p\}.$$
All the Lagrangians we have constructed lie in $S^\circ$.

It easily follows from the symplectic reduction that:

\begin{lemma} \label{lem:grading}
	For each $i=1,\ldots,n+1$, 
	$$L_i  := Z^{-1}(C_i) \cap \mu^{-1}\{0\}\subset S^\circ $$ 
	is a Lagrangian immersed sphere with a single nodal point at $X=Y=0, Z=\epsilon_i$.
\end{lemma}

The Lagrangian immersions $L_i$ are graded by the holomorphic volume form $dX \wedge dY / (Z - p)$ on $S^\circ$.  Thus, the Maslov index formula of \cite[Lemma 3.1]{2007-Auroux-morror-symmetry-and-t-duality-in-the-complement-of-an-anticanonical-divisor} can be applied.

We decorate the Lagrangian torus $L_0$ with flat $\Lambda_0^\times$-connections.  
To write them down explicitly, we proceed as follows.  First, the fibration $Z: S \to \C$ is trivialized after restricting to the open subset $\{Y\not=0\} \subset S$.  Then we take a basis $\{e_1,e_2\}$ of $\pi_1(L_0) \cong \Z^2$, where $e_1$ is along the Hamiltonian $\bS^1$-action, and $e_2$ is clockwise along the base circle $C$.  Then the flat connections are parametrized by $z,w \in \Lambda_0^\times$, where $z,w$ are the holonomies along $e_1,e_2 \in \pi_1(L_0)$ respectively.  $z$ is the monodromy-invariant direction.  We denote these flat connections by $\nabla^{L_0, (z,w)}$.

For the immersed spheres $L_i$, $i=1,\ldots,n+1$, the self-nodal point gives two immersed generators denoted by $U_i$ and $V_i$, which correspond to the two branch jumps $q_1\mapsto q_2$ and $q_2 \mapsto q_1$, where $\{q_1,q_2\}\subset \hat{L}_i$ is the preimage of the nodal point in the normalization $\hat{L}_i \cong \bS^2$ of $L_i$.  Using the grading in Lemma \ref{lem:grading}, these generators have degree $1$.  We shall consider the deformations $b_i=u_i U_i + v_i V_i \in \CF^1(L_i,L_i)$.

Below, we shall follow the construction in \cite{HKL}.  Note that the immersed Lagrangians $\bL_i$ are invariant under complex conjugation, and so the argument for weakly unobstructedness still applies.  Readers are referred to there for detail.  

\begin{lemma}[Lemma 3.3 of \cite{HKL}]
	Consider $L_i \subset S^\circ$ and let $b_i=u_i U_i + v_i V_i \in \CF^1(L_i,L_i)$ with $u_i,v_i \in \Lambda_0$ and $\val(u_iv_i)>0$.  We have
	$m_0^{L_i,b_i}=0$.
\end{lemma}

For $i=1,\ldots,n$, the immersed spheres $L_i$ and $L_{i+1}$ cleanly intersect at two circles (projecting to the two intersection points between $C_i$ and $C_{i+1}$ in the base).  Fix a perfect Morse function on each of these circles.  The critical points in one of the circles give Floer generators of degree 0 and 1 in $\CF(L_i,L_{i+1})$ (or degree 2 and 1 in $\CF(L_{i+1},L_i)$); the critical points in the other circle give Floer generators of degree 1 and 2 in $\CF(L_i,L_{i+1})$ (or degree 1 and 0 in $\CF(L_{i+1},L_i)$).  Denote by $\alpha_i, \beta_i$ the degree zero generators in $\CF(L_i,L_{i+1})$ and $\CF(L_{i+1},L_i)$ respectively.  We find transition between $b_i$ and $b_{i+1}$ such that $(\alpha_i,\beta_i)$ forms an isomorphism pair:
\begin{equation} \label{eq:isom}
	m_1^{b_i,b_{i+1}}(\alpha_i)=0, m_1^{b_{i+1},b_{i}}(\beta_i)=0; m_2^{b_i,b_{i+1},b_i}(\alpha_i, \beta_i) = 1_{L_i}, m_2^{b_{i+1},b_{i},b_{i+1}}(\alpha_i, \beta_i) = 1_{L_{i+1}}
\end{equation}
where $1_{L}$ denotes the unit of $L$.
Similarly, let $\alpha_0$ and $\beta_0$ denote the degree zero generators in $\CF(L_0,L_1)$ and $\CF(L_1,L_0)$.

\begin{theorem} \label{thm:iso-An}
	For $i=1,\ldots,n$, $\alpha_i$ is an isomorphism between $(L_i,b_i)$ and $(L_{i+1},b_{i+1})$ if $v_i = u_{i+1}^{-1}$ and $u_i = u_{i+1}^2 v_{i+1}$.
	Moreover, $\alpha_0$ is an isomorphism between $(L_0,b_0)$ and $(L_1,b_1)$ if $w=u_1$ and $z=u_1v_1-1$.
\end{theorem}

\begin{proof}
	The assertion that $\alpha_0$ is an isomorphism under the given transition map was proved in \cite[Theorem 3.7]{HKL}.  The key ingredient is that
	$$ m_1^{b_0,b_1}(\alpha_0) = (1-wu_1^{-1}) X + h(u_1v_1)(z + 1 - u_1v_1) Y $$
	for a certain series $h$, where $X$ and $Y$ denote the degree one Floer generators over the base intersection points of $\beta_0$ and $\alpha_0$ respectively.  Then $m_1^{b_0,b_1}(\alpha_0)=0$ gives the gluing formula.
	
	Take a circle of the same radius as $C_i$ and $C_{i+1}$, whose center lies in the real line, and that passes through a point between $\epsilon_i$ and $\epsilon_{i+1}$, see Figure \ref{fig:isom}.  This corresponds to a Lagrangian torus $L_{i,i+1}$.  As in the previous paragraph, we have isomorphisms $\alpha^{L_i,L_{i,i+1}}$ and $\alpha^{L_{i+1},L_{i,i+1}}$, under the gluing maps
	$$ v_i = w^{L_{i,i+1}} = u_{i+1}^{-1}, \, u_iv_i = 1+z^{L_{i,i+1}} = u_{i+1}v_{i+1}.  $$
	Moreover, we have $m_2(\alpha_i,\alpha^{L_{i+1},L_{i,i+1}}) = \alpha^{L_i,L_{i,i+1}}$: consider the triangle bounded by $L_i, L_{i+1}, L_{i,i+1}$ shaded in Figure \ref{fig:isom}.  The conic fibration trivializes over a neighborhood of this triangle.  In particular, it lifts uniquely to a holomorphic triangle in $S^\circ$ with corners passing through the maximum points corresponding to $\alpha_i$ and $\alpha^{L_{i+1},L_{i,i+1}}$.  This is the only holomorphic polygon with input corners being $\alpha_i$ and $\alpha^{L_{i+1},L_{i,i+1}}$.
	
	Thus, $\alpha^{L_i,L_{i,i+1}}$ is an isomorphism under the same gluing equation.  This gives the claimed transition map.
\end{proof}

\begin{figure}[h]
	\begin{center}
		\includegraphics[scale=0.4]{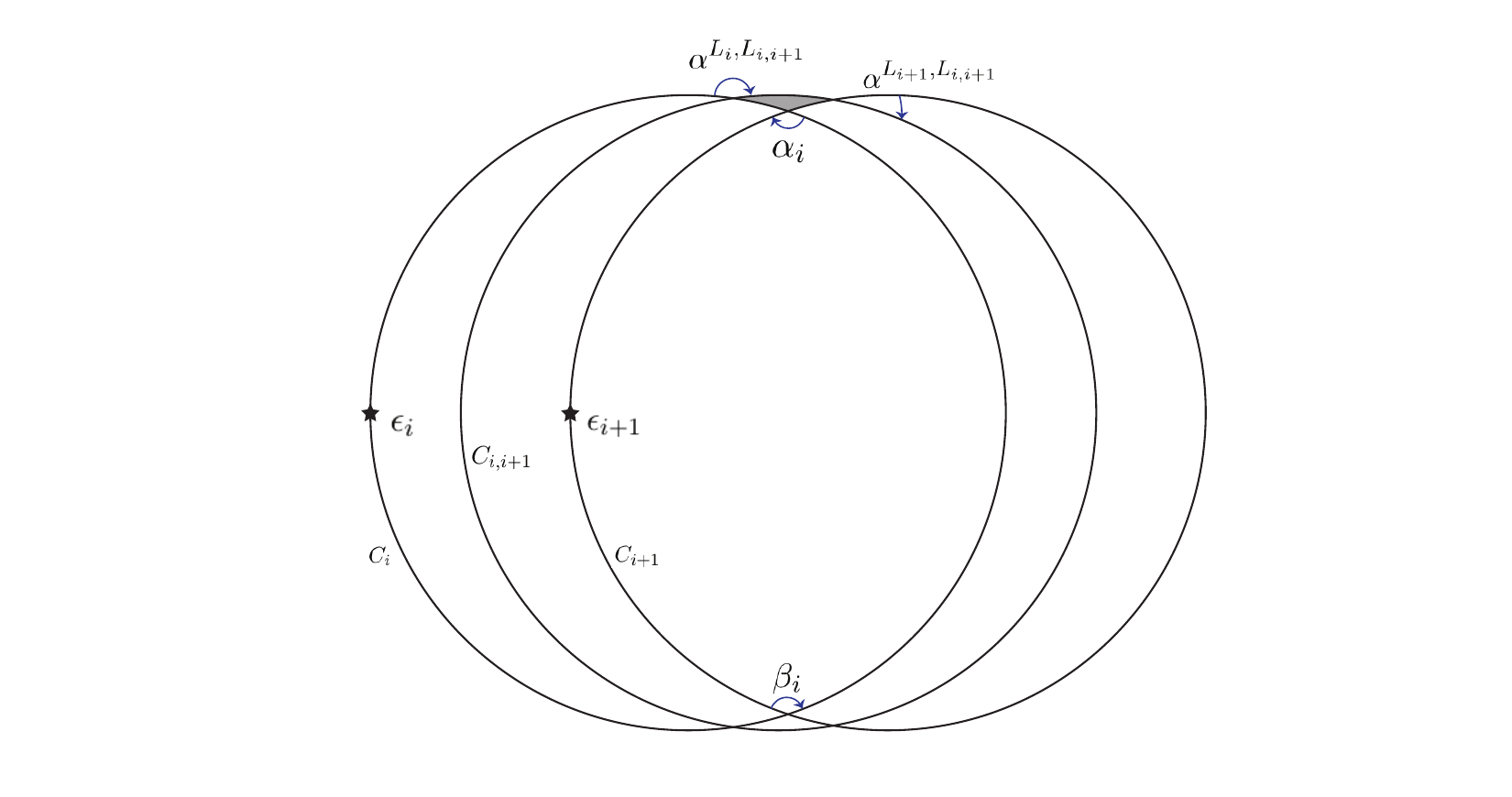}
		\caption{The isomorphisms between $L_i$, $L_{i+1}$, and the torus $L_{i,i+1}$.}\label{fig:isom}
	\end{center}
\end{figure}


Note that the above gluing equations are the transition maps for a toric resolution of an $A_n$ singularity.  

\subsection{Blowing-up over the Novikov field}
To have a better geometric understanding of the above moduli of Lagrangians, we define the following analog of blowing-up over the Novikov ring.

\begin{definition} \label{def:blowup}
	$$\mathbb{P}^1_{\Lambda} := \left(\Lambda^2 - \{0\}\right)/ \Lambda^\times.$$
	The blowing up of $\Lambda^2$ at a point $(u_0,v_0) \in \Lambda^2$ is defined as 
	$$ \{((u,v),a:b) \in \Lambda^2 \times \mathbb{P}^1_{\Lambda}: b(u-u_0) = a(v-v_0)\}. $$
	The blowing-down map $\pi$ to $\Lambda^2$ is given by forgetting the component $\mathbb{P}^1_{\Lambda}$.  
	Given a subset $S \subset \Lambda^2$ that contains $(u_0,v_0) \in U + \Lambda_+^2 \subset S$ for some open subset $U \subset \C^2$,  the blowing up of $S$ at $(u_0,v_0)$ is defined as $\pi^{-1}(S)$.
	
	We also have the toric resolution of the orbifold singularity at the origin
	\begin{equation} \label{eq:A_{n-1}}
		A_{n-1} := \{(u,v,\tilde{z}) \in \Lambda^3: uv = \tilde{z}^n\}
	\end{equation}
	over $\Lambda$, which is given as a toric surface glued from charts $\Lambda_{v_i,u_i}^2$ by the transitions $v_i = u_{i+1}^{-1}$, $u_i = v_{i+1}u_{i+1}^{2}$ for $i=1,\ldots,n$ and $n>1$.  (Note that when $n=1$, it is simply $\Lambda^2$.) The blowing down map $\pi$ is given by $u_1 = u, v_n = v, v_iu_i = \tilde{z}$ for all $i$.   We also have the map $f:A_{n-1}\to \Lambda^2$ forgetting the $\tilde{z}$-coordinate.  
	For a subset $S \subset \Lambda^2$ with $U+\Lambda_+^2 \subset S$ for some open $U\subset \C^2$ containing $0$, the blowing up of $f^{-1}(S) \subset A_{n-1}$ at the origin is defined as $\pi^{-1}(f^{-1}(S))$.
\end{definition}

\begin{remark}
	We have used the notation $\tilde{z}$ for a coordinate in the above definition, to distinguish from the previous holonomy variable $z$ for the torus.  They are related by $\tilde{z}=1+z$.
\end{remark}

Then we have:
\begin{corollary} \label{cor:MisA_n}
	The space $\check{X}_\Lambda$ glued by the unobstructed deformation spaces of the immersed Lagrangians $L_0,\ldots,L_{n+1} \subset S^\circ$, where the transition is taken as the solutions to the isomorphism equations for $(\alpha_i,\beta_i) \in \CF^0(L_i,L_{i+1})\times \CF^0(L_{i+1},L_{i})$, is equal to the resolution of $A_n \cap (\Lambda_0^2 \times (1+\Lambda_0^\times))$ at $0$ when $n>0$, where $A_n$ is given by Definition \ref{def:blowup}.  When $n=0$, it is equal to $A_0 \cap (\Lambda_0^2 \times (1+\Lambda_0^\times))$.
\end{corollary}

\begin{proof}
	Note that $(1+\Lambda_0^\times)^{\frac{1}{n}} = 1+\Lambda_0^\times$ and $1+\Lambda_0^\times = \bigcup_{\lambda\in \C-\{1\}} (\lambda + \Lambda_+)$.  Consider an element in $1+\Lambda_0^\times$.  Its $n$-th power lies in $\lambda^n + \Lambda_+$ where $\lambda^n\not=1$.  Thus $(1+\Lambda_0^\times)^{\frac{1}{n}} \supset 1+\Lambda_0^\times$.  Conversely, consider the $n$-th root of an element in $\lambda + \Lambda_+$ for $\lambda\not=1$.
	If $\lambda=0$, its $n$-th root still belongs to $\Lambda_+ \subset 1+\Lambda_0^\times$; if $\lambda\not=0$, its $n$-th root is of the form $\lambda^{1/n} + \Lambda_+$ where $\lambda^{1/n} \not=1$ (since $\lambda\not=1$).   Thus $(1+\Lambda_0^\times)^{\frac{1}{n}} \subset 1+\Lambda_0^\times$.
	
	Then $A_n \cap (\Lambda_0^2 \times (1+\Lambda_0^\times)) = f^{-1}(S)$, where $S=\{(u,v)\in \Lambda_0^2: uv \in 1+\Lambda_0^\times\} \subset \Lambda^2$.  It makes sense to talk about the blowing up of $ f^{-1}(S)$ by using the above definition.
	
	By Theorem \ref{thm:iso-An}, the gluing equations for the blowup in Definition \ref{def:blowup} are satisfied.  What remains to show is that the preimage $\pi^{-1}(\Lambda_0^2 \times (1+\Lambda_0^\times))$ (where $\pi$ is the blowing down map) is equal to the union of the charts $(\Lambda_0^\times)^2$ of $L_0$ and $\{(u_i,v_i)\in \Lambda_0^2: \val(u_iv_i) > 0 \}$ of $L_i$ for $i=1,\ldots,n+1$.  
	
	First, we show that the union of charts is a subset of $\pi^{-1}(\Lambda_0^2 \times (1+\Lambda_0^\times))$.
	It is easy to see that for $(u,v,\tilde{z}) \in \pi((\Lambda_0^\times)^2_{u_1,\tilde{z}-1})$ (where $(\Lambda_0^\times)^2_{u_1,\tilde{z}-1}$ is the chart for $L_0$), $u=u_1 \in \Lambda_0^\times \subset \Lambda_0$, $\tilde{z} \in 1 + \Lambda_0^\times$, and $v = u^{-1}\tilde{z}^n \in \Lambda_0$.  Also for $(u,v,\tilde{z}) \in \pi(\{(u_i,v_i)\in \Lambda_0^2: \val(u_iv_i) > 0 \})$, $\tilde{z} = u_iv_i \in \Lambda_+ \subset 1 + \Lambda_0^\times$; moreover, since both $u=u_1$ and $v=v_{n+1}$ are of the form $u_i^rv_i^s$ for $r,s\in \Z_{\geq 0}$, we have $u,v \in \Lambda_0$.  Thus we see that the union of charts is a subset of $\pi^{-1}(\Lambda_0^2 \times (1+\Lambda_0^\times))$.
	
	To show the converse, first we prove the statement for $n=0$.  Namely, the union of the chart $(\Lambda_0^\times)^2_{u_1,\tilde{z}-1}$ of $L_0$ with the chart $\{(u_1,v_1)\in \Lambda_0^2: \val(u_1v_1) > 0 \}$ of  $L_1$ is equal to $\{(u_1,v_1,\tilde{z}) \in \Lambda_0^2 \times (1+\Lambda_0^\times): u_1v_1 = \tilde{z} \}$.  The subset relation is known from the previous paragraph.  Conversely, $\val(u_1v_1)$ is either positive or zero.  If it is positive, the point belongs to the chart of $L_1$.  If it is zero, then $u_1 \in \Lambda_0^\times$, and hence the point belongs to the chart of $L_0$.
	
	Note that the above statement is symmetric with respect to $u_1$ and $v_1$.  This means the union of $(\Lambda_0^\times)^2_{v_1,\tilde{z}-1}$ with $\{(u_1,v_1)\in \Lambda_0^2: \val(u_1v_1) > 0 \}$ is equal to $\{(u_1,v_1,\tilde{z}) \in \Lambda_0^2 \times (1+\Lambda_0^\times): u_1v_1 = \tilde{z} \}$.  Note that by $u_2 = v_1^{-1}$, we have $(\Lambda_0^\times)^2_{v_1,\tilde{z}-1} = (\Lambda_0^\times)^2_{u_2,\tilde{z}-1}$.  
	Proceeding in the same way, the union of $(\Lambda_0^\times)^2_{u_2,\tilde{z}-1}$ and $\{(u_2,v_2)\in \Lambda_0^2: \val(u_2v_2) > 0 \}$ gives $\{(u_2,v_2,\tilde{z}) \in \Lambda_0^2 \times (1+\Lambda_0^\times): u_2v_2 = \tilde{z}\}$.  Inductively, we conclude that the union of the charts of $L_i$ for $i=0,\ldots,n+1$ is equal to the union of $\{(u_j,v_j,\tilde{z}) \in \Lambda_0^2 \times (1+\Lambda_0^\times): u_jv_j = \tilde{z}\}$ for $j=1,\ldots,n+1$.
	
	Now we are ready to show the converse for general $n$.   Given $(u,v,\tilde{z}) \in \pi^{-1}(\Lambda_0^2 \times (1+\Lambda_0^\times))$, we need to show that it belongs to the union of $\{(u_j,v_j,\tilde{z}) \in \Lambda_0^2 \times (1+\Lambda_0^\times): u_jv_j = \tilde{z}\}$ for $j=1,\ldots,n+1$.  We already know that $u_1 = u \in \Lambda_0$.  If $v_1 \in \Lambda_0$, we are done.  Otherwise, $u_2 = v_1^{-1} \in \Lambda_0$.  If $v_2 \in \Lambda_0$, then we are done.  Inductively, either we get that the point belongs to $\{(u_j,v_j,\tilde{z}) \in \Lambda_0^2 \times (1+\Lambda_0^\times): u_jv_j = \tilde{z}\}$ for some $j=1,\ldots,n$, or we get $u_{n+1} \in \Lambda_0$.  Since $v_{n+1} = v \in \Lambda_0$, the point belongs to $\{(u_{n+1},v_{n+1},\tilde{z}) \in \Lambda_0^2 \times (1+\Lambda_0^\times): u_{n+1}v_{n+1} = \tilde{z}\}$ in this case.
\end{proof}

In the above, the local charts are $\{(u,v,\tilde{z}) \in \Lambda_0^2 \times (1+\Lambda_0^\times): uv = \tilde{z}\}$.  They turn out to have a very nice relation with an open subset over $\C$:

\begin{proposition} \label{prop:C^2}
	$\{(u,v,\tilde{z}) \in \Lambda_0^2 \times (1+\Lambda_0^\times): uv = \tilde{z}\}=(\C^2 - \{uv=1\}) + \Lambda_+^2$.
\end{proposition}

\begin{proof}
	The proof is by stratifying both sides into two pieces:  the left hand side is equal to
	$$ \{\val uv > 0\} \sqcup \{\val u = \val v = 0, uv \in (\C-\{0,1\})+\Lambda_+ \}$$
	and the right hand side is equal to
	$$ \left(\{uv=0\} + \Lambda_+^2\right) \sqcup \left(\{uv\not=0,1\} + \Lambda_+^2\right).$$
	To verify that $\{\val uv > 0\}=\left(\{uv=0\} + \Lambda_+^2\right)$, we can further stratify to three pieces 
	$$
	\{\val u \textrm{ and } \val v > 0 \} \sqcup \{\val u > 0 \textrm{ and } \val v = 0 \} \sqcup \{\val v > 0 \textrm{ and } \val u = 0 \}
	$$ 
	and 
	$$
	\left(\{u=v=0\} + \Lambda_+^2\right) \sqcup 	\left(\{u=0 \textrm{ and }v\in \C^\times \} + \Lambda_+^2\right) \sqcup 	\left(\{v=0 \textrm{ and }u\in \C^\times \} + \Lambda_+^2\right)
	$$ respectively.  Then it is easy to see that they are equal to each other.  
	
	To verify that $\{\val u = \val v = 0, uv \in (\C-\{0,1\})+\Lambda_+ \}=\left(\{uv\not=0,1\} + \Lambda_+^2\right)$, one can check that both sides consist of elements $(u,v)$ of the form $u=u_0+u_+, v=v_0+v_+$, where $u_0,v_0 \in \C^\times$ with $u_0v_0\not=1$, and $u_+,v_+\in \Lambda_+$.
\end{proof}

\begin{remark}
	The proof shows that the glued mirror from $(L_i,b_i)$ for $i=0,\ldots,n+1$ in the smoothing $S^{(n)}$ is equal to the union of these charts $\left((\C^2 - \{uv=1\}) + \Lambda_+^2\right) \subset \Lambda^2$.  Taking the intersection of each chart with $\C^2 \subset \Lambda^2$, we get the complex surface $\{(u_j,v_j,\tilde{z}) \in \C^2 \times (\{1\}+\C^\times): u_jv_j = \tilde{z}\}$, which is equal to the usual $\C$-valued $A_n$-resolution minus the anti-canonical divisor with local description $u_jv_j=1$.
	
	The $\C$-valued mirror is glued from the Clifford torus $L_0$ with flat $\C^\times$-connections (for $\tilde{z} \in 1+\C^\times$ and $u_1 \in \C^\times$), and the immersed sphere $L_i$ (for $i=1,\ldots,n$) with boundary deformations $u_i U_i + v_i V_i$ with $u_iv_i=0$.
\end{remark}

As a result, resolutions of an $A_n$-singularity are mirror to smoothings of the $A_n$-singularity.  By the construction in \cite{CHL-glue}, there exists an $A_\infty$ functor from the Fukaya category of $S^\circ$ to the category of twisted complexes over $\check{X}_\Lambda$.  In this sense, the local $A_n$-singularity is self-mirror.

The relationship between $\Lambda_0$ and $\C$ can be formulated more systematically by defining the following.

\begin{definition} \label{def:NovExt}
	Given a complex manifold $M$, its extension over the Novikov ring, denoted by $M + \Lambda_+^n$, is defined as the union of charts $U+\Lambda_+^n$ glued by the same transition functions of $M$, where $U$ are charts of $M$.  For an analytic subset $Z \subset M$, its extension over the Novikov ring, denoted by $Z+\Lambda_+^n$, is the union of $Z_p+\Lambda_+^n$, where $Z_p \subset U_p$ is the neighborhood of $p$ in $Z$ lying in a local chart $U_p \subset \C^n$ of $M$ and is given as the zero locus of a (finite) set of complex analytic functions.
\end{definition}

The above is well-defined since for an analytic function $f$, $f(z) \in f(z_0) + \Lambda_+$ where $z\in z_0 + \Lambda_+$ and $z_0\in \C$.

Using the above definition, Proposition \ref{prop:C^2}, and Corollary \ref{cor:MisA_n}, we get the following.

\begin{corollary}
	The resolution of $A_n \cap (\Lambda_0^2 \times (1+\Lambda_0^\times))$ is equal to $\hat{A}_n + \Lambda_+^2$, where $\hat{A}_n$ is the resolution of $\{(u,v,\tilde{z})\in \C^2 \times (\C-\{1\}): uv=\tilde{z}^n\}$ over $\C$.
\end{corollary}

Similarly,

\begin{proposition}
	For an open subset $U \subset \C^2$ with $(u_0,v_0) \in U$, the blowing up of $U + \Lambda_+^2$ (over $\Lambda$) at $(u_0,v_0)$ is equal to $\hat{U} + \Lambda_+^2$ where $\hat{U}$ denotes the blowing up of $U$ at $(u_0,v_0)$ (over $\C$).
\end{proposition}

For the purpose of the next section, it is useful to have another description of the above total space of an $A_n$-resolution in terms of the usual repeated blowing-up at a point.

\begin{proposition} \label{prop:mult-blowup}
	The resolution $\check{X}_\Lambda$ of $A_n \cap (\Lambda_0^2 \times (1+\Lambda_0^\times))$ (where the resolution does nothing for $n=0$) is equal to $M'$ constructed as follows.  Take the blowing-up of $\Lambda_0 \times (1+\Lambda_0^\times)=(\C_{u,\tilde{z}}^2 - \{\tilde{z}=1\}) + \Lambda_+^2$ at $(u,\tilde{z})\in (0,0)$, and repeatedly take the blowing-up again at a point in the new exceptional curve $n$ times, so that we have $(n+1)$ exceptional curves in total.  $M'$ is defined to be the complement of $(Z' + \Lambda_+^2)$, where $Z'$ (over $\C$) is
	the strict transform of the $\tilde{z}$-axis $\{u=0\} \subset \C_{u,\tilde{z}}^2 - \{\tilde{z}=1\}$.  
\end{proposition}


\begin{remark}
	We have a global analytic function $\tilde{z}: \check{X}_\Lambda \to \Lambda - \{1\}$.
	By taking $\tilde{z}=0$ (that is, the holonomy variable $z=-1+\tilde{z}=-1$) and gluing the valuation images of $(\val(u_i),\val(v_i))$ in all the charts,
	we get the `skeleton' as shown in Figure \ref{fig:An-skeleton}.
\end{remark}

\begin{figure}[h]
	\begin{center}
		\includegraphics[scale=0.3]{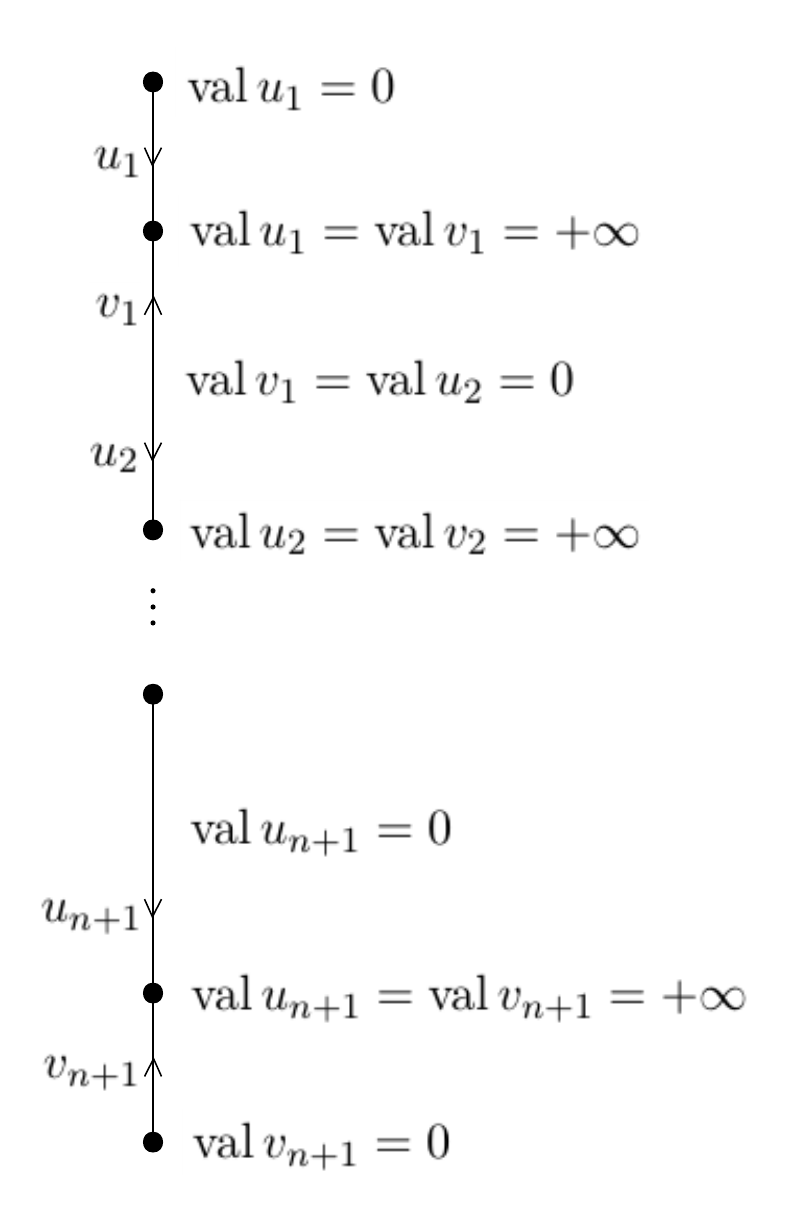}
		\caption{A skeleton of $\check{X}$ obtained by taking valuation image.}\label{fig:An-skeleton}
	\end{center}
\end{figure}

\subsection{Application to Del Pezzo surfaces of degree higher than two}

Let's consider toric Gorenstein Fano surfaces.  Their smoothings give del Pezzo surfaces of degree $\geq 3$.  $A_n$ singularities occur at the toric fixed points.  The gluing method using Fukaya isomorphisms from the last subsection can be used to construct their Landau--Ginzburg mirrors.
The mirror pairs we construct in this way are summarized by Figure \ref{fig:delPezzo}.

Take a toric Gorenstein Fano surface $M_0$ equipped with a toric K\"ahler form, such that its moment-map polytope is integral and contains the origin as the unique interior lattice point.  Let $\bL_0$ be the moment map torus fibre over the origin.  Then $M_0$ is a monotone symplectic manifold, and $\bL_0$ is a monotone Lagrangian torus.

Let's denote by $F$ the set of toric fixed points.
Each $p\in F$ corresponds to a maximal cone of the fan, whose dual cone gives a toric chart containing the fixed point.  The toric chart is given by 
$$S_0 := \{(X,Y,Z)\in \C^3: XY=Z^{n+1} \}$$ 
when the toric fixed point $0$ is an $A_n$ singularity (and $n=0$ if it is a smooth point).

If necessary (when $n\geq 1$), we take a smoothing family $S_\epsilon := \{(X,Y,Z)\in \C^3: XY = \prod_{i=1}^{n+1} (Z- \epsilon_i) \}$ as in the last subsection.  This is a family over $\C^{n+1}$, where $\epsilon$ lives.  Such a smoothing is $\bS^1$-equivariant, where $\bS^1$ acts by $(X,Y,Z)\mapsto (e^{i\theta}X, e^{-i\theta}Y,Z)$.  This can be understood as an $\bS^1$-equivariant symplectic fibration over $\C^{n+1}$ (where $S_\epsilon$ is equipped with the restricted standard symplectic form of $\C^3$).  By symplectic parallel transport, $U_\epsilon - \bigcup_{i=1}^n \bS_i$ is symplectomorphic to $U_0 - \{0\}$, where $\bS_i$ are the vanishing spheres whose images in the reduced base are line segments in the real line joining $Z=\epsilon_i$ and $Z=\epsilon_{i+1}$; $U_\epsilon$ and $U_0$ are certain $\bS^1$-invariant neighborhoods of $\bigcup_{i=1}^n \bS_i \subset S_\epsilon$ and $0\in S_0$ respectively.  Moreover, by the Moser argument, $(S_0,\omega_{M_0})-\{0\}$ is symplectomorphic to an open subset of $(S_0,\omega_{\C^3})-\{0\}$.  Combining these, we have an $\bS^1$-equivariant symplectomorphism between $(U_\epsilon - \bigcup_{i=1}^n \bS_i,\omega_{\C^3})$ and $(S_0-\{0\},\omega_{M_0})$ for some neighborhoods $U_\epsilon \subset S_\epsilon$.

By gluing the patches $U_\epsilon$ and $M_0-\{0\}$ using the above symplectomorphism, we obtain a (partial) symplectic smoothing of $M_0$.  Repeating the surgery at all singular toric fixed points, we obtain a symplectic del Pezzo surface which is denoted by $X$.

Similarly to the $\bS^1$-equivariance, we also have $\Z_2$-equivariance for the anti-symplectic involution on $S_\epsilon$: $(X,Y,Z) \mapsto (\bar{Y},\bar{X},\bar{Z})$, and the above symplectomorphisms are made to be $\Z_2$-equivariant.  

In particular, the monotone Lagrangian torus $\bL_0 \subset (S_0,\omega_{M_0})-\{0\}$ is sent via the symplectomorphism to a corresponding Lagrangian torus, which is denoted by $L_0$, in $U_\epsilon - \bigcup_{i=1}^n \bS_i \subset X$ .  $L_0$ is invariant under $\bS^1$ and the anti-symplectic involution.  This matches the setting for $L_0$ in the last subsection.  Thus for the $i$-th toric fixed point (ordered counterclockwise around the moment-map polytope), which is an $A_{n_i}$-singularity for $n_i\geq 0$, the construction gives Lagrangian immersed spheres $\bL_{i,k}$ for $k=1,\ldots,n_i+1$ in $X$.

\begin{figure}[h]
	\begin{center}
		\includegraphics[scale=0.245]{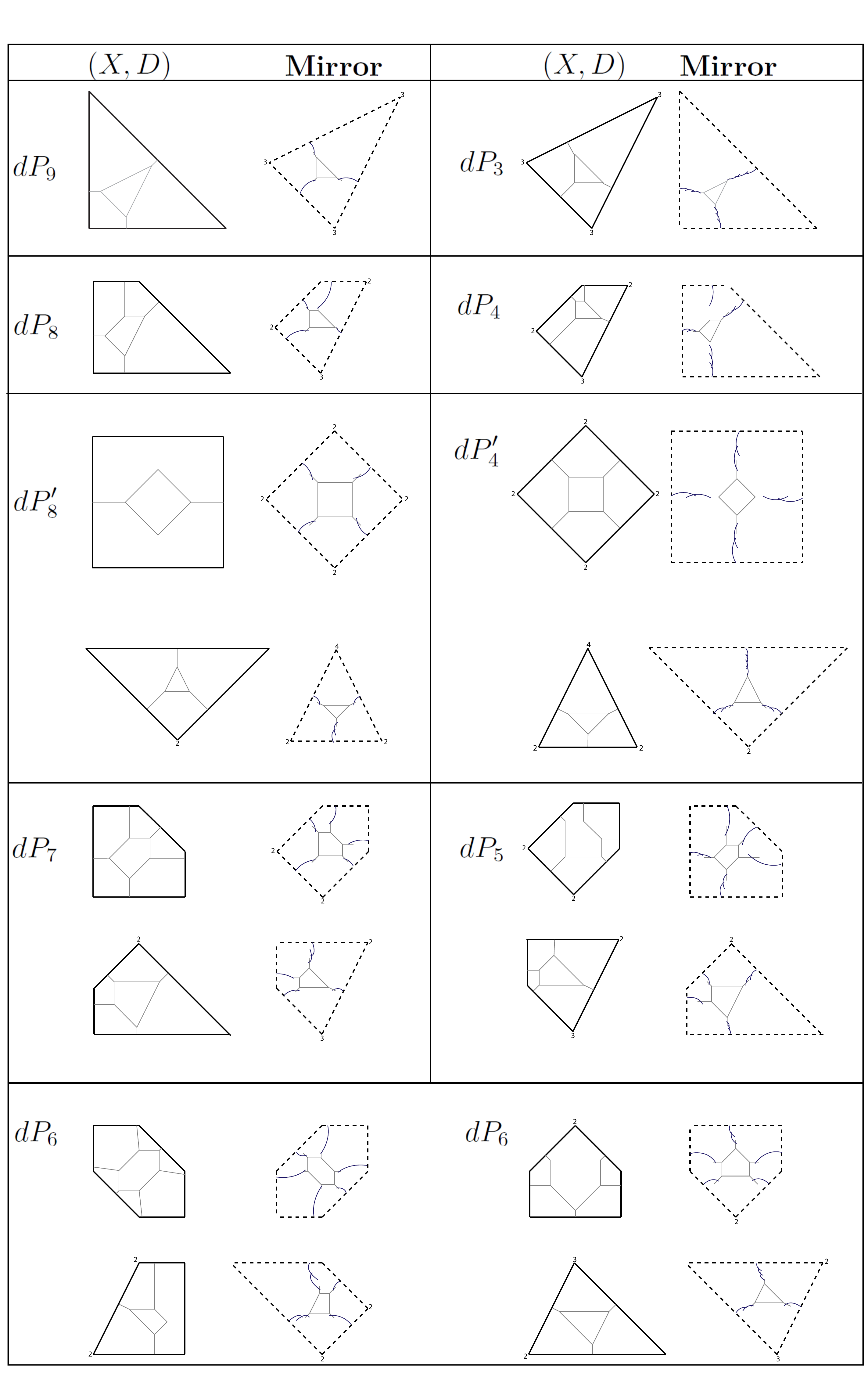}
		\caption{Del Pezzo surfaces that come from smoothings of toric Gorenstein surfaces and their mirrors.}\label{fig:delPezzo}
	\end{center}
\end{figure}

\begin{lemma}
	$X$ and $\bL_0 \subset X$ are monotone.
\end{lemma}

\begin{proof}
	The curve classes in $\pi_2(M_0)$ have symplectic areas and Maslov indices unchanged under smoothing.  Moreover, the matching spheres in the smoothing have both Maslov index and area being zero.  Thus $X$ is still monotone.
	
	Similarly, the basic disc classes in $\pi_2(M_0,\bL_0)$ have areas and Maslov indices unchanged after smoothing, and hence $\bL_0$ remains monotone.
\end{proof}

\begin{lemma}
	$\bL_0$ and the immersed Lagrangians $\bL_{i,j}$ have minimal Maslov index two.  In particular $m_0^b$ is proportional to the unit.
\end{lemma}
\begin{proof}
	We have a symplectomorphism $M_0 - \{\textrm{toric fixed points}\} \cong X - \{\textrm{vanishing spheres}\}$.  In  $M_0 - \{\textrm{toric fixed points}\}$, we already know that $\bL_0$ has minimal Maslov index two, since it is special with respect to the toric meromorphic volume form.  In the smoothing of each toric chart around a singular fixed point, the non-constant holomorphic discs bounded by $\bL_0$ have Maslov indices greater than zero.  Since any disc class bounded by $\bL_0$ is a linear combination of disc classes in the charts, $\bL_0$ has minimal Maslov index two.
	
	Holomorphic polygon classes bounded by the immersed Lagrangian $\bL_{i,j}$ is a sum of non-constant holomorphic disc classes corresponding to those of $\bL_0$ with constant polygon classes that has Maslov index zero.  Thus $\bL_{i,j}$ also has minimal Maslov index two.
	
	$m_0^b$ has degree two.  Since the minimal Maslov index is two, the degree of an output in $m_0^b$ cannot be bigger than zero, which can only be the unit.
\end{proof}

By the isomorphisms between the Lagrangian objects $\bL_0$ and $\bL_{i,j}$ for $p \in F$ and $j=1,\ldots,n_i+1$ given in the last subsection, we obtain a manifold glued from the formal deformation spaces of the Lagrangians.  We call it the Floer-theoretical mirror space.

Let $\check{\Delta}$ be the polar dual polytope of $\Delta$, and $X_{\check{\Delta}}$ be the corresponding toric variety over $\C$.  
For each toric chart corresponding to a fixed point of $X_{\check{\Delta}}$, we have the points $(-1,0)$ and $(0,-1)$ lying in the toric divisors, which are invariant under toric change of coordinates.  We call these special points in $X_{\check{\Delta}}$.  It is easy to see that there is a natural one-to-one correspondence between the collection of special points in $X_{\check{\Delta}}$ and the toric fixed points of $X_{\Delta}$.

We consider the toric variety over $\Lambda$, $X_{\check{\Delta}} + \Lambda_+^2$ (see Definition \ref{def:NovExt}).

\begin{remark}
	A toric variety over $\Lambda$ can also be defined as a GIT quotient of $\Lambda^n$ over $\Lambda^\times$ (like $\mathbb{P}^1_{\Lambda}$ in Definition \ref{def:blowup}).  When the corresponding fan picture is complete, it agrees with the extension of the toric variety over the Novikov ring.
\end{remark}

To make notations precise, recall that the toric fixed points of $X_{\Delta}$ are labeled by $i$.  Toric fixed points of $X_{\Delta}$ correspond to toric prime divisors of $X_{\check{\Delta}}$.  We denote the toric coordinates on these toric divisors by $z_i$.   The special points in the divisors are $z_i=-1$. 

\begin{theorem} \label{thm:mir-del-Pezzo}
	The Floer-theoretical mirror space is equal to $\check{X}_\Lambda = \check{X}+\Lambda_+^2$, where $\check{X}$ is given as follows.  First, we take the (multiple) blowing up of the toric variety $X_{\check{\Delta}}$ at every special point $z_i=-1$ in the $i$-th toric divisor for $(n_i+1)$-times (where $n_i+1$ is the multiplicity of  the $i$-th toric fixed point of $X_\Delta$).  Then we define $\check{X}$ to be the complement of the strict transform of all the toric divisors of $X_{\check{\Delta}}$.
\end{theorem}

\begin{proof}
	By Corollary \ref{cor:MisA_n} and Proposition \ref{prop:mult-blowup}, fixing $i$, the deformation spaces of $\bL_0$ and $\bL_{i,j}$ for $j=1,\ldots,n_i+1$ glue up to the Novikov extension of a multiple blowing up of $\C^\times_{z_i} \times \C_{u_i}$ at $(u_i,z_i)=(0,-1)$, with the strict transform of the $z_i$-axis removed, where $z_i$ is the holonomy variable corresponding to the vanishing cycles of $\bL_{i,j}$.  (In the notation of the previous section, $z_i=\tilde{z}_i-1$.)  The direction of $z_i$ determines the toric compactification $\C^\times_{z_i} \times \C_{u_i}$ of $(\C^\times)^2$ (which corresponds to adding a ray dual to $z_i$ in the fan picture).  These holonomy directions are obtained by taking the orthogonal complement of $\xi_i + \eta_i$, where $\xi_i$ and $\eta_i$ are the edge directions of the $i$-th corner of the moment-map polytope of $M_0$.  From this, we see that the charts $\C^\times_{z_i} \times \C_{u_i}$ glue to the toric variety $X_{\check{\Delta}} - \{\textrm{toric fixed points of } X_{\check{\Delta}}\}$.  It follows that the entire space glued from $\bL_0$ and $\bL_{i,j}$ for all $i,j$ is equal to $\check{X}+\Lambda_+^2$.
\end{proof}

\begin{remark}
	$\check{X}_\Lambda$ can also be constructed by taking the (multiple) blowing up of $X_{\check{\Delta}} + \Lambda_+^2$ using Definition \ref{def:blowup}, and then taking the complement of $\hat{D} + \Lambda_+^2$ for all the toric divisors $D \subset X_{\check{\Delta}}$ (where $\hat{D}$ denotes the strict transform).
	
	The coordinates $z_i$ form analytic maps $\check{X}_\Lambda \to \mathbb{P}^1_\Lambda$ (since the strict transform of the toric divisors, and in particular the toric fixed points, have been removed).  The union of the valuation images of $(u_{i,j},v_{i,j})$ on $z_i=-1$ form a skeleton.  See Figure \ref{fig:dP-skeleton} for an example.
\end{remark}

\begin{figure}[h]
	\begin{center}
		\includegraphics[scale=0.5]{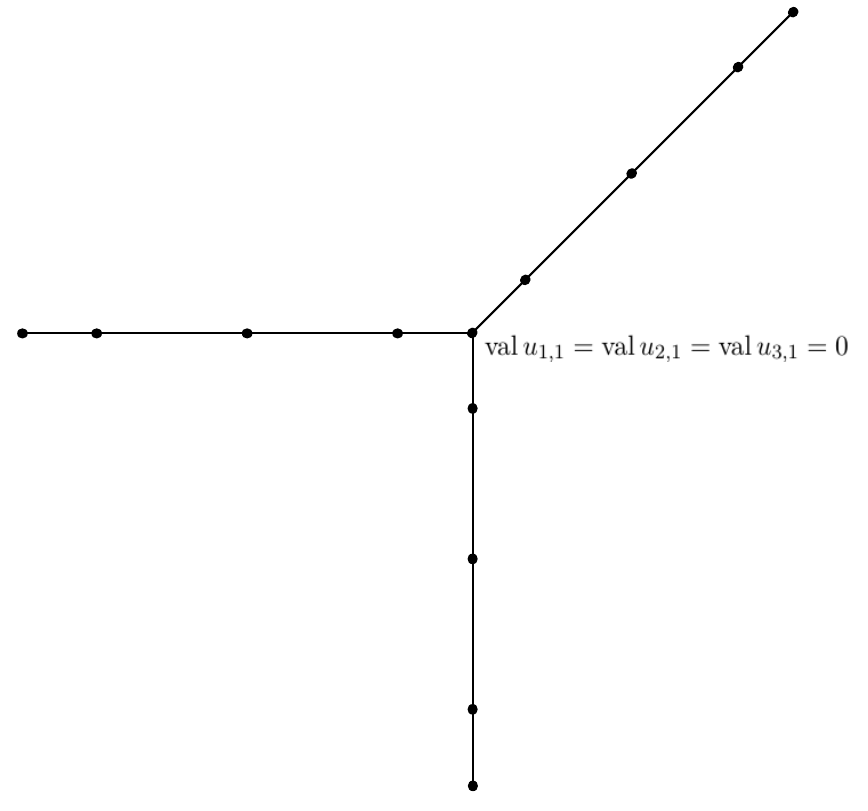}
		\caption{A skeleton of mirror of $dP_3 - E$ formed by the valuation image.}\label{fig:dP-skeleton}
	\end{center}
\end{figure}

By Theorem \ref{thm:mir-del-Pezzo}, the space $\widehat{X_{\check{\Delta}}} - D_{\check{\Delta}}$ is obtained by gluing according to isomorphisms of Lagrangians.  Thus the disc potential $W_{L_0}$ of $L_0$ must extend as a well-defined function to the whole $\widehat{X_{\check{\Delta}}} - D_{\check{\Delta}}$.

\begin{corollary}
	The disc potential $W_{\bL_0}$ extends as a well-defined function over $\widehat{X_{\check{\Delta}}} - D_{\check{\Delta}}$.
\end{corollary}

Let $v_1,\ldots,v_p$ be the primitive inward normal vectors to the facets of the moment-map polytope of the toric variety $M_0$, which are ordered counterclockwisely.  (Thus $(v_i,v_{i+1})$ agrees with the standard orientation of the plane.)  The multiplicities at toric fixed points are equal to $n_i = \det (v_i,v_{i+1})$.  The toric fixed points are $A_{n_i-1}$-singularities.  Then $u_{i,j} = (jv_i + (n_i-j)v_{i+1})/n_i$ for $j\in [2,n_i-1] \cap \Z$ are integral vectors.

\begin{proposition}
	The disc potential $W_{\bL_0}$ is equal to
	$$ W_{\bL_0} = \sum_i z^{v_i} + \sum_i \sum_{j=2}^{n_i-1} {n_i \choose j} z^{u_{i,j}}.$$
\end{proposition}
\begin{proof}
	The disc potential for $\bL_0$ in the $A_{n_i}$-smoothing of the $i$-th toric chart corresponding to $(v_i,v_{i+1})$ is equal to
	$$ z^{v_i} + z^{v_{i+1}}+ \sum_{j=2}^{n_i-1} {n_i \choose j} z^{u_{i,j}}.  $$
	
	$$ z^{v_i}(1 + z^{u_{i,2}-v_i})^{n_i}$$
	
	Since $\bL_0$ is monotone, the disc potential is invariant under deformation of complex structures \cite{EP}.
	We show below that a Maslov-two holomorphic disc bounded by $\bL_0$ must be entirely contained in one of the charts.  Then it follows that the disc potential of $\bL_0$ in $X$ is a sum of the potentials in the various charts.
	
	Let's choose an embedding of the open torus orbit $(\C^\times)^2$ of the toric degeneration to $X$, whose complement $C$ is a union of Lagrangian vanishing spheres together with chosen branches of degenerate conics in the charts.  In dimension two, $C$ represents the anti-canonical class.  This can be easily seen by deforming $X$ to a resolution of singularities in which the vanishing spheres become holomorphic.  It follows that a holomorphic disc of Maslov-index two intersects the anti-canonical class $C$ once.   Thus it is contained in a chart.
\end{proof}

Explicit expressions of $W_{L_0}$ are given in Table \ref{tab:W0}.  Comparing to Figure \ref{fig:delPezzo}, for each $dP_n$, we only write down the disc potential for the monotone torus associated to the toric variety shown as the first figure, since the others can be obtained via wall-crossing.  They give different realizations of the \emph{same mirror space} as the complement of a (non-toric) blowing-up of a toric variety. 

\begin{ex}
	Consider $dP_4'$ obtained as smoothings of the two toric varieties shown in Figure \ref{fig:delPezzo}.  The disc potential associated to the first one is equal to 
	$$W_0 =xy+xy^{-1}+x^{-1}y+x^{-1}y^{-1}+2(x+y+x^{-1}+y^{-1}).$$
	It is related to the disc potential associated to the second one by wall-crossing
	$$ y' = y(x+x^{-1}+2);\,\,x' = x.$$
	The potential becomes
	\begin{align*}
		&(x+x^{-1}+2)y^{-1}+2x^{-1}+2x+y(x+x^{-1}+2)\\
		=&(x'+(x')^{-1}+2)^2 (y')^{-1} + 2x^{-1}+2x+y'\\ =& (x')^2(y')^{-1}+(x')^{-2}(y')^{-1}+y'+2x'+2(x')^{-1}+4\left(x'(y')^{-1}+(x')^{-1}(y')^{-1}\right)+6(y')^{-1}.
	\end{align*}
	The coordinate change between $y'$ and $y$ is composed of crossing two parallel walls in the $A_1$ configuration.  Namely, $y' = yx^{-1}(1+x)^2 = y(x+x^{-1}+2)$.
\end{ex}

\begin{figure}[h]
	\begin{center}
		\includegraphics[scale=0.5]{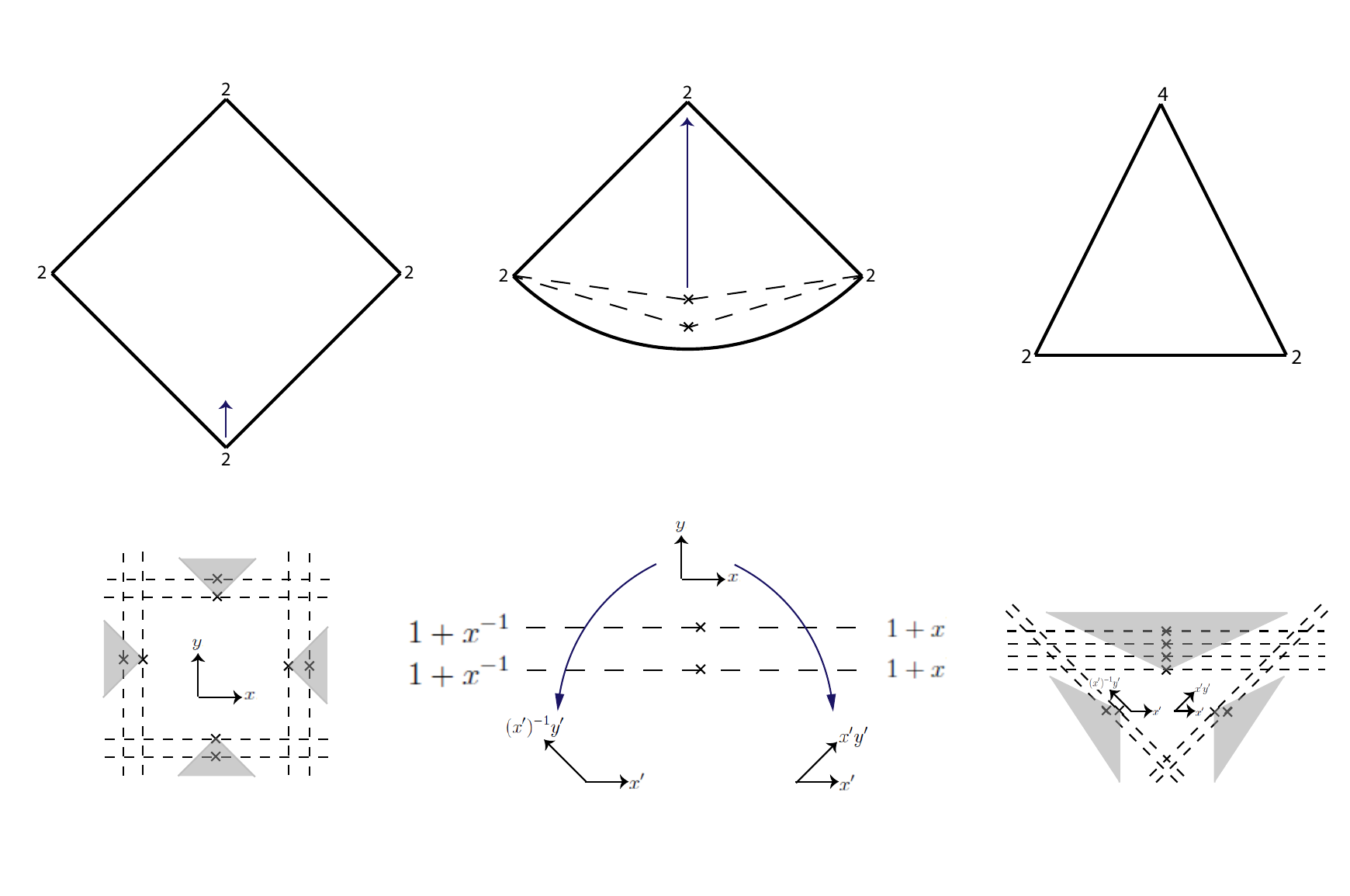}
		\caption{Wall crossing for $dP_4'$.}\label{fig:wcdP}
	\end{center}
\end{figure}

\begin{table}[h]
	\centering
	\begin{tabular}{| l | l | } 
		\hline
		$dP_9$ & $W_0 = x + y + (xy)^{-1}$ \\
		\hline
		$dP_8$ & $W_0 = x + y + y^{-1} + x^{-1}y^{-1}$ \\
		\hline
		$dP_8'$ & $W_0 = x + y + x^{-1} + y^{-1}$ \\
		\hline
		$dP_7$ & $W_0 = x + y + x^{-1} + y^{-1} + x^{-1}y^{-1}$  \\
		\hline
		$dP_6$ & $W_0 = x + y + xy + x^{-1} + y^{-1} + x^{-1}y^{-1}$ \\
		\hline
		$dP_5$ & $W_0 = xy + x^{-1} y + xy^{-1} + x^{-1} + y^{-1} + 2x + 2y$ \\
		\hline
		$dP_4$ & $W_0 =xy+xy^{-1}+y^{-1}+x^{-2}y+2x+2x^{-1}+3y+3x^{-1}y$ \\
		\hline	
		$dP_4'$ & $W_0 =xy+xy^{-1}+x^{-1}y+x^{-1}y^{-1}+2(x+y+x^{-1}+y^{-1})$ \\
		\hline
		$dP_3$ & $W_0 = xy + x^{-2}y + x y^{-2} + 3(x+xy^{-1}+y+x^{-1}y+x^{-1}+y^{-1})$\\
		\hline
	\end{tabular}
	
	\caption{The disc potentials of toric monotone tori in del Pezzo surfaces.}
	\label{tab:W0}
\end{table}

\appendix
\section[]{Construction of rational elliptic surfaces with an 
\texorpdfstring{\(I_{d}\)}{} fibre and trivial periods} \label{ext res}

In the appendix, by
reverse engineering the algorithm in proof of 
\cite{2015-Gross-Hacking-Keel-moduli-of-surfaces-with-an-anti-canonical-cycle}*{Proposition 1.3}, 
we construct the rational surface with an $I_d$ fibre 
and trivial period explicitly below. 

The del Pezzo surfaces \(\mathrm{dP}_{d}\) of degree \(d=6,7,8,9\) are toric varieties.
Applying Batyrev's toric mirror construction
and taking a further blow-up yield the desired rational elliptic surfaces.
We thus focus on the case \(1\le d\le 5\).

\subsection{$d=5$} Recall that the del Pezzo surface of degree seven \(\mathrm{dP}_{7}\)
is a toric surface whose fan 
is generated by $(-1,0), (0,-1), (1,0), (1,1), (0,1)$. 
Denote by $D_i$, $i=1,\ldots 5 $, the corresponding toric divisors. 
Then the del Pezzo surface of degree five $Y_5$ can be realized as the 
blow-up of \(\mathrm{dP}_{7}\) at the point ``\(-1\)'' 
on the toric divisors \(D_{1}\) and \(D_{2}\). 
Then $\check{Y}_e$ is the blow-up along the 
preimage of ``\(-1\)'' on 
each of $D_i$ in $Y_t$ and $\check{D}_e$ is the 
proper transform of $\sum_{i=1}^{5}D_i$. 

\subsection{\(d=4\)}
We start with the toric variety $\mathbf{P}^1\times \mathbf{P}^1$. Let
$D_i$, $i=1,\ldots 4$, be the irreducible toric boundaries. 
Let $Y'_4$ be the blow up of \(\mathbf{P}^{1}\times\mathbf{P}^{1}\)
at ``$-1$'' on each \(D_{i}\). Then $\check{Y}_e$ is the blow up 
of the intersection of the exceptional divisors with the proper transform of $D_i$. 
Then $\check{Y}$ is a blow-up of $\mathbf{P}^1\times \mathbf{P}^1$ at 
eight points and thus a rational elliptic surface. The $I_4$ fibre is the 
proper transform of $\sum_i D_i$. 
In this case the singular configuration is $I^*_1I_4I_1$. 

\subsection{$d=3$}
We begin with $\mathbf{P}^2$ and irreducible toric boundaries $D_i$, $i=1,2,3$. 
Let $Y'_3$ be the blow-up of $\mathbf{P}^2$ along ``$-1$'' on each toric boundary. 
Let $Y''_3$ be the blow-up of $Y'_3$ along the intersection of the exceptional divisors 
with the proper transform of $D_i$. Then $\check{Y}_e$ is the blow-up of $Y''_3$ along the intersections 
of the exceptional divisors of $Y''_3\rightarrow Y'_3$ 
and the proper transform of $D_i$. The boundary divisor is then the proper transform of $D_i$ to $\check{Y}_e$. Then $\check{Y}_e$ is the blow up of $\mathbf{P}^2$ at nine points and thus a rational elliptic surface. 
In this case, the singular configuration is $\mathrm{IV}^*\mathrm{I}_3\mathrm{I}_1$. 

\subsection{\(d=2\)}
Recall that the rays in the fan of $\mathbf{P}^1\times \mathbf{P}^1$ are
generated by $(1,0)$, $(0,1)$, $(-1,0)$, $(0,-1)$. 
Denote by $D_i$, $i=1,\ldots, 4$, the corresponding toric divisor. Let $Y'_2$ be obtained from
blowing up \(\mathbf{P}^{1}\times\mathbf{P}^{1}\) at ``$-1$'' on $D_2$ and $D_4$ respectively
and then contract the proper transform of $D_2$, $D_4$. 
Then $Y'_2$ is a minimal rational surface of second Betti number $2$. 
Since the proper transform of $\{-1\}\times \mathbf{P}^1\subseteq \mathbf{P}^1\times \mathbf{P}^1$ 
(which is homologous to $D_1$ or $D_3$) is a $(-2)$-curve, we have $Y'_2\cong \mathbb{F}_2$. 
Then $\check{Y}_2$ is 
the four-time iterated blow up at $(-1)$ of $D_1,D_3$. Then $\check{Y}_e$ is the blow up of $Y'_2$ at the points corresponding to $-1$ on the proper transform of $D_2,D_4$.  $\check{D}_e$ is the proper transform of $D_1,D_3$. Then
  $\check{Y}_2$ is a rational elliptic surface with singular configuration $\mathrm{III}^*\mathrm{I}_2\mathrm{I}_1$, where the component of the $\mathrm{III}^*$ fibre with multiplicity $4$ is the proper transform of $\mathbf{P}^1\times \{-1\}$ from $\mathbf{P}^1 \times \mathbf{P}^1$, which is homologous to $D_2,D_4$.  

\subsection{\(d=1\)}
Let $(x,t)$ and $(z,s)$ be two charts of $\mathcal{O}_{\mathbf{P}^1}(3)$ such that $x=1/z, t=x^3s$. Let $L$ be the compactification of the curve locally defined by $\{t+(x+1)^3=0\}$ in the Hirzebruch surface $\mathbb{F}_3\cong \mathbb{P}(\mathcal{O}\oplus \mathcal{O}_{\mathbf{P}^1}(3))$. Then $L$ intersect two toric boundary fibres at $-1$ and is tangent to the $(+3)$-toric boundary divisor at $-1$ with multiplicity $3$. Blowing up the intersection of $L$ with the two toric boundary fibres and blow down the two fibres and the original unique $-3$ curve leads to $\mathbf{P}^2$. The proper transform of the original $(+3)$-toric boundary divisor is a nodal cubic $C$ and is tangent to the proper transform of $L$ with multiplicity $3$ at $p$. Then $\check{Y}_1$ is the nine-times successive blow up at the point corresponding to $p$ on the proper transform of $C$. The proper transform of $L$ on $\check{Y}_1$ is the component of the $\mathrm{II}^*$ fibre with multiplicity $3$ and adjacent to the unique component of multiplicity $6$.

\section[]{Proof of Lemma \ref{initial ray b}}
\label{app-proof-lemma}
We retain the notation in \S\ref{sec:complex-affine-structure-of-slf}. We can compute
\begin{equation*}
\operatorname{Im}\left(\int_{\partial b(q)}\iota_{\partial/\partial q}\check{\Omega}\right)
<0~\mbox{for}~q\in (-4,0).
\end{equation*}
To prove the lemma, it suffices to show that there exists a constant \(\kappa>0\)
such that 
\begin{equation}
\label{eq:appendix-wts}
\operatorname{Im}\left(\int_{-\partial a(q)}\iota_{\mathrm{i}\partial/\partial q}\check{\Omega}\right)
\ge\frac{\kappa}{|q|},~\mbox{for all \(q\) large}.
\end{equation}
Let \(p_{1}\) and \(p_{2}\) be ramification points such that 
\(\operatorname{Re}p_{1}>0\), \(\operatorname{Im}p_{1}>0\) and 
\(\operatorname{Re}p_{2}>0\), \(\operatorname{Im}p_{2}<0\). 
(\(p_{1}\) and \(p_{2}\) are roots of \(t^2-(2+q)t+1=0\).)
We may write \(p_{1}=R\exp(\mathrm{i}\theta_{0})\) and
\(p_{2}=\epsilon \exp(-\mathrm{i}\theta_{0})\)
where \(R=|p_{1}|\) and \(\epsilon=|p_{2}|\). 
Note that \(R\), \(\epsilon\), and \(\theta_{0}\) indeed depend on \(q\)
but we shall drop out \(q\) in the notation for simplicity.
Let \(\gamma_{1,1}\) be a circular arc of radius \(R\)
connecting \(p_{1}\) and the positive real axis, 
\(\gamma_{1,2}\) be the line segment connecting \(\epsilon\) and \(R\) on 
the positive real axis, and \(\gamma_{1,3}\) be 
a circular arc of radius \(\epsilon\) connecting \(p_{2}\) and the 
positive real axis
(cf.~\textsc{Figure} \ref{fig:deform-the-path}). 
We can deform \(\gamma_{1}\) into \(\gamma_{1,1}\cup \gamma_{1,2}\cup\gamma_{1,3}\).
Note that 
\begin{equation*}
\sqrt{(t_{2}^{2}+1-qt_{2})^{2}-4t_{2}^{2}} = t_{2}\cdot
\sqrt{(t_{2}+t_{2}^{-1}-q)^{2}-4}.
\end{equation*}
on \(\gamma_{1,1}\cup \gamma_{1,2}\cup\gamma_{1,3}\). 
We shall compute 
\begin{equation*}
\int_{\gamma_{1,j}}\frac{\mathrm{d}t_{2}}{\sqrt{(t_{2}^{2}+1-qt_{2})^{2}-4t_{2}^{2}}}
=\int_{\gamma_{1,j}}\frac{1}{\sqrt{(t_{2}+t_{2}^{-1}-q)^{2}-4}}\frac{\mathrm{d}t_{2}}{t_{2}}.
\end{equation*}

\begin{figure}[!htbp]
\includegraphics[scale=1.1]{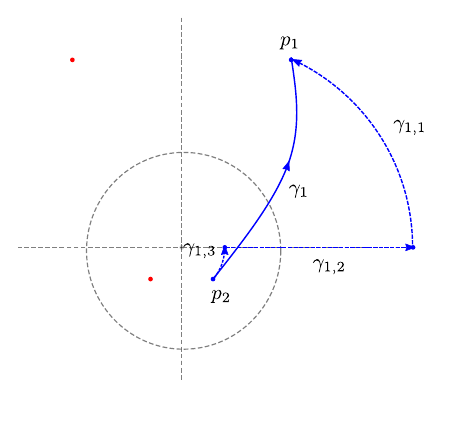}
\caption{\(p_{1}\) and \(p_{2}\) are ramification points.}
\label{fig:deform-the-path}
\end{figure}

\begin{lemma}
\label{lemma:countour-integrals}
For \(j=1,2,3\), we have
\begin{equation*}
\operatorname{Im}\left(
\int_{\gamma_{1,j}}\frac{1}{\sqrt{(t_{2}+t_{2}^{-1}-q)^{2}-4}}\frac{\mathrm{d}t_{2}}{t_{2}}\right)\ge 0.
\end{equation*}
\end{lemma}
\begin{proof}
We claim that 
\begin{equation}
\label{lem:claim}
\operatorname{Im}\left((t_{2}+t_{2}^{-1}-q)^{2}-4\right)\le 0,~\mbox{for}~t_{2}\in \gamma_{1,j}.
\end{equation}
The lemma immediately follows from the claim.

For \(j=2\), \eqref{lem:claim} is true since \(t_{2}\in\mathbb{R}_{+}\) 
and \(q\in\mathrm{i}\cdot\mathbb{R}_{+}\).
Let us consider the case \(j=1\). The case \(j=3\) is completely parallel.
For \(j=1\), substituting \(t_{2}=R\exp(\mathrm{i}\theta)\) and \(q=\mathrm{i}x\), we see that
\begin{align*}
(t_{2}+t_{2}^{-1}-q)^{2}-4&=(R\exp(\mathrm{i}\theta)+R^{-1}\exp(-\mathrm{i}\theta)-\mathrm{i}x)^{2}-4\\
&=((R+R^{-1})\cos\theta+\mathrm{i}((R-R^{-1})\sin\theta-x))^{2}-4.
\end{align*}
Therefore, 
\begin{equation*}
\operatorname{Im}\left((t_{2}+t_{2}^{-1}-q)^{2}-4\right)=2(R+R^{-1})\cos\theta
\left((R-R^{-1})\sin\theta-x\right)
\end{equation*}
and \(\operatorname{Im}\left((t_{2}+t_{2}^{-1}-q)^{2}-4\right)\le 0\) if and only if 
\begin{equation*}
(R-R^{-1})\sin\theta-x\le 0.
\end{equation*}
Recall that \(p_{1} = R\exp (\mathrm{i}\theta_{0})\).
Since \((R-R^{-1})\sin\theta_{0}-x=0\), we deduce that 
\((R-R^{-1})\sin\theta-x<0\) for all \(\theta\in [0,\theta_{0}]\) as desired.
\end{proof}
By virtue of Lemma \ref{lemma:countour-integrals}, we see that
\begin{align}
\label{eq:partial-contour}
\begin{split}
&\operatorname{Im}\left(\int_{-\partial a(q)}\iota_{\mathrm{i}\partial/\partial q}\check{\Omega}\right)\\
&=\sum_{j=1}^{3}
\operatorname{Im}\left(
\int_{\gamma_{1,j}}\frac{1}{\sqrt{(t_{2}+t_{2}^{-1}-q)^{2}-4}}\frac{\mathrm{d}t_{2}}{t_{2}}\right)\\
&\ge\operatorname{Im}\left(\int_{\mathcal{C}}\frac{1}{\sqrt{(t_{2}+t_{2}^{-1}-q)^{2}-4}}
\frac{\mathrm{d}t_{2}}{t_{2}}\right)
\end{split}
\end{align}
where \(\mathcal{C}\) is the counterclockwise oriented
contour \(R\exp(\mathrm{i}\theta)\) with \(\theta\in[0,\pi/4]\) for all \(|q|\) large.
This can be done since \(\theta_{0}>\pi/4\) for all \(q\) with \(|q|\gg 0\).

The third inequality holds as one can see
in the above proof.
Notice that 
\begin{equation*}
\frac{\mathrm{d}t_{2}}{t_{2}}=\mathrm{i}\mathrm{d}\theta~\mbox{on}~\mathcal{C}.
\end{equation*}
Thus it suffices to compute the real part of the integrand.

Put \(q=\mathrm{i}x\). We can easily compute \(p_{1}=((q+2)+\sqrt{q(q+4)})/2\) and
\begin{equation*}
\lim_{x\to\infty} \frac{|p_{1}|}{|q|}=\lim_{x\to\infty} \frac{R}{x}=1. 
\end{equation*}
For \(t_{2}=R\exp(\mathrm{i}\theta)\), we see that 
\begin{align*}
|(t_{2}+t_{2}^{-1}-q)^{2}-4|^{2}&=|(R\exp(\mathrm{i}\theta)+R^{-1}\exp(-\mathrm{i}\theta)-q)^{2}-4|^{2}\\
&\le \left(R+R^{-1}+|q|\right)^{2}+4\\
&\le \kappa_{1}|q|^{2}
\end{align*}
for some constant \(\kappa_{1}>0\) and for all \(|q|\) large.
On the other hand,
\begin{align*}
(t_{2}+t_{2}^{-1}-q)^{2}-4&=(R\exp(\mathrm{i}\theta)+
R^{-1}\exp(-\mathrm{i}\theta)-\mathrm{i}x)^{2}-4\\
&=\left[(R+R^{-1})\cos\theta+\mathrm{i}\left((R-R^{-1})\sin\theta-x\right)\right]^{2}-4.
\end{align*}
Let \(\eta=(R+R^{-1})\cos\theta+\mathrm{i}\left((R-R^{-1})\sin\theta-x\right)\).
Here we remind the reader that \(R\), and hence \(\eta\), depends on \(x=|q|\).
Since \(\theta\in[0,\pi/4]\), it follows that \((R-R^{-1})\sin\theta-x<0\). Moreover, 
there exists a positive constant \(C>0\) such that
\begin{align}
\label{eq:eta-bound}
\frac{(R-R^{-1})\sin\theta-x}{(R+R^{-1})\cos\theta}
=\frac{R-R^{-1}}{R+R^{-1}}\tan\theta-\frac{x}{(R+R^{-1})\cos\theta}\ge -C
\end{align}
for all \(\theta\in [0,\pi/4]\) and all \(|q|\gg 0\).

Our goal is to estimate
\begin{equation*}
\operatorname{Re}\left(\frac{1}{\sqrt{\eta^{2}-4}}\right)=
\operatorname{Re}\left(\frac{\overline{\sqrt{\eta^{2}-4}}}{|\eta^{2}-4|}\right)=
\operatorname{Re}\left(\frac{\sqrt{\eta^{2}-4}}{|\eta^{2}-4|}\right).
\end{equation*}

\begin{lemma}
For any \(\theta\in [0,\pi/4]\), there exists a constant \(\kappa_{3}>0\) such that
\begin{equation*}
|\eta|\ge \kappa_{3}x,~\mbox{for all}~x\gg 0.
\end{equation*}
\end{lemma}
\begin{proof}
We can compute
\begin{align*}
|\eta|^2 &= (R+R^{-1})^{2}\cos^{2}\theta+(R-R^{-1})^{2}\sin^{2}\theta-2x(R-R^{-1})\sin\theta+x^{2}\\
&=R^{2}+R^{-2}+2\cos2\theta-2x(R-R^{-1})\sin\theta+x^{2}\\
&\ge R^{2} - 2xR\sin\theta +x^{2}\\
&\ge R^{2} - \sqrt{2}xR +x^{2}.
\end{align*}
There exists a constant \(\kappa_{3}>0\) such that
\begin{equation*}
|\eta|^2 \ge R^{2} - 2xR\sin\theta +x^{2}\ge \kappa_{3}x^{2} 
\end{equation*}
for all \(x\gg 0\).
\end{proof}
In particular, the lemma above implies
\begin{equation*}
\lim_{x\to\infty}\frac{|\sqrt{\eta^{2}-4}|}{|\eta|}
=\lim_{x\to\infty}\left|\sqrt{\frac{\eta^{2}-4}{\eta^{2}}}\right|=1.
\end{equation*}
The first equality holds since \(\eta\) lies in the fourth quadrant.
Together with \eqref{eq:eta-bound}, it then follows that there exist 
constants \(\kappa_{4}>0\) and \(\kappa_{5}>0\) such that
\begin{equation*}
\operatorname{Re}(\sqrt{\eta^{2}-4})\ge \kappa_{4}\operatorname{Re}(\eta)\ge \kappa_{5}|q|
\end{equation*}
for all \(x\gg 0\); in other words, we have
\begin{equation*}
\operatorname{Re}\left(\frac{1}{\sqrt{\eta^{2}-4}}\right)
=\operatorname{Re}\left(\frac{\sqrt{\eta^{2}-4}}{|\eta^{2}-4|}\right)\ge 
\frac{\kappa_{5}|q|}{\kappa_{1}|q|^{2}}=
\frac{\kappa_{6}}{|q|},~~\kappa_{6}:=\kappa_{5}\kappa_{1}^{-1}
\end{equation*}
for all \(x\gg 0\).
Thus, 
\begin{align*}
&\operatorname{Im}\left(
\int_{\mathcal{C}}\frac{1}{\sqrt{(t_{2}+t_{2}^{-1}-q)^{2}-4}}\frac{\mathrm{d}t_{2}}{t_{2}}\right)\\
&=\int_{\mathcal{C}}\operatorname{Re}\left(\frac{1}{\sqrt{(t_{2}+t_{2}^{-1}-q)^{2}-4}}\right)\mathrm{d}\theta\\
&\ge \int_{\mathcal{C}}\frac{\kappa_{6}}{|q|}\mathrm{d}\theta\ge \frac{\kappa}{|q|}
\end{align*}
for some constant \(\kappa>0\) and for all large \(|q|\). This proves
\eqref{eq:appendix-wts}.

\clearpage

\bibliographystyle{amsxport}
\bibliography{reference}
\end{document}